\documentclass[a4paper,12pt,draft]{amsart}
\usepackage[margin=30mm]{geometry}
\usepackage{amssymb, amsmath, amsthm, amscd}

\usepackage[T1]{fontenc}
\usepackage{eucal,mathrsfs,dsfont}%gives nicer set-names. Use \mathds instead of \mathds
\usepackage{color}%cyan,red;magenta,green;yellow,blue
\renewcommand{\leq}{\leqslant}
\renewcommand{\geq}{\geqslant}
\renewcommand{\le}{\leqslant}
\renewcommand{\ge}{\geqslant}

\newcommand{\lC}{{\underline{c}}}
\newcommand{\uC}{{\overline{C}}}
\newcommand{\la}{{\underline{\alpha}}}
\newcommand{\ua}{{\overline{\alpha}}}

\definecolor{mno}{rgb}{0.5,0.1,0.5}

\newcommand{\R}{\mathds R}
\newcommand{\Rd}{{\mathds R^d}}
\newcommand{\Pp}{\mathds P}
\newcommand{\Ee}{\mathds E}

\newcommand{\I}{\mathds 1}
\newcommand{\N}{\mathds{N}}

\newcommand{\ext}{\textup{Exit}}
\newcommand{\dhk}{\textup{DHK}}
\newcommand{\ND}{\textup{NDHK}}
\newcommand{\OD}{\textup{ODHK}}
\newcommand{\J}{\textup{J}}

\newcommand{\E}{\mathscr{E}}
\newcommand{\F}{\mathscr{F}}
\newcommand{\Nn}{\mathcal{N}}

\newtheorem{theorem}{Theorem}[section]
\newtheorem{lemma}[theorem]{Lemma}
\newtheorem{proposition}[theorem]{Proposition}
\newtheorem{corollary}[theorem]{Corollary}

\theoremstyle{definition}
\newtheorem{definition}[theorem]{Definition}
\newtheorem{example}[theorem]{Example}
\newtheorem{remark}[theorem]{Remark}

\numberwithin{equation}{section}
\begin{document}
\allowdisplaybreaks
\title[Non-local Dirichlet forms on Unbounded Open Sets] {\bfseries Intrinsic  Ultracontractivity of
Non-local Dirichlet forms on Unbounded Open Sets}
\author{Xin Chen\qquad Panki Kim \qquad Jian Wang}
\thanks{\emph{X.\ Chen:}
   Department of Mathematics, Shanghai Jiao Tong University, 200240 Shanghai, P.R. China; \& Fujian Key Laboratory of Mathematical Analysis and Applications (FJKLMAA), Fujian Normal University, 350007 Fuzhou, P.R. China. \texttt{chenxin217@sjtu.edu.cn}}
\thanks{\emph{P.\ Kim:}
 Department of Mathematics, Seoul National University,
Seoul 151-742, South Korea.
\texttt{pkim@snu.ac.kr}}
  \thanks{\emph{J.\ Wang:}
   School of Mathematics and Computer Science \& Fujian Key Laboratory of Mathematical Analysis and Applications (FJKLMAA), Fujian Normal University, 350007 Fuzhou, P.R. China. \texttt{jianwang@fjnu.edu.cn}}

\date{}

\maketitle

\begin{abstract}
In this paper we
consider a large class of symmetric Markov processes $X=(X_t)_{t\ge0}$ on $\R^d$ generated by non-local Dirichlet forms, which include
jump processes with small jumps of $\alpha$-stable-like type and with large jumps of super-exponential decay. Let $D\subset \R^d$ be an open (not necessarily bounded and connected) set, and $X^D=(X_t^D)_{t\ge0}$ be the killed process of $X$ on exiting $D$.
We obtain explicit criterion for the compactness and the intrinsic ultracontractivity of the Dirichlet Markov semigroup $(P^{D}_t)_{t\ge0}$ of $X^D$. When $D$ is a horn-shaped region, we further obtain two-sided estimates of ground state in terms of
jumping kernel of $X$ and the reference function of the horn-shaped region $D$.

\medskip

\medskip

\end{abstract}
\allowdisplaybreaks

\bigskip

\section{Introduction and Main Results}\label{section1}
Suppose that $(P_t^D)_{t\ge0}$ is  a strongly continuous Markov
semigroup in $L^2(D; dx)$ generated by a symmetric Markov process in an open subset $D$ of $\R^d$, and that $(P_t^D)_{t\ge0}$ has the transition density
$p^D(t, x, y)$ with respect to the Lebesgue measure.
If $(P_t^D)_{t\ge0}$ is a compact semigroup, then it is known that there
exists a complete orthonormal set of eigenfunctions $\{\phi^D_n\}_{n\ge1}$
such that $P_t^D\phi^D_n(x)=e^{-\lambda_n^Dt}\phi^D_n(x)$ for all $n\ge1$,
$t>0$ and $x\in D$, where $\{\lambda_n^D\}_{n\ge1}$ are eigenvalues of
$(P_t^D)_{t\ge0}$ such that $0<\lambda_1^D\le \lambda_2^D\le
\lambda^D_3\le \cdots\to \infty$ as $n\to \infty$. In the literature, the first eigenfunction
$\phi^D_1$ is
called ground state. Suppose furthermore that $\phi_1$ (from now we write $\phi_1^D$ as $\phi_1$ for simplicity) can be chosen to be bounded, continuous and strictly positive on $D$.
The semigroup $(P_t^D)_{t\ge0}$  is said to be
intrinsically ultracontractive,
if for every $t>0$ there is a constant $C_{t,D}>0$ such that
$$p^D(t,x,y)\le C_{t,D}\phi_1(x)\phi_1(y),\quad x,y\in D.$$

The notion of intrinsic ultracontractivity for symmetric semigroups was first introduced by Davies and Simon in
\cite{DS}. It has wide applications in the area of analysis and probability. Recently,
the intrinsic ultracontractivity of Markov semigroups (including Dirichlet semigroups and Feyman-Kac semigroups) has been intensively established for various L\'evy type processes, see e.g.\ \cite{CWi14, CW1, CW16, CS,CS1,G, KK,KL,KS1, K, KS, Kw}.

The aim of this paper is to study the intrinsic ultracontractivity of Dirichlet semigroup $(P_t^D)_{t \ge 0}$ for a large class of symmetric jump processes on an unbounded open set $D$. It is twofold.
Firstly, under quite general setting we obtain sufficient conditions and necessary conditions for the intrinsic ultracontractivity of $(P_t^D)_{t \ge 0}$. We emphasize that these results are illustrated to be optimal for
symmetric jump processes with small jumps of $\alpha$-stable-like type and with large jumps of
super-exponential decay on horn-shaped regions, even the knowledge of such interesting processes is still far from completeness.
Secondly, for horn-shaped regions we establish sharp two-sided estimates of ground state explicitly in terms of jumping kernel and the reference function of horn-shaped regions. This is the most sophisticated part of this paper, since usually the ground state is
very sensitive with respect to the behavior of the process and the shape of the open set.
\subsection{Basic setting}\label{subsection1-1}

Let $J(x,y)$ be a non-negative symmetric Borel measurable function on $\R^d\times
\R^d\setminus {\rm diag}$ satisfying that \begin{equation}\label{e:dfc}\sup_{x  \in \R^d} \int_{y\neq x}
\min\{1, |x-y|^2\}
 J(x,y)\,dy < \infty.\end{equation}Here diag denotes the diagonal set, i.e., diag $=\{(x,x):x\in \R^d\}$.
Consider the following non-local quadratic form $(\E,\F)$ on $L^2(\R^d;dx)$:
\begin{equation}\label{non-local}
\begin{split}
 \E(f,f)&=\iint_{x\neq y}\big(f(x)-f(y)\big)^2 J(x,y)\,dx\,dy,\\
\F&=\overline{C_c^\infty(\R^d)}^{\E_1}.
\end{split}
\end{equation}
Here, $C_c^\infty(\R^d)$ denotes the space of $C^\infty$ functions
on $\R^d$ with compact support, $\E_1(f,f):=\E(f,f)+\int_{\R^d}
f^2(x)\,dx$ and $\F$ is the closure of $C_c^\infty(\R^d)$ with
respect to the metric $\E_1(f,f)^{1/2}.$
Under \eqref{e:dfc}, it is known that $(\E,\F)$ is a regular Dirichlet form on
$L^2(\R^d;dx)$, see e.g.\ \cite[Example 1.2.4]{FOT}. Hence there exist a subset $\mathcal{N}\subset\R^d$
having zero capacity with respect to the Dirichlet form $(\E,\F)$,
and a symmetric Hunt process $X=\big((X_t)_{t\ge 0}, (\Pp^x)_{x \in
\R^d\setminus\mathcal{N}}\big)$ with state space
$\R^d\setminus\mathcal{N}$. See \cite[Chapter 7]{FOT}. Throughout this paper, we always \emph{assume that $\Nn=\emptyset$ $($i.e.,\ the associated Hunt process $X$ can start from all $x\in \R^d$$)$,
and that there exists a transition
density function $p(\cdot,\cdot,\cdot):(0,\infty)\times \R^d
 \times \R^d \rightarrow[0,\infty)$
so that
\begin{equation*}
P_tf(x)=\Ee^x f(X_t)=\int_{\R^d}p(t,x,y)f(y)\,dy, \quad f\in
L^2(\R^d;dx), x\in\R^d, t>0,
\end{equation*}
where $(P_t)_{t \ge 0}$ is the $L^2(\R^d;dx)$ semigroup associated
with $(\E,\F)$}.

Let $D\subseteq \R^d$ be an open (not necessarily bounded and connected) set. We define a subprocess
$X^D$ of $X$ as follows
\begin{equation*}
X^D_t:=\begin{cases}
X_t,\quad\text{if}\ t<\tau_D,\\
\,\partial,\,\,\quad \text{if}\ t\ge \tau_D,
\end{cases}
\end{equation*}
where $\tau_D:=\inf \{t > 0: X_t \notin D\}$ and $\partial$ denotes the cemetery point. The process $X^D:=(X^D_t)_{t\ge0}$
is called the killed process of $X$ upon exiting $D$. By the strong Markov property, it is easy to see that the process $X^D$ has a transition density (or Dirichlet heat kernel) $p^D(t,x,y)$, which enjoys the following relation with $p(t,x,y)$:
\begin{equation}\label{e:dden}
\begin{split} p^D(t,x,y)&=p(t,x,y)-\Ee^x\big[p(t-\tau_D,X_{\tau_D},y)\I_{\{t\ge \tau_D\}}\big],\quad x, y \in D;\\
p^D(t,x,y)&=0,\quad x \notin D \text{ or } y \notin D.
\end{split}
\end{equation}
Define
$$
P_t^D f(x)=\Ee^x f(X_t^D)=\int_D p^D(t,x,y)f(y)\, dy,\quad t>0, x\in
D, f\in L^2(D;dx).
$$
It is well known that $(P_t^D)_{t \ge 0}$ is a strongly
continuous contraction semigroup on $L^2(D;dx)$, which is called Dirichlet semigroup associated with the process $X^D$.
Furthermore, in this paper except Appendix, we assume that
\emph{for every $t>0$ the
function $p^D(t,\cdot,\cdot)$ is
bounded, continuous and strictly positive on $D\times D$}.
This assumption is also mild in a number of applications. See Propositions \ref{p:pDcont}  and \ref{p:pDp} in Appendix.

\subsection{Main results for horn-shaped regions}\label{sec1.2}
In the following, denote by $a \wedge b= \min \{ a, b\}$,
$a \vee b= \max \{ a, b\}$, and $a^+= a \vee 0$. For any non-negative function $f$ and $g$, $f\simeq g$ means that there is a constant $c\ge1$ such that $c^{-1}f(r)\le g(r)\le cf(r)$, and $f\asymp g$ means that there exist positive constants $c_i$ $(i=1,2,3,4)$ such that
$c_1f(c_2r)\le g(r)\le c_3f(c_4r)$.
 For any open subset $D\subset\R^d$ and $x\in D$,
$\delta_D(x)$ stands for the Euclidean distance between
$x$ and $D^c$. We call an open set $D\subset\R^d$ a domain, when it is connected.

The contribution of this paper is to obtain efficient criterion for the intrinsic ultracontractivity of Dirichlet semigroup $(P^D_t)_{t\ge0}$ generated by non-local Dirichlet forms on general unbounded open sets, and to establish two-sided estimates for the corresponding ground state (i.e.,\ the first eigenfunction).
To illustrate how powerful our approach is and to show how precise and sharp our estimates are,
here we summarize our
results on horn-shaped regions with specific reference functions for
several classes of
jump processes, including symmetric jump processes with small jumps of stable-like type and large jumps of super-exponential decay.

Let us first recall the definition of
horn-shaped region (domain). For any $x=(x_1,x_2, \ldots, x_d)\in \R^d$, let $\tilde x=(x_2, x_3, \ldots, x_d)$. Suppose that $f:(0,\infty)\to (0,\infty)$ is a bounded and continuous function
such that $\lim_{u\to\infty} f(u)=0$. Then, the open set $D_f=\{x\in \R^d: x_1>0, |\tilde x|< f(x_1)\}$ is called the horn-shaped region.
We call $f$ the reference function of $D_f$.
The study of Dirichlet semigroups for Brownian motions on horn-shaped regions has a long history.
For (both mathematical and physical) backgrounds
and motivations on such subject,
 see \cite{B, BB, BD, Be, CL, DS, LPZ} and the references therein.

Let $\Phi$ be a strictly increasing function on $(0, \infty)$ satisfying that there exist
constants  $0<\la\le \ua<2$ and $0<\lC\le \uC<\infty$ such that
\begin{align}
\label{e:Phi}
\lC\left(\frac{R}{r}\right)^{\la}\le \frac{\Phi(R)}{\Phi(r)}\le \uC \left(\frac{R}{r}\right)^{\ua}, \quad
0<r\, \le  R.
\end{align}
Obviously \eqref{e:Phi} implies that $\Phi(0):=\lim_{r \downarrow 0}\Phi(r)=0$.

Let $\chi$ be a nondecreasing function on $(0,\infty)$ with $\chi(r)\equiv \chi(0)$ for $r\in(0,1]$, and there exist constants $c_1,\,c_2,\,L_1,\,L_2>0$ and $\gamma \in[0,\infty]$ such
that
\begin{equation}\label{e:Exp}
 L_1\exp(c_1r^{\gamma})\leq \chi(r)\leq  L_2\exp(c_2r^{\gamma}),\quad r>1.
\end{equation}
We next consider the jumping kernel $J(x,y)$ in the regular Dirichlet form $(\E,\F)$ given by \eqref{non-local}
with the following expression
\begin{align}
\label{e:jgeneral}
J(x, y) =
\frac{ \kappa(x,y) }{|x-y|^{d}\Phi(|x-y|)  \chi (|x-y|)},
\end{align}
where
$\kappa:\Rd\times\Rd\to (0,\infty)$  is a
 measurable function satisfying that $\kappa(x,y)=\kappa(y,x)$ and
 there exists a constant ${L_0}\ge 1$ such that
\begin{align*}\label{a:kappa}
L_0^{-1}\leq \kappa(x,y)\leq L_0,\quad
x,y\in\Rd.
\end{align*}
According to \cite{CKK1},
the Dirichlet form $(\mathscr{E}, \mathscr{F})$  with  jumping kernel $J(x,y)$ above
generates an
 symmetric Hunt process $X=(X_t)_{t\ge0}$, which
 starts from all $x\in \R^d$ and has a transition density function $p(t,x,y)$ with respect to the Lebesgue measure, such that $p(\cdot,\cdot,\cdot): (0,\infty)\times \R^d\times\R^d \to (0,\infty)$ is bounded, continuous and strictly positive.

Let $D_f$ be a horn-shaped region with respect to the reference function $f$, and denote by $(P^{D_f}_t)_{t\ge0}$ the associated Dirichlet semigroup. Then, the Dirichlet heat kernel $p^{D_f}(t,x,y)$ exists, and $(x,y)\mapsto p^{D_f}(t,x,y)$ is bounded, continuous and strictly positive on $D_f\times D_f$ for every $t>0$.
When $(P^{D_f}_t)_{t\ge0}$ is intrinsically
ultracontractive, we denote by $\phi_1$ the corresponding ground state. To obtain two-sided estimates for $\phi_1$, we should further assume that the reference function $f$ is $C^{1,1}$, and also
impose the following additional assumptions on the coefficient function $\kappa$ and the scaling function $\Phi$, respectively.
We remark here that such assumptions are not needed for Theorems \ref{ex1-1n1}(1)$(a)$,  \ref{ex1-1n1}(2)$(a)$ and \ref{ex1-1n3}$(a)$ below, which consider the intrinsic ultracontractivity of $(P^{D_f}_t)_{t\ge0}$.
\begin{itemize}
\item[{\bf (K$_\eta$)}]
{\it There are constants $L>0$ and $\eta> \ua/2$ such that $$|\kappa(x,x+h)-\kappa(x,x)|\leq L |h|^{\eta}$$ for every $x,\,h\in \Rd$ with $|h|\leq 1$, where $\ua$ is the constant in \eqref{e:Phi}.}
\item[{\bf(SD)}]
{\it The function $\Phi\in C^1(0,\infty) $ such that $ r\mapsto -(\Phi(r)^{-1}r^{-d})'/r$ is decreasing on $(0,\infty)$.}
\end{itemize}
 Note that condition {\bf(SD)} above holds for pure jump isotropic unimodal L\'{e}vy processes including all subordinated Brownian motions,
 whose characteristic exponents satisfy the weak scaling conditions in \eqref{e:Phi}.
  See e.g.\ \cite[Remark 1.4]{GKK}.

\begin{theorem}\label{ex1-1n1}
With all notations and assumptions above,  we have the following two statements.

\begin{itemize}
\item[(1)] Suppose that $\gamma=0$ in \eqref{e:Exp} and that
\begin{equation}\label{ex1-1-1ag}
f(s)\simeq \Phi^{-1}(\log^{-\theta}s)
\end{equation} for some $\theta>0$.
Then,
\begin{itemize}
\item[ $(a)$] $(P^{D_f}_t)_{t\ge0}$ is intrinsically ultracontractive if and only if $\theta>1;$
 \item[$(b)$] if $(P^{D_f}_t)_{t\ge0}$ is intrinsically ultracontractive,
then for all $x\in D_f$ with $|x|$ large enough,
\begin{equation}\label{e:sta}   \phi_1(x)\simeq \Phi(\delta_{D_f}(x))^{1/{2}} \Phi(f(|x|))^{{1}/{2}} \frac{1}{ |x|^{d} \Phi(|x|)}.\end{equation}\end{itemize}

\item[(2)] Suppose that $\gamma \in (0,\infty]$ in \eqref{e:Exp} and that
\begin{equation}\label{ex1-1-2g}
f(s) \simeq
\Phi^{-1}(s^{-\theta})
\end{equation} for some $\theta>0$.
 Then,
 \begin{itemize}
 \item[$(a)$] $(P^{D_f}_t)_{t\ge0}$ is intrinsically ultracontractive if and only if $\theta>{\gamma} \wedge 1;$
 \item[$(b)$] if $(P^{D_f}_t)_{t\ge0}$ is
intrinsically
ultracontractive,
then for all $x\in D_f$ with $|x|$ large enough,
\begin{equation}\label{ex1-1-3}
 \phi_1(x)\asymp\Phi(\delta_{D_f}(x))^{1/{2}} \Phi(f(|x|))^{{1}/{2}}\exp\Big(-|x|^{\gamma \wedge 1}\log^{{(\gamma-1)^+}/{\gamma}}|x|\Big) .
 \end{equation}\end{itemize}
\end{itemize}

 \end{theorem}

For the class of jump processes considered in this subsection,
our method works for $D^f$ with not only specific reference functions in
 \eqref{ex1-1-1ag} and \eqref{ex1-1-2g},
but also general reference functions $f$. Here we restrict ourselves on \eqref{ex1-1-1ag} and \eqref{ex1-1-2g} to
 light up the structure of $\phi_1$.
 In the setting of Theorem \ref{ex1-1n1}, two-sided estimates of $\phi_1$
can be decomposed into two terms. Roughly speaking, for $x\in D_f$ with $|x|$ large enough, the term $\Phi(\delta_{D_f}(x))^{{1}/{2}}\Phi(f(|x|))^{{1}/{2}}$ in \eqref{e:sta} and \eqref{ex1-1-3} represent the probability (called exiting probability later) of the process $X^{D_f}$ exiting from $B(x, \delta_{D_f}(x))$,
and both of the other term in \eqref{e:sta} and \eqref{ex1-1-3} describe the
probability (called returning probability later) of the process $X^{D_f}$ from $x \in D_f$ to the origin, which is independent of $f$ and correlates with $p(1,x,0)$.
By \cite[Theorem 1.2]{CK1} and \cite[Theorem 1.2 and Theorem 1.4]{CKK1}, we know that two-sided heat kernel estimates of $p(1,x,0)$
are comparable to the jumping kernel
$J(x,0)$ for $|x|$ large enough, when $J(x,y)$ enjoys the form  \eqref{e:jgeneral} with
$\gamma \in [0,1]$.

\ \

\begin{theorem}\label{ex1-1n3}
Suppose that $\gamma \in (1,\infty]$ in \eqref{e:Exp} and  that \begin{equation}\label{ex5-1-2}
 f(s)\asymp
 \exp(-s^{\theta})
 \end{equation}
 for some constant
 $\theta>0$.
Then,
\begin{itemize}
\item[ $(a)$] the Dirichlet semigroup $(P^{D_f}_t)_{t\ge0}$ is intrinsically
ultracontractive;
\item[$(b)$] for any $x\in D_f$ with $|x|$ large enough,
\begin{equation}\label{ex5-1-5}
\phi_1(x)\asymp \Phi(\delta_{D_f}(x))^{{1}/{2}} \Phi(f(|x|))^{{1}/{2}}\exp\Big(-|x|^{1+(\gamma-1)(\theta\wedge \gamma)/\gamma}\Big).
\end{equation}\end{itemize}
In particular, when $\gamma=\infty$,
$$\phi_1(x)\asymp \Phi(\delta_{D_f}(x))^{{1}/{2}} \Phi(f(|x|))^{{1}/{2}}\exp\big(-|x|^{1+\theta}\big).$$
\end{theorem}

\medskip
The result above indicates the term that describes the returning probability of $X^{D_f}$ from $x \in D_f$ to the origin should depend on the reference function $f$, if $f$ is decay
faster than polynomials,
see e.g.\ \eqref{ex5-1-2}.
In particular, \eqref{ex5-1-5} implies that, when $1 < \gamma \le \theta$,
the returning probability dominates the exit probability since the exponential term can be absorbed into $\Phi(f(|x|))^{{1}/{2}}$;
while, when $1 <\theta < \gamma $ the returning probability
reveals some delicate interactions between the reference function $f$ and the jumping kernel $J(x,y)$.

To the best of our knowledge, both results above concerning the intrinsic ultracontractivity and two-sided estimates of ground state for general symmetric (non-L\'evy) jump processes on unbounded open sets are new. In fact,
previously the intrinsic ultracontractivity of symmetric jump processes on unbounded open sets is considered only for symmetric $\alpha$-stable L\'evy processes,
 e.g.,\ see \cite[Example 2]{Kw} for related conclusions on horn-shaped regions.
The argument of \cite{Kw} is heavily based on uniform boundary Harnack inequalities in \cite{BKKw1}.  Even though the uniform boundary Harnack inequalities for a quite general discontinuous Feller process in metric measure space
have  been proved in \cite{BKKw2, KSV17}, it is still not available  for symmetric jump processes whose jumping kernel $J(x,y)$ given by \eqref{e:jgeneral} with $\gamma\in (1,\infty)$ and it is not true when $\gamma=\infty$, see \cite[Section 6]{KS0}. Thus the approach of \cite{Kw} can not yield Theorem \ref{ex1-1n1} when $\gamma\in (1,\infty]$ and Theorem \ref{ex1-1n3}.
Moreover, our criterion could also
be applied to a class of jump processes
whose scaling orders
 depend on their position, see e.g. Example \ref{ex4-1} below.

We would like to mention that, in both theorems
above
we only present two-sided estimates of ground state $\phi_1$ for $x \in D_f$ with
$|x|$ large enough. The reason is that the horn-shaped region $D_f$ may not be a $C^{1,1}$
domain even if $f \in C^{1,1} $, because the boundary of
$ D_f$
  at corner near the point
$y_0=\big(0,f(0),0\cdots 0\big)\in
\partial D_f$
is only Lipschitz. (Note that for Lipschitz domain $D$ no explicit estimates for ground state are available even when $D$ is bounded, see e.g.\ \cite{BGR}.)
If we assume additionally that $D_f$ is a $C^{1,1}$ domain, then it is not
difficult to see that Theorem \ref{ex1-1n1} and Theorem \ref{ex1-1n3}
hold true for all $x \in D_f$.

\subsection{Further comments} We will make further comments on the setting and the approach of our paper.

\begin{itemize}
\item[(1)]{\it Brownian motions on horn-shaped regions.}
As mentioned before, the intrinsic ultracontractivity of Dirichlet semigroups for Brownian motions on horn-shaped regions has been studied by many authors.
Our assumptions on horn-shaped regions are more general than
those in previous literatures.
 For example, in \cite[Section 7, p.\ 366]{DS} the reference function $f$ is required to satisfy that $f'(r)/f(r)\to 0$ as $r\to\infty$. In \cite[Proposition 1]{Kw} the function $f$ fulfills that $f'/f$ is bounded. Clearly, the function $f$ given by \eqref{ex5-1-2} does not necessarily satisfy such assumption.
Besides, our estimates in Theorem \ref{ex1-1n1} and Theorem \ref{ex1-1n3} are
 quite precise,
 since they consist both the exit probability
term (i.e., involving the behavior of ground state near the  boundary) and the returning probability term (i.e., involving the behavior of ground state far from the  boundary).

\item[(2)]{\it Symmetric pure-jump processes on bounded open sets.} When $D$ is bounded, there are already a lot of works on the intrinsic ultracontractivity of symmetric jump processes, see \cite{CWi14} and the references therein.
Several differences and difficulties occur when $D$ is unbounded.
For instance, firstly the Dirichlet semigroup
 $(P_t^D)_{t\ge0}$ is always compact when $D$ is bounded; however, it is not true when $D$ is unbounded.
 Secondly  the (uniform) $C^{1,1}$ property
  of open set $D$ has a crucial effect on explicit estimates of ground state in
the bounded open set $D$,
see e.g.\ \cite{CKS,GKK,KKi};
however, as mentioned above, even if we assume that the  reference function of an unbounded horn-shaped region $D$ is a $C^{1,1}$ function,
$D$ only enjoys $C^{1,1}$ characteristics  locally.
\end{itemize}

We believe that our approach, to yield the intrinsic ultracontractivity of Dirichlet semigroups $(P_t^D)_{t\ge0}$ via intrinsic super Poincar\'e inequalities of non-local Dirichlet forms,
is interesting of its own. Such idea has been efficiently used to consider the corresponding problem for Feyman-Kac semigroups of general symmetric jump processes in \cite{CW1,CW16}. Also our methods could be used to study related topics for non-local Dirichlet forms on general metric measure spaces. On the other hand, in order to obtain two-sided estimates for ground state $\phi_1$ some new ideas and techniques are required. In particular, the approach via Harnack inequalities or harmonic measure for Brownian motions (see \cite{DS} and \cite{BB} respectively) and the idea by using uniform boundary Harnack inequality for symmetric $\alpha$-stable L\'evy processes are not feasible
in this paper.
Instead, we make full use of the  formula for the L\'evy system, two-sided estimates for Dirichlet heat kernels and sharp estimates for distributions of the first exit time and the return probability.

\ \

The remainder of the paper is arranged as follows. In the next section, we study the compactness of Dirichlet semigroup $(P_t^D)_{t\ge0}$.
Our criterion on the compactness of $(P_t^D)_{t\ge0}$ is quite general and works, in particular, for
symmetric jump processes with small jumps of high intensity in Example \ref{ex:com2}.
In Section \ref{sec-3} we derive (rough) lower bound estimates for ground state, which is one of key ingredients to establish sufficient conditions for the intrinsic ultracontractivity of $(P_t^D)_{t\ge0}$.
General results about sufficient conditions and necessary conditions for the intrinsic ultracontractivity of $(P_t^D)_{t\ge0}$ are presented in Section \ref{sec-4}.
In Section \ref{sec444}, we will apply the results in Section \ref{sec-4} to study the intrinsic ultracontractivity of $(P_t^D)_{t\ge0}$ associated with the Dirichlet form $(\E,\F)$ with jumping kernel given by \eqref{e:jgeneral} on
two specific unbounded regions, including horn-shaped regions and an unbounded and disconnected open set with locally $\kappa$-fat property.
Proofs of Theorems \ref{ex1-1n1}(1)$(a)$, \ref{ex1-1n1}(2)$(a)$ and \ref{ex1-1n3}$(a)$ are given in the end of subsection \ref{sec4444}.
Finally, by using the characterization of horn-shaped regions and estimates for (Dirichlet) heat kernel, we will obtain two-sided estimates of ground state corresponding to $(P_t^{D_f})_{t\ge0}$ in Section \ref{ss:tsehs}.
Proofs of Theorems \ref{ex1-1n1}(1)$(b)$, \ref{ex1-1n1}(2)$(b)$ and \ref{ex1-1n3}$(b)$ are given after the statement of Theorem \ref{ex1-1h}.

\ \

\noindent{\bf Notations}\,\, Throughout the paper, we use $c$, with or without subscripts, to denote strictly positive finite constants whose values are insignificant and may change from line to line.
We will use \lq\lq$:=$\rq\rq\,\, to denote a
definition, which is read as ``is defined to be\rq\rq.
Denote by $B(x,r)$ the ball with center at $x\in\R^d$ and radius
$r>0$. For $A, B\subset \R^d$, dist$(A,B)=\inf\{|x-y|:x\in A, y\in B\}$.  For any Borel subset $A\subseteq \R^d$,  we denote  $|A|$ the volume of $A$, $A^c$ the complementary set corresponding to $A$, and $\bar{A}$ the closure of the set $A$. For $a \in \R$,
$\lceil a \rceil$ is the smallest integer greater than or equal to $a$, and $\lfloor a\rfloor$ is the largest integer smaller than or equal to $a$.
For a measurable function $f:(0,\infty)\to (0,\infty)$, we use notations
$f^*(r)=\sup_{s\ge r} f(s)$ for all $r>0$, and  $f_*(r)=\inf_{1\le s\le r} f(s)$ for all $r \ge 1$.
For a decreasing function $g:(0,\infty)\to (0,\infty)$, we denote
$g^{-1}(r):=\inf\{s>0: g(s)\le r\}$ for any $r>0$, where $\inf \emptyset:=\infty$. Similarly, for an increasing function $g:(0,\infty)\to (0,\infty)$, we denote $g^{-1}(r):=\inf\{s>0: g(s)\ge r\}$ for $r>0$.
  For open set $D\subset \R^d$, denote by $C_c^\infty(D)$ the set of $C^\infty$ functions on $D$ with compact
  supports.

\section{Compactness of the Dirichlet semigroup}\label{s:iu}
Consider the symmetric Hunt process $X=\{X_t,t\ge0;\Pp^x, x\in \R^d\}$ as in Subsection \ref{subsection1-1}. The associated
Dirichlet form $(\E,\F)$ is given by  \eqref{non-local}, and denote by $p(t,x,y)$ the corresponding transition density function with respect to the Lebesgue measure.

For an open  (not necessarily bounded and connected) subset $D\subset \R^d$, let $X^D$ be the killed process of $X$ upon exiting $D$.
Denote by $(P_t^D)_{t\ge0}$ its associated semigroup on $L^2(D;dx)$, and by $p^D(t,x,y)$ its transition density (or Dirichlet heat kernel).
Recall that we always assume that $p^D(t,\cdot,\cdot)$ is
bounded, continuous and strictly positive on $D\times D$ for every $t>0$.

According to \cite[Theorem 4.4.3 (i)]{FOT}, the Dirichlet form $(\E^D,\F^D)$ associated with  $(P_t^D)_{t\ge0}$ is given by
\begin{equation}\label{e-d}
\begin{split}
  \E^D(f,f)&= \E(f,f)\\
 &=\iint_{D \times D}\big(f(x)-f(y)\big)^2 J(x,y)\,dx\,dy+2\int_D f^2(x)V_D(x)\,dx,\\
\F^D&=\overline{C_c^{\infty}(D)}^{\E_1^D},
\end{split}
\end{equation}
where
\begin{equation*}\label{e:com-1} V_D(x):=\int_{D^c} J(x,y)\,dy, \quad x \in D.\end{equation*}

\begin{definition}\it Let $\Theta$ be a deceasing function on $(0,\infty)$ such that
\begin{equation}\label{e:dhk0}\int_s^\infty \frac{\Theta^{-1}(r)}{r}\,dr<\infty\quad \text{for large } s >0.\end{equation}
We say that the on-diagonal upper bound estimate $\dhk_{\Theta,\le}$ holds for the
Dirichlet heat kernel
 $p^D(t,x,y)$, if
{$$p^D(t,x,y)\le \Theta(t)        \quad \text{ for all }    t>0 \text{ and  } x,y\in D.  $$}
\end{definition}
In particular, according to the relation \eqref{e:dden} between $p(t,x,y)$ and $p^D(t,x,y)$,
if $\dhk_{\Theta,\le}$ holds for $p(t,x,y)$, then $\dhk_{\Theta,\le}$ holds for $p^D(t,x,y)$ for all open subsets $D$.
We next present two examples to show that $\dhk_{\Theta,\le}$ is satisfied for a large class of symmetric jump processes.

\begin{definition}\label{de:J} \it Let $\Phi$ be a strictly increasing function on $(0, \infty)$
satisfying \eqref{e:Phi}.
We say that $\J_{\Phi,\ge}$ holds for the jumping kernel $J(x,y)$, if {there are constants $C_0, r_0>0$ such that
\begin{align}\label{e:JL}
J(x,y) \ge \frac{C_0}{|x-y|^{d} \Phi(|x-y|)},\quad 0<|x-y| \le r_0.
\end{align}}\end{definition}
Without loss of generality, whenever the  assumption $\J_{\Phi,\ge}$ is considered, we will assume that
the constant  $r_0$ in \eqref{e:JL} is $1$.

\begin{example}\label{ex:com1}
According to \cite[Proposition 3.1]{CKK1}, if $\J_{\Phi,\ge}$ is satisfied, then $\dhk_{\Theta,\le }$ holds with
$$\Theta(t)=c_0  \left(\Phi^{-1}(t)^{-d}
\vee t^{-d /2} \right)$$
for some constant $c_0>0$.
In particular, it follows that $\dhk_{\Theta,\le }$ holds for the heat kernel of
symmetric jump processes with small jumps of $\alpha$-stable-like type with $\alpha\in (0,2)$.
\end{example}

\begin{example} \label{ex:com2} Suppose that there is a constant $\beta\in(0,1]$ such that
$$J(x,y)\simeq \frac{1}{|x-y|^{d+2} \log^{1+\beta}\left(\frac{2}{|x-y|}\right)},\quad 0<|x-y|\le 1.$$
 Then, according to \cite[Theorem 1.1]{Ante}, $\dhk_{\Theta,\le}$ holds with
$$\Theta(t)=c_0 (t\wedge 1)^{-d/2}\left(\log \frac{2}{t\wedge 1}\right)^{\beta d/2}$$ for some constant $c_0>0$. This indicates that $\dhk_{\Theta,\le}$ is satisfied for the heat kernel of symmetric
jump processes with small jumps of high
intensity. \end{example}

\ \

The main contribution of this section is the following result.

\begin{theorem}\label{T:compact}
Suppose that $D$ is an open subset of $\R^d$ and  $\dhk_{\Theta,\le}$ holds for Dirichlet heat kernel $p^D(t,x,y)$.
If
\begin{equation}\label{e:com-2} \left|\{x\in D: V_D(x)<r\}\right|<\infty    \quad \text{for every } r>0,    \end{equation}
then the semigroup $(P_t^D)_{t\ge0}$ is compact.
 \end{theorem}

To prove Theorem \ref{T:compact}, we first cite \cite[Corollary 1.2]{WW08} in our setting. Let $(\E^D,\F^D)$ be the Dirichlet form associated with $(P_t^D)_{t\ge0}$, which is given by \eqref{e-d}.
Define $$
\E_0^D(f,f)= \iint_{D\times D}\big(f(x)-f(y)\big)^2 J(x,y)\,dx\,dy+\int_D f^2(x)V_D(x)\,dx,\quad f\in \F^D.$$
Since
\begin{equation*}\label{t2-1}
\E^D(f,f)=\E_0^D(f,f)+\int_D f^2(x)V_D(x)\,dx,\quad f \in \F^D,
\end{equation*} $(\E_0^D, \F^D)$ is also a regular Dirichlet form on $L^2(D;dx)$, and $(\E^D,\F^D)$ can be seen as the perturbation of $(\E_0^D, \F^D)$ with potential $V_D$.
Note that, because $V_D\ge0$, \cite[(1.3)]{WW08} holds trivially. Then, according to \cite[Corollary 1.2]{WW08}, we have the following statement.

\begin{proposition}\label{p:c1.2}
Suppose that condition \eqref{e:com-2} is satisfied
and that the following super Poincar\'{e} inequality holds
\begin{align}\label{p:sPi}
\int_D f^2(x)\,dx\le s \E_0^{D}(f,f)+
\beta_0(s) \left(\int_D |f|(x)\,dx\right)^2,\quad s>0, f\in \F^D
\end{align}
 with a decreasing function $\beta_0: (0, \infty) \to (0, \infty)$ satisfying
\begin{align}
\label{e:1.4}
 \int_{s}^\infty \frac{\beta_0^{-1}(r)}{r}\,dr < \infty \quad \text{ for large } s >0.
\end{align}
Then the semigroup $(P_t^D)_{t\ge0}$ is compact.
\end{proposition}

Next, we present the

\begin{proof}[Proof of Theorem $\ref{T:compact}$]
Let $\E_0^{D,*}(f,f)= \frac{1}{2}\E^D(f,f)$. Then, the process
associated with $(\E_0^{D,*}, \F^D)$ is a time-change of the process
associated with $(\E^{D}, \F^D)$. Let $p^{D,*}(t,x,y)$ be transition density of the process
corresponding to $(\E_0^{D,*}, \F^D)$. Then, by $\dhk_{\Theta,\le}$,
\begin{align}
\label{e:e:dhkn}
p^{D,*}(t,x,y)=
p^D(t/2,x,y)\le \Theta(t/2),
\quad t>0, x,y\in D.
\end{align}
Denote by $(P^{D,*}_t)_{t\ge 0}$ the transition semigroup of $(\E_0^{D,*}, \F^D)$. In particular, according to \cite[Theorem 4.5]{Wang00} (or see \cite[Theorem 3.3.15]{WBook}), the following super-Poincar\'{e} inequality holds
$$\int_D f^2(x)\,dx\le s \E_0^{D,*}(f,f)+ \beta_0(s)\left(\int_D |f|(x)\,dx\right)^2,\quad s>0, f\in \F^D,$$
   where
     \begin{align*}\beta_0(s)=\inf_{s\ge t, r\ge0}\frac{t\|P^{D,*}_r\|_{L^1(D;dx)\to L^\infty(D;dx)}}{r} \exp\left(\frac{r}{t}-1\right).
     \end{align*}
     Since
      \begin{align*}
  \beta_0(s) \le \|P^{D,*}_s\|_{L^1(D;dx)\to L^\infty(D;dx)} = \sup_{x,y\in D} p^{D,*}(s,x,y)\le
  \Theta(s/2)
\end{align*}
   and $\E_0^{D,*}(f,f)\le\E_0^D(f,f)$, we arrive at

$$\int_D f^2(x)\,dx\le s \E_0^{D}(f,f)+
\Theta(s/2)
 \left(\int_D |f|(x)\,dx\right)^2,\quad s>0, f\in \F^D.$$
In particular, \eqref{p:sPi}
holds with $\beta_0(s)=
\Theta(s/2),$
 and \eqref{e:1.4} also holds due to \eqref{e:dhk0}.
 Combining these
with \eqref{e:com-2},  we know that all the
assumptions of Proposition \ref{p:c1.2} are satisfied,
so $(P_t^D)_{t \ge 0}$ is compact by Proposition \ref{p:c1.2}.
\end{proof}

As a direct consequence of Theorem \ref{T:compact}, we immediately have

\begin{corollary}\label{c:com}
Suppose that $\dhk_{\Theta,\le}$
holds. Then, the semigroup $(P_t^D)_{t\ge0}$ is compact, if either $D$ has finite Lebesgue measure or
\begin{equation}\label{e:com-3}\lim_{ x\in D\textrm{ and }|x|\to\infty}V_D(x)=\infty.\end{equation}
\end{corollary}

The result of above corollary
under
assumption \eqref{e:com-3} extends \cite[Lemma 2]{Kw},
where symmetric $\alpha$-stable L\'evy processes is considered.

\ \

In the remainder of this paper, we are concerned with the case that \eqref{e:com-3} holds true.
In order to verify \eqref{e:com-3}, it is necessary to consider lower bounds of $V_D(x)$.
Thus we will study it under the assumption $\J_{\Phi,\ge}$.
In fact, in the proof below we can relax the assumption  in \eqref{e:Phi} to $0\le \la\le \ua\le 2$ and $\Phi(0+)=0$ (instead of
$0< \la\le \ua< 2$).

Recall that a Borel subset $U$ is called $(\kappa,r)$-fat at
$x\in\overline{U}$, if there exists a point $\xi_x \in U$ such that
$B(\xi_x, \kappa r)\subseteq U\cap B(x,r)$. We say that $U$ is
$\kappa$-fat, if there exists a constant $R_0>0$ such that $U$ is
$(\kappa,r)$-fat at every $x \in \overline{U}$ for each $r \in
(0,R_0]$.
\begin{proposition}\label{p3-1}
Suppose that $D$ is an open subset of $\R^d$ and that $\J_{\Phi,\ge}$ holds.
If
there is a constant $\kappa\in (0,1)$ such that for any $x\in D$, there exist a constant $R_x\in (0,1)$
and $z_x\in \partial D$ with
$|x-z_x|=\delta_D(x)$ satisfying that $D^c$ is $(\kappa,r)$-fat at $z_x\in \partial D$ for each $r\in (0,R_x]$, then
 there is a constant $c_1>0$ such that
\begin{equation}\label{p3-1-1}
V_D(x)\ge c_1\frac{(\delta_D(x)\wedge R_x )^d}{\delta_D(x)^d\Phi(\delta_D(x))}\quad \text{ for all } x\in D \text{ with } \delta_D(x) \le \frac{1}{2}.
\end{equation}
In particular, if $D^c$ is $\kappa$-fat, then
there exists a constant $c_2>0$ such that
\begin{equation}\label{p3-1-1-1}V_D(x)\ge \frac{c_2}{\Phi(\delta_D(x))}\quad  \text{ for all } x\in D \text{ with }\delta_D(x) \le \frac{1}{2} .\end{equation}
\end{proposition}
\begin{proof}
Noticing that $D^c$ is
$(\kappa,\delta_D(x)\wedge R_x )$-fat at $z_x\in \partial D$, we
can find a point $\xi_{z_x} \in D^c$ such that
$$B(\xi_{z_x},\kappa
(\delta_D(x)\wedge R_x) )\subseteq D^c \cap B(z_x,\delta_D(x)\wedge R_x ).$$
Since for $x\in D$ with $\delta_D(x)<{1}/{2}$ and $y\in B(\xi_{z_x},\kappa (\delta_D(x)\wedge R_x))\subseteq D^c \cap B(z_x,\delta_D(x)\wedge R_x )$,
   $$|y-x|\le |y-z_x|+|z_x-x|<(\delta_D(x)\wedge R_x) + \delta_D(x) \le 2\delta_D(x)\le  1,$$
we have, by $\J_{\Phi,\ge}$ (i.e., \eqref{e:JL}),
$$J(x,y)\ge c_1|x-y|^{-d}   \Phi(|x-y|)^{-1}  \ge c_2
(\delta_D(x))^{-d}
\Phi(\delta_D(x))^{-1}. $$
Therefore,
\begin{equation*}
\begin{split}
V_D(x)&=\int_{D^c} J(x,y)\,dy\ge \int_{B(\xi_{z_x},\kappa
(\delta_D(x)\wedge R_x))}J(x,y)\,dy\\
&\ge \frac{c_2}{\delta_D(x)^d  \Phi(\delta_D(x))}  \int_{B(\xi_{z_x},\kappa
(\delta_D(x)\wedge R_x) )}\,dy  \\
&\ge \frac{c_3  (\delta_D(x)\wedge R_x )^d}
{\delta_D(x)^d\Phi(\delta_D(x))   } ,
\end{split}
\end{equation*}
which proves \eqref{p3-1-1}.

 Furthermore, if $D^c$ is $\kappa$-fat, then we can find some constant $R_0\in (0,1)$ such that $R_x=R_0$ for all $x\in D$ in the argument above, and so the second assertion follows.
\end{proof}

Combining Example \ref{ex:com1} and Corollary \ref{c:com} with Proposition \ref{p3-1}, we have the following simple sufficient condition for the compactness of $(P_t^D)_{t\ge0}$.
\begin{corollary}  \label{c:kacomp}
Suppose $\J_{\Phi,\ge}$
holds. If $D^c$ is $\kappa$-fat and
\begin{equation}\label{e:do}\lim_{x\in D \text{ and } |x|\to \infty} \delta_D(x)=0,\end{equation} then the semigroup $(P_t^D)_{t\ge0}$ is compact. \end{corollary}

\section{Lower bound estimates of ground state}\label{sec-3}
In Section \ref{s:iu}, we considered the symmetric Hunt process $X=\{X_t,t\ge0;\Pp^x, x\in \R^d\}$ as in Subsection \ref{subsection1-1}.
Throughout this section, we continue considering  the symmetric Hunt process $X=\{X_t,t\ge0;\Pp^x, x\in \R^d\}$ as in Subsection \ref{subsection1-1}. Suppose that $D$ is a fixed open subset of $\R^d$, and  we assume that  the Dirichlet semigroup $(P_t^D)_{t\ge0}$ is compact.
Then, by assumption that $p^D(t,\cdot,\cdot)$ is bounded, continuous and strictly positive on $D\times D$ for every $t>0$, and the standard theory for symmetric compact semigroups, see e.g.\ \cite[Theorem VI.\ 16]{RS1} and \cite[Theorem XIII. 43]{RS2}, there
exists a complete orthonormal set of eigenfunctions $\{\phi^D_n\}$
such that $P_t^D\phi^D_n(x)=e^{-\lambda_n^Dt}\phi^D_n(x)$ for all
$t>0$ and $x\in D$, where $\{\lambda_n^D\}_{n\ge1}$ are eigenvalues of
$(P_t^D)_{t\ge0}$ such that $0<\lambda_1^D< \lambda_2^D\le
\lambda^D_3\le \cdots\to \infty$ as $n\to \infty$; moreover, the first eigenfunction $\phi^D_1$ can be chosen to be bounded, continuous and strictly positive on $D$ (see Proposition \ref{p5-1-0} in the Appendix for this fact). In the literature, $\phi^D_1$ is
called ground state.
In what follows, we write $\phi^D_1$ as $\phi_1$ for simplicity.

This section is devoted to driving lower bound estimates for $\phi_1$, which is one of key
ingredients to establish sufficient conditions for the intrinsic ultracontractivity of $(P_t^D)_{t \ge 0}$.

We begin with the following simple lemma.

\begin{lemma}\label{L:lower1}
For any  relatively compact open set $D_0\subset D$
and $t_0>0$, $$\phi_1(x)\ge   c_{t_0, D_0}  P_{t_0}^D\I_{D_0}(x),\quad x\in D,$$
where $c_{t_0, D_0}:=  e^{\lambda_1t_0}  \left(\inf_{y\in D_0} \phi_1(y)\right) >0.$
\end{lemma}
\begin{proof}
For any $t_0>0$ and $x\in D$, \begin{align*} P_{t_0}^D\I_{D_0}(x)&=\int_{D_0} p^D(t_0,x,y)\,dy\le \frac{1}{\inf_{y\in D_0}\phi_1(y)}\int_{D_0}p^D(t_0,x,y)\phi_1(y)\,dy \\
&\le  \frac{1}{\inf_{y\in D_0}\phi_1(y)}\int_{D}p^D(t_0,x,y)\phi_1(y)\,dy=  \frac{P^D_{t_0}\phi_1(x)}{\inf_{y\in D_0}\phi_1(y)}= \frac{e^{-\lambda_1t_0}\phi_1(x)}{\inf_{y\in D_0}\phi_1(y)},\end{align*}
where we have used the property that
$\inf_{y\in D_0} \phi_1(y)>0$ for any relatively compact open set $D_0\subset D$, thanks to the fact that
$\phi_1$ is continuous and strictly positive.
The desired assertion follows from the inequality above. \end{proof}

Recall that for any $x\in \R^d$, stopping time $\tau$ (with respect to the natural filtration of the precess $X$), and non-negative measurable function $f$ on $[0,\infty)\times \R^d\times \R^d$ with $f(s, z, z)=0$ for all $z\in\R^d$ and $s\ge 0$, we have the following  L\'evy system:
\begin{equation}\label{e:levy}
\Ee^x \left[\sum_{s\le \tau} f(s,X_{s-}, X_s) \right] = \Ee^x \left[ \int_0^\tau \left(\int_{\R^d} f(s,X_s, z) J(X_s,z)\, dz \right)\, ds \right],
\end{equation}
see e.g.\ \cite[Lemma 4.7]{CK} and \cite[Appendix A]{CK1}.

Let $\Psi$ be an increasing function on $(0, \infty)$ satisfying that there exist constants  $ 0<\beta_1\le \beta_2<\infty$ and $0<c_1\le c_2<\infty$ such that
\begin{align}
\label{e:Psi}
c_1\left(\frac{R}{r}\right)^{\beta_1}\le \frac{\Psi(R)}{\Psi(r)}\le c_2 \left(\frac{R}{r}\right)^{\beta_2}, \quad 0<r\, \le  R.
\end{align} For any $A\subset \R^d$,
the first exit time
from $A$ of  the process $X$ is defined by
$$\tau_A=\inf\{t>0:X_t\in A^c\}.$$

\begin{definition}\label{d:ext} \it Let $\Psi$ be an increasing function on $(0, \infty)$ such that \eqref{e:Psi} holds.
We say that
$\ext_{\Psi,\ge}$ holds on $D$,
if there exist constants $C_0, \tilde r_0 >0$ and $C_1\in (0,1)$ such that for all
$r\in (0, \tilde r_0 ]$,
\begin{equation}\label{e:exit}
\inf_{x \in D, B(x,r) \subset D}
\Pp^x\left(\tau_{B(x,r)}> C_0\Psi(r) \right)\ge C_1.
\end{equation}
For simplicity, we say $\ext_{\Psi,\ge}$ holds if $\ext_{\Psi,\ge}$ holds on $\R^d$.
\end{definition}
Obviously, for any open subset $D_1\subseteq D_2 \subseteq \R^d$,
if $\ext_{\Psi,\ge}$ holds on $D_2$, then it also holds on $D_1$.
\begin{example} \label{exp:order}
Suppose that there exist constants $0<\alpha_1\le\alpha_2<2$
and $c_1,c_2\in(0,\infty)$ such that
$$
\frac{c_1}{|x-y|^{d+\alpha_1}}\le J(x,y)\le \frac{c_2}{|x-y|^{d+\alpha_2}}, \quad
0<|x-y|\le 1.$$
Then, by \cite[Lemma 3.1]{CW16} (or \cite[Theorem 2.1]{BKK}),
$\ext_{\Psi,\ge}$ holds on $\R^d$ with $$\Psi(r)=r^{\alpha_2+(\alpha_2-\alpha_1)d/\alpha_1}.$$ \end{example}

\ \

\begin{proposition}\label{p:lower1}\,\,
Suppose that
$\ext_{\Psi,\ge}$ holds on $D$.
For any fixed $x_0\in D$, set $r_{x_0}= ({\delta_D(x_0)}/{4})
\wedge \tilde r_0$, where $\tilde r_0>0$ is the constant in \eqref{e:exit}. Then, there exists a constant $c:=c(x_0,r_{x_0},D)>0$ such that for all $x\in D$,
\begin{equation}\label{e:lower1-1}\phi_1(x)\ge c \Psi( \delta_D(x)\wedge r_{x_0} )\left( \inf_{|y-z|\le |x-x_0|+2r_{x_0}} J(y,z)\right).\end{equation}  \end{proposition}

\begin{proof} Fix $x_0\in D$.
For any $x\in D$ such that $|x-x_0|\ge 3r_{x_0}$ or $\delta_D(x)\le r_{x_0}$,
set $D_0=B(x_0,2 r_{x_0})$, $\tilde D_0= B(x_0,r_{x_0})$, $D_1=B(x,\delta_D(x)\wedge r_{x_0})$  and
$t_0=C_0 \Psi (
r_{x_0}),$ where
$C_0>0$ is the constant in \eqref{e:exit}.
Then, it holds that dist$(D_1,\tilde D_0)>0$. Indeed, for any $y\in \tilde D_0$ and $z\in D_1$, if $|x-x_0|\ge3r_{x_0}$, then
\begin{equation*}
|y-z|\ge |x_0-x|-|x-z|-|y-x_0|\ge 3r_{x_0}-r_{x_0}-r_{x_0}=r_{x_0}>0;
\end{equation*}
if $\delta_D(x)\le r_{x_0}$, then,
\begin{equation*}
\begin{split}
|y-z|&\ge |x_0-x|-|x-z|-|y-x_0|\\
&\ge \delta_D(x_0)-\delta_D(x)-\delta_D(x)-r_{x_0}\\
&\ge \delta_D(x_0)-3r_{x_0}\ge \delta_D(x_0)/4>0.
\end{split}
\end{equation*}

Hence,
by strong Markov property,
\begin{equation}\label{e3-3}\begin{split}
&P^{D}_{t_0}\I_{D_0}(x)\\
&\ge \Pp^x\Big(0<\tau_{D_1}\le \frac{C_0}{2}\Psi( \delta_D(x)\wedge r_{x_0} ),
X_{\tau_{D_1}}\in \tilde D_0; \forall s\in[\tau_{D_1}, t_0], X_s\in
D_0\Big)\\
&\ge \Pp^x\Bigg(\Pp^{X_{\tau_{D_1}}}\Big(\tau_{D_0}>t_0\Big); 0<\tau_{D_1}\le \frac{C_0}{2}\Psi( \delta_D(x)\wedge r_{x_0} ),
X_{\tau_{D_1}}\in \tilde D_0\Bigg)\\
&\ge  \Pp^x\Bigg(\Pp^{X_{\tau_{D_1}}}\Big(\tau_{B(X_{\tau_{D_1}},r_{x_0})}>t_0\Big); 0<\tau_{D_1}\le \frac{C_0}{2}\Psi( \delta_D(x)\wedge r_{x_0} ),
X_{\tau_{D_1}}\in \tilde D_0\Bigg)\\
&\ge C_1\Pp^x\Big(0<\tau_{D_1}\le \frac{C_0}{2}\Psi( \delta_D(x)\wedge r_{x_0} ), X_{\tau_{D_1}}\in \tilde D_0\Big),\end{split}
\end{equation}
where in the third and fourth inequalities we have used the fact that for all $y \in \tilde D_0$,
$$\Pp^y\left(\tau_{B(X_{\tau_{D_1}},r_{x_0})}>t_0\right)\ge
\Pp^y(\tau_{B(y, r_{x_0})}>t_0)\ge
C_1>0,$$
thanks to $\ext_{\Psi,\ge}$.

Furthermore, by the L\'{e}vy system in \eqref{e:levy},
\begin{align*}&\Pp^x\Big(0<\tau_{D_1}\le \frac{C_0}{2}\Psi( \delta_D(x)\wedge r_{x_0} ), X_{\tau_{D_1}}\in \tilde D_0\Big)\\
&=\Ee^x  \int_0^ {\frac{C_0}{2}\Psi( \delta_D(x)\wedge r_{x_0} )}\,
\int_{\tilde D_0} J(X_s^{D_1},z)\,dz\,ds\\
&\ge  \left(\inf_{y\in D_1, z\in \tilde D_0} J(y,z) \right)|\tilde D_0| \int_0^ {\frac{C_0}{2}\Psi( \delta_D(x)\wedge r_{x_0} )} \Pp^x(\tau_{D_1}>s)\,ds\\
&\ge  \frac{C_0}{2}\Psi( \delta_D(x)\wedge r_{x_0} )  \left(\inf_{y\in D_1, z\in \tilde D_0} J(y,z) \right)|\tilde D_0| \Pp^x\Big( \tau_{D_1}>\frac{C_0}{2}\Psi( \delta_D(x)\wedge r_{x_0} )\Big)\\
&\ge c_1 \Psi( \delta_D(x)\wedge r_{x_0} ) \left(\inf_{y\in D_1, z\in \tilde D_0} J(y,z)\right), \end{align*}
where in the equality above we have used the fact that dist$(D_1,\tilde D_0)>0$, and in the last inequality we
used $\ext_{\Psi,\ge}$ again.

Combining both estimates above, we arrive at
\begin{align*}P^{D}_{t_0}\I_{D_0}(x)
\ge & c_2 \Psi( \delta_D(x)\wedge r_{x_0} )\left(
\inf_{y\in D_1, z\in \tilde D_0} J(y,z)\right)\\
\ge &c_2 \Psi( \delta_D(x)\wedge r_{x_0} )\left( \inf_{|y-z|\le |x-x_0|+(r_{x_0}+\delta_D(x)\wedge r_{x_0})} J(y,z)\right). \end{align*} This
along with Lemma \ref{L:lower1} yields the desired assertion
for any $x \in D$ with $|x-x_0|\ge 3r_{x_0}$ or $\delta_D(x)\le r_{x_0}$.

Next, we set $D_{x_0}:=\{x \in D: |x-x_0|< 3r_{x_0} \,\,{\rm and }\,\, \delta_D(x)>r_{x_0}\}$. Then, $D_{x_0}$ is a precompact open subset of $D$. Since by assumption $\phi_1$ is continuous and strictly positive on $D$, we have
$\inf_{z \in D_{x_0}}\phi_1(z)\ge c_3$ for some constant $c_3>0$.
While, since by \eqref{e:dfc}
\begin{align*}
 \infty >c_4&:=\sup_{y  \in \R^d} \int_{z\neq y}
(1 \wedge |y-z|^2)
 J(y,z)\,dz\\
 & \ge
(1 \wedge  r_{x_0}^2) \sup_{y  \in \R^d}
 \int_{\{r_{x_0} \le  |y-z|\le 2r_{x_0}\}} J(y,z)\,dz \\
 &\ge (1 \wedge  r_{x_0}^2) |\{r_{x_0} \le  |w|\le 2r_{x_0}  \} |\sup_{y  \in \R^d} \inf_{z\in \R^d: r_{x_0} \le |y-z|\le 2r_{x_0}}  J(y,z),\end{align*}
we have
\begin{align*}
\inf_{|y-z|\le |x-x_0|+2r_{x_0}} J(y,z) &\le
\sup_{y  \in \R^d}  \inf_{z\in\R^d : r_{x_0} \le |y-z|\le 2r_{x_0}} J(y,z) \\
&\le \frac{c_4}{ (1 \wedge  r_{x_0}^2) |\{r_{x_0} \le  |w|\le 2r_{x_0}  \} |}.
\end{align*}
Hence, by changing the constant $c$ properly, we know that
\eqref{e:lower1-1} also holds on $D_{x_0}$.

Therefore, the desired assertion follows from both estimates above.
\end{proof}

Note that the right hand side of \eqref{e:lower1-1} is zero if the
process $X$ has finite range jumps and $|x-x_0|$ is large enough. Thus, Proposition \ref{p:lower1} mainly concerns with lower
bounds of $\phi_1$ for processes with infinite range jumps. We next consider
another type of lower bounds of $\phi_1$, which is suitable for processes with finite range
jumps or with sup-exponentially  decaying jumps, e.g.,\ the jumping kernel given by \eqref{e:jgeneral} with $\gamma\in(1,\infty]$.

\begin{definition}\label{d:c-reason} {\it
For fixed $y\in D$, we call $x\in D$ is connected
with $y$ in a reasonable way with respect to constants $0<a_1\le a_2<1$, if there exist $n:=n(y,x) \in \N$ and $\{x^{(i)}\}_{i=0}^n$
such that $x^{(0)}=y$, $x^{(n)}=x$, $x^{(i)}\in D$ and $a_1\le
|x^{(i)}-x^{(i-1)}|\le a_2 $ for $1\le i\le n$.
For simplicity, we write $x\sim_{(n; a_1, a_2)} y$ if $x$ is connected with $y$ in a reasonable way with respect to constants $0<a_1\le a_2<1$.}
\end{definition}

\begin{proposition}\label{p:lower2}
Suppose that $\ext_{\Psi,\ge}$ holds on $D$ and that
\begin{equation}\label{e:Jlow}\inf_{0<|x-y|<1}J(x,y)>0.\end{equation}
For any fixed $x_0\in D$,
there exist constants $c_1, c_2>0$ $($which may depend on $x_0, a_1,a_2, \tilde r_0, C_0,C_1$ but are independent of $n$$)$  such that for all $x\in D$ with $x\sim_{(n; a_1, a_2)}  x_0$ it holds that
\begin{equation}\label{e:lower2-1}\begin{split} \phi_1(x)\ge&c_1\Psi( \delta_D(x)\wedge r_0) \exp\left[-c_2\Big(n\log n +\sum_{i=0}^{n-1} \log \frac{1}{r_{i}} \Big)\right],\end{split}\end{equation}
 where $r_i
:=(\delta_D(x^{(i)})/3)\wedge r_0$ for all $1\le i\le n$, $r_0:=(\delta_D(x_0)/3)\wedge  (( (1-a_2)\wedge a_1)/4)\wedge \tilde r_0$,
and $\tilde r_0$ is the constant in \eqref{e:exit}.
\end{proposition}
\begin{proof}
In the following, we fix $x_0,x\in D$ such that $x\sim_{(n; a_1, a_2)} x_0$.
 For $0\le i\le n$, set
$D_i=B(x^{(i)}, 2r_i)$ and $\tilde D_i=B(x^{(i)},r_i).$ Then, for any $z\in D_{i}$ and $u\in \tilde D_{i-1}$ with $1\le i\le n$,
$$|z-u|\le |x^{(i)}-x^{(i-1)}|+|x^{(i)}-z|+|x^{(i-1)}-u|\le a_2+\frac{(1-a_2)}{2}+\frac{(1-a_2)}{4}< 1,$$ and $$|z-u|\ge  |x^{(i)}-x^{(i-1)}|-|x^{(i)}-z|-|x^{(i-1)}-u|> a_1-\frac{a_1}{2}-\frac{a_1}{4}=\frac{a_1}{4}>0.$$
In particular, $ D_{i} \cap \tilde D_{i-1}=\emptyset$ for $1\le i\le n$.

 Define
$\tilde{\tau}_{D_i}=\inf\{t\ge \tilde{\tau}_{D_{i+1}}: X_t\notin D_i\}$ for $1\le i\le n-1$, and
$\tilde{\tau}_{D_n}={\tau}_{D_n}.$
By the convention, we also set $\tilde{\tau}_{D_{n+1}}=0$ (and so $X_{\tilde{\tau}_{D_{n+1}}}=x$ under $\Pp^x$). Let
$t_0= C_0\Psi(r_0)$,
where $C_0>0$ is the constant in \eqref{e:exit}. Then,
\begin{equation}\label{e:piuf}
\begin{split}
&P^D_{t_0}\I_{D_0}(x)
=\Ee^x\Big(\I_{D_0}(X_{t_0})\I_{\{\tau_D>t_0\}}\Big)\\
& \ge
\Pp^x\Big(0<\tilde{\tau}_{D_{i}}-\tilde{\tau}_{D_{i+1}}\le \frac{C_0\Psi(r_i)}{2n}, X_{\tilde{\tau}_{D_i}}\in \tilde D_{i-1}\
{\rm for\ each}\ 1\le i \le n;\\
&\qquad\quad  X_{s} \in D_0 \,  \text{ for all }{s\in[\tilde{\tau}_{D_1},t_0] }\Big)\\
& = \Pp^x\Big(0<{\tau}_{D_{n}}\le \frac{C_0\Psi(r_n)}{2n}, X_{{\tau}_{D_{n}}}\in
\tilde D_{n-1}\\
&\quad\quad\times\Pp^{X_{\tilde{\tau}_{D_{n}}}}\Big(0<\tau_{D_{n-1}}\le \frac{C_0\Psi(r_{n-1})}{2n},X_{{\tau}_{D_{n-1}}}\in
\tilde D_{n-2}\\
&\quad\quad\,\times \Pp^{X_{\tilde{\tau}_{D_{n-1}}}}\Big(\cdots
\Pp^{X_{\tilde{\tau}_{D_{2}}}}\Big(0<{\tau}_{D_{1}}\le \frac{C_0\Psi(r_1)}{2n},X_{\tau_{D_{1}}}\in
\tilde D_0\\
&\quad\quad\,\,\times
\Pp^{X_{\tilde{\tau}_{D_{1}}}}\Big( X_{s}
\in D_0 \text{ for all } s\in[0,t_0-\tilde {\tau}_{D_1}] \Big)\Big)\cdots\Big)\Big)\Big),
\end{split}
\end{equation}
where in the last equality we have used the strong Markov property.

According to the L\'{e}vy system in \eqref{e:levy}, for any
$1\le i\le n$, if $X_{\tilde{\tau}_{D_{i+1}}} \in \tilde D_{i}$, then,
 \begin{equation}\label{e:piuf1}\begin{split}
&
\Pp^{X_{\tilde{\tau}_{D_{i+1}}}}\Big(0<{\tau}_{D_{i}}\le \frac{C_0\Psi(r_{i})}{2n},X_{{\tau}_{D_{i}}}\in
\tilde D_{i-1}\Big)\\
&\ge \inf_{y\in \tilde D_{i}} \Pp^y\Big(0<{\tau}_{D_{i}}\le \frac{C_0\Psi(r_{i})}{2n},X_{{\tau}_{D_{i}}}\in
\tilde D_{i-1}\Big)\\
&= \inf_{y\in \tilde D_{i}}
\Ee^y \int_0^{\frac{C_0\Psi(r_{i})}{2n}}\int_{\tilde D_{i-1}}J(X^{D_{i}}_s,u)\,du\,ds\\
&\ge \left( \inf_{z\in D_{i}, u\in \tilde D_{i-1}} J(z,u)\right) |\tilde D_{i-1}| \inf_{y\in \tilde D_{i}} \int_0^{\frac{C_0\Psi(r_{i})}{2n}}\Pp^y(\tau_{D_{i}}>s)\,ds\\
&\ge  \left( \inf_{z\in D_{i}, u\in \tilde D_{i-1}} J(z,u)\right) |\tilde D_{i-1}| {\frac{C_0\Psi(r_{i})}{2n}} \inf_{y\in \tilde D_{i}}\Pp^y\Big(\tau_{D_{i}}> \frac{{C_0\Psi(r_{i})}}{2}\Big)\\
&\ge  \left( \inf_{z\in D_{i}, u\in \tilde D_{i-1}} J(z,u)\right) |\tilde D_{i-1}| {\frac{C_0\Psi(r_{i})}{2n}}  \inf_{y\in \tilde D_{i}}\Pp^y\Big(\tau_{B(y, r_{i})}> C_0\Psi(r_{i})\Big)\\
&\ge c_1  {\frac{\Psi(r_{i})}{n}} r_{i-1}^d,
 \end{split}\end{equation}
where in the equality above we used the fact that $ D_{i} \cap \tilde D_{i-1}=\emptyset$, and in the last
inequality we used $\ext_{\Psi,\ge}$ and \eqref{e:Jlow}.

On the other hand, due to $\ext_{\Psi,\ge}$ again, if
$X_{\tilde{\tau}_{D_1}}\in \tilde D_0$, then
\begin{equation}\label{e:piuf3}
\begin{split}
\Pp^{X_{\tilde{\tau}_{D_{1}}}}\Big(\forall_{s\in  [0,t_0-\tilde {\tau}_{D_1}]} X_{s}
\in D_0\Big)&\ge\inf_{y \in \tilde
D_0}\Pp^y\Big (\tau_{D_0}>t_0\Big)\\
&\ge \inf_{y \in \tilde
D_0}\Pp^y\Big
(\tau_{B(y,r_0)}>t_0\Big)
\ge C_1>0.
\end{split}
\end{equation}

Therefore,  combining  \eqref{e:piuf} with \eqref{e:piuf1} and \eqref{e:piuf3}
and using \eqref{e:Psi},
we find that
\begin{equation*}\label{e:piuf4}\begin{split}
P^D_{t_0}\I_{D_0}(x)&\ge
  (1-C_1)
  c_1^n \prod_{i=1}^{n} \left({\frac{\Psi(r_{i})}{n}} r_{i-1}^d\right)\\
  & \ge c_2 \frac{\Psi( \delta_D(x)\wedge r_0 )}{n^n}
  c_1^n \prod_{i=1}^{n-1} {\Psi(r_{i})} r_{i}^d
  \\
 & = c_2 \Psi( \delta_D(x)\wedge r_0 )\\
   & \qquad \times  \exp \left[-\left(n (\log n-\log c_1)  +\sum_{i=1}^{n-1}
  \log \frac{1}{\Psi(r_{i})}
+ d\sum_{i=1}^{n-1}\log \frac{1}{r_{i}} \right)
 \right] \\
  &  \ge c_3 \Psi( \delta_D(x)\wedge r_0 )   \exp \left[-c_4\left(n \log n  +\sum_{i=0}^{n-1}
 \log \frac{1}{r_{i}} \right)
 \right],\end{split}
\end{equation*}
where in the last inequality we have used  the property that
$$\log \frac{1}{\Psi(r_i)}\le \beta_2 \log \frac{1}{r_i}+c_5,$$
due to \eqref{e:Psi}.
This along with Lemma \ref{L:lower1}
 proves the desired assertion.
\end{proof}

\section{Intrinsic ultracontractivity of Dirichlet semigroups: general results}\label{sec-4}
The main purpose of this section is to present sufficient conditions and necessary conditions for intrinsic ultracontractivity of $(P_t^D)_{t\ge0}$.

\subsection{Sufficient conditions for intrinsic ultracontractivity of $(P_t^D)_{t\ge 0}$}
In this part, we consider the Dirichlet form $(\E,\F)$ given by \eqref{non-local} such that
$\J_{\Phi,\ge}$ is satisfied. Recall that, without loss of generality, we have assumed that $r_0=1$ in \eqref{e:JL}, i.e.,\
$$J(x,y) \ge \frac{C_0}{|x-y|^{d} \Phi(|x-y|)},\quad 0<|x-y| \le 1.$$
Let $D\subset \R^d$ be an open set such that \eqref{e:com-3} holds; that is,
$$\lim_{ x\in D\textrm{ and }|x|\to\infty}V_D(x)=\infty.$$
Then, according to Example \ref{ex:com1} and Corollary \ref{c:com}, we know that $(P_t^D)_{t\ge0}$ is compact.
As stated before, we always assume that
the Dirichlet heat kernel $p^D(t,x,y)$ exists, and $p^D(t,\cdot,\cdot)$ is bounded, continuous and strictly positive on $D\times D$ for every $t>0$. Let $\phi_1$ be the associated ground state, which can be chosen to be bounded, continuous and strictly positive, see e.g.\ Proposition \ref{p5-1-0} below.

Recall that the Dirichlet form $(\E^D, \F^D)$ is given in \eqref{e-d}.
Our approach to prove the intrinsic ultracontractivity  of $(P_t^D)_{t\ge0}$ is based on the intrinsic super Poincar\'{e} inequality for $(\E^D, \F^D)$.
For any open set $D$ and constants $R,r>0$, let $$D_{R, r}=\big\{x\in D: |x|<R\textrm{ and }
\delta_D(x)>r\big\}.$$
We first present the following form of  local intrinsic super
Poincar\'{e} inequality for $(\E^D, \F^D)$.

\begin{proposition}\label{p:local}
Suppose that $\J_{\Phi,\ge}$ holds  and $D$ satisfies \eqref{e:com-3}.
Then
there exists a constant $c>0$
such that
for any $R,r,s>0$,
and $f\in C_c^\infty(D)$,
\begin{equation}\label{e:local} \int_{D_{R,r}}f^2(x)\,dx\le s\E^D(f,f)+\alpha(R,r,s;\phi_1) \left(\int_D|f(x)|\phi_1(x)\,dx\right)^2,  \end{equation}
where
$$\alpha(R,r,s;\phi_1):= \frac{c}{\inf_{x\in D_{R+1, r/2} }\phi_1^2(x)} \big(\Phi^{-1}(s)\wedge r\wedge 1\big)^{-d}.$$
\end{proposition}

\begin{proof}
For  $f\in C_c^\infty(D)$, define $$f_s(x)=\frac{1}{|B(x,s)|}\int_{B(x,s)} f(z)\,dz, \quad s>0, x \in D.$$
Note that for any $x\in D_{R,r}$ and $0<s<r/2$, by the definition of $D_{R,r}$, $B(x,s)\subset D_{R+s, r-s}$.   Thus,  we obtain
$$\sup_{x\in D_{R,r}}|f_s(x)|\le \frac{1}{|B(0,s)|}\int_{D_{R+s,r-s}}|f(z)|\,dz$$ and
\begin{align*}\int_{D_{R,r}}|f_s(x)|\,dx&\le \int_{D_{R,r}} \frac{1}{|B(0,s)|}\int_{B(x,s)} |f(z)|\,dz\,dx\\
&\le \int_{D_{R+s,r-s}}\left(\frac{1}{|B(0,s)|}\int_{B(z,s)}\,dx\right)|f(z)|\,dz\\
&\le \int_{D_{R+s,r-s}} |f(z)|\,dz.\end{align*}
Hence, \begin{align*}\int_{D_{R,r}}f_s^2(x)\,dx\le &\left(\sup_{x\in D_{R,r}}|f_s(x)|\right)\int_{D_{R,r}}|f_s(z)|\,dz
= \frac{1}{|B(0,s)|}\left(\int_{D_{R+s,r-s}} |f(z)|\,dz\right)^2.\end{align*} Therefore, we have for any $f\in C_c^\infty(D)$ and $0<s\le (r/2)\wedge1$,
\begin{align*}\int_{D_{R,r}}f^2(x)\,dx&\le 2\int_{D_{R,r}}(f(x)-f_s(x))^2\,dx+2\int_{D_{R,r}}f^2_s(x)\,dx\\
&\le 2 \int_{D_{R,r}} \frac{1}{|B(0,s)|}\int_{B(x,s)}(f(x)-f(y))^2\,dy\,dx\\
&\quad+ \frac{2}{|B(0,s)|}\left(\int_{D_{R+s,r-s}}|f(z)|\,dz\right)^2\\
&\le c_1\Phi(s)\iint_{\{(x,y)\in D\times D: |x-y|\le s\}}(f(x)-f(y))^2 J(x,y)\,dx\,dy\\
&\quad+ c_2 s^{-d}\left(\int_{D_{R+s,r-s}}|f(z)|\,dz\right)^2\\
&\le c_1\Phi(s)\E^D(f,f)+\frac{ c_2s^{-d}}{\inf_{x\in D_{R+1, r/2} }\phi^2_1(x)} \left(\int_{D_{R+1,r/2}}|f(z)|\phi_1(z)\,dz\right)^2,\end{align*}
where in the third inequality we have used \eqref{e:JL} and the fact that for any $x\in D_{R,r}$ and $0<s\le r/2$, $B(x,s)\subset D$.

Setting $c_1\Phi(s)$ as $s$ in the inequality above and using \eqref{e:Phi}, we arrive
at that there exist constants $c_3, c_4>0$ such that for any $0<s \le c_3(\Phi(r) \wedge1)$ and $f\in
C_c^\infty(D)$,
\begin{align}
\label{e:localP1}
\int_{D_{R,r}}f^2(x)\,dx\le s\E^D(f,f)+\frac{ c_4(\Phi^{-1}(s))^{-d}}{\inf_{x\in D_{R+1, r/2} }\phi_1^2(x)} \left(\int_D|f(x)|\phi_1(x)\,dx\right)^2.
\end{align}
When $c_3(\Phi(r) \wedge1)<s$, we apply \eqref{e:localP1} with $s=c_3(\Phi(r) \wedge1)$ and get
\begin{align*}
\int_{D_{R,r}}f^2(x)\,dx
&\le c_3(\Phi(r) \wedge1)\E^D(f,f)\\
&\quad +\frac{ c_4(\Phi^{-1}(c_3(\Phi(r) \wedge1)))^{-d}}{\inf_{x\in D_{R+1, r/2} }\phi_1^2(x)} \left(\int_D|f(x)|\phi_1(x)\,dx\right)^2\\
&\le s\E^D(f,f)+\frac{ c_5 (\Phi^{-1}(s) \wedge r\wedge 1)^{-d}}{\inf_{x\in D_{R+1, r/2} }\phi_1^2(x)} \left(\int_D|f(x)|\phi_1(x)\,dx\right)^2,
\end{align*} where in the last inequality we used \eqref{e:Phi} again.
Therefore, the desired assertion follows from both conclusions above.
\end{proof}

The following theorem is the main result in this subsection.

\begin{theorem}\label{T:IU}
Suppose that $\J_{\Phi,\ge}$ holds  and $D$ satisfies \eqref{e:com-3}.
Then, for any $f\in C_c^\infty(D)$ and $s>0$,
\begin{equation}\label{e:whole}\int_Df^2(x)\,dx\le s\E^D(f,f)+ \beta(s)\left(\int_D|f(x)|\phi_1(x)\,dx\right)^2,  \end{equation} where
\begin{equation}\label{beta}
\beta(s):=\inf\left\{ \alpha(R,r,s/2;\phi_1): \, 0<r <1/2\textrm{ and  }
\inf_{x\in
D\setminus D_{R,r}}V_D(x)\ge 1/s
 \right\},
\end{equation}
and $\alpha(R,r,s;\phi_1)$ is defined in Proposition $\ref{p:local}$.
By convention, $\inf \emptyset =\infty$ here.

Consequently, the semigroup $(P^D_t)_{t\ge 0}$ is intrinsically
ultracontractive, if
\begin{equation}\label{e:iurate}
\int_t^{\infty}\frac{\beta^{-1}(s)}{s}\,ds<\infty,\quad t>\inf
\beta.
\end{equation}
\end{theorem}

\begin{proof}
Due to \eqref{e-d}, it holds that $$\int_D f^2(x)V_D(x)\,dx\le \frac{1}{2}\E^D(f,f),\quad f \in C_c^\infty(D).$$
Then, for any $R$, $r>0$,
\begin{align*}
\int_{D\setminus D_{R,r}} f^2(x)\,dx&\le \frac{1}{\inf_{x\in D\setminus D_{R,r}}V_D(x)}\int_D f^2(x)V_D(x)\,dx\\
&\le  \frac{1}{2\inf_{x\in D\setminus D_{R,r}}V_D(x)} \E^D(f,f).\end{align*}  This along with \eqref{e:local} yields that for any $t>0$,
\begin{align*}
\int_Df^2(x)\,dx
\le & \left(t+ \frac{1}{2\inf_{x\in D\setminus D_{R,r}}V_D(x)}\right)\E^D(f,f)\\
&\quad + \alpha(R,r,t;\phi_1) \left(\int_D|f(x)|\phi_1(x)\,dx\right)^2.\end{align*}
 For any $s>0$, letting $t= s/2$ and choosing $0<r<1/2$ and $R>0$ such that $ \inf_{x\in D\setminus D_{R,r}}V_D(x)\ge 1/s$,
 we obtain the first desired assertion.

The second assertion is a consequence of the first one and \cite[Theorem 3.3]{Wang02} or
\cite[Theorem 3.3.14]{WBook}.
See e.g.\ \cite[Theorem 2.1]{CW16} for the proof.
 \end{proof}

\begin{remark}\label{r:new} {
Since  $\inf_{x\in D\setminus D_{R,r}}V_D(x) \le \inf_{x\in D \text{ and } |x|\ge R} V_D(x)$ for any $r>0$, \eqref{e:com-3} does not
guarantee  the finiteness of the rate function $\beta(s)$.
 According to Proposition \ref{p3-1}, we know that if $D^c$ is $\kappa$-fat such that \eqref{e:do} is satisfied, then  the rate function $\beta(s)$ defined by \eqref{beta} is finite
} \end{remark}

\subsection{Necessary conditions for intrinsic ultracontractivity of $(P_t^D)_{t \ge 0}$}
We first introduce two definitions.
\begin{definition}\label{d:ND}
\it
 Let $\Psi$ be an increasing function on $(0, \infty)$ satisfying \eqref{e:Psi}.
  We say that the lower bound near diagonal estimate
   $\ND_{\Psi,\ge}$ holds $D$,
  if there are $c_0\in (0,1)$ and $r_0,c_1,c_2>0$ such that for all
 $0<r\le r_0$ and    $B(x,r) \subset D$,
$$ p^{B(x,r)}(c_1\Psi(r), y,z)\ge c_2r^{-d},\quad y,z\in B(x,c_0r).$$\end{definition}

\begin{definition} \it Let $\Gamma$ be a non-increasing positive function on $(1,\infty)$ such that $\lim_{s \rightarrow \infty}
\Gamma(s)=0$. We say that the off-diagonal upper bounded estimate $\OD_{\Gamma,\le}$ holds for
the Dirichlet heat kernel $p^D(t,x,y)$
if there are $t_0, c_0 > 0$ such that for all $t\in (0,t_0]$ and $x,y\in D$ with $|x-y|\ge c_0$,
$$p^D(t,x,y)\le C(t)\Gamma(|x-y|),$$
where $C(t)$ is a positive constant depending on $t$.
\end{definition}
If
$\ND_{\Psi,\ge}$ holds on $D$,
 then there exists a positive constant $c>0$ such that for all $r\in (0,r_0]$ and $B(x,r) \subset D$,
$$ \Pp^x\big(\tau_{B(x,r)}
> c_1\Psi(r)\big)\ge \int_{B(x,c_0r)}p^{B(x,r)}
\big( c_1\Psi(r),x,y\big)\,dy\ge c,$$ and so
$\ext_{\Psi,\ge}$ holds on $D$ too.
Next, we give the other consequence of $\ND_{\Psi,\ge}$.
\begin{lemma}\label{l:sur_low}
Suppose that $\ND_{\Psi,\ge}$ holds on $D$.
Then, there exist positive constants $c_1,c_2$ such that for all $t>0$ and $x\in D$ with $\delta_D(x)\le r_0$,
$$
P_{t}^D\I_{B(x,\delta_{D}(x))}(x)\ge  c_1 e^{-c_2 t/\Psi(\delta_{D}(x))},
$$
where $r_0$ the positive constant in Definition $\ref{d:ND}$.
\end{lemma}
\begin{proof} Throughout the proof, let $r_0$ and $c_1$ be positive constants in the definition $\ND_{\Psi,\ge}$. Fix $x\in D$ with  $\delta_D(x)\le r_0$.
As mentioned above, $\ND_{\Psi,\ge}$ implies $\ext_{\Psi,\ge}$. Hence,  if $t>0$ such that $c_1\Psi(\delta_D(x))\ge t$, then
\begin{equation*}
\begin{split}
P_{t}^D\I_{B(x,\delta_{D}(x))}(x)&\ge \Pp^x(\tau_{B(x,\delta_{D}(x))}>t)\\
&\ge  \Pp^x\big(\tau_{B(x,\delta_{D}(x))}>c_1\Psi(\delta_{D}(x))\big)\\
&\ge C_1\ge C_1 e^{- t/(c_1\Psi(\delta_{D}(x)))}.
\end{split}
\end{equation*}

Next, we consider $t>0$
such that $c_1\Psi(\delta_D(x))<t$. Set $n=\big\lceil \frac{t}{c_1\Psi(\delta_{D}(x))}\big \rceil$. Then, by $\ND_{\Psi,\ge}$,
\begin{align*}
P_{t}^D\I_{B(x,\delta_{D}(x))}(x)
&\ge \Pp^x(\tau_{B(x,\delta_{D}(x))}>t)\\
& \ge  \Pp^x(\tau_{B(x,\delta_{D}(x))}>n c_1 \Psi(\delta_{D}(x)))\\
&\ge \int_{B(x,c_0\delta_{D}(x))} \cdots \int_{B(x,c_0\delta_{D}(x))}p^{B(x,\delta_{D}(x))}(c_1 \Psi(\delta_{D}(x)),x,z_1)  \cdots\\
& \qquad \times p^{B(x,\delta_{D}(x))}(c_1 \Psi(\delta_{D}(x)),z_{n-1},z_{n})\,dz_n\cdots dz_1\\
&\ge C_2^n\ge C_3e^{- C_4t/\Psi(\delta_{D}(x))}.
\end{align*}

Therefore, combining with both estimates above, we prove the desired assertion.
\end{proof}

Now, we are in the position to present
necessary conditions for the intrinsic ultracontractivity of
$(P_t^{D})_{t \ge 0}$.

\begin{proposition}\label{p3-1a}
Suppose that $D$ is an open subset satisfying \eqref{e:do}, and that
$\ND_{\Psi,\ge}$ holds on $D$ and $\OD_{\Gamma,\le}$ hold for the Dirichlet heat kernel $p^D(t,x,y)$.
If $(P_t^{D})_{t \ge 0}$ is intrinsically
ultracontractive, then
\begin{equation}\label{p3-1-2}
\lim_{x \in D \text{ and } |x|\rightarrow \infty} \Psi(\delta_D(x)) \big| \log \Gamma (|x|)\big|=0.  \end{equation}
\end{proposition}
\begin{proof}
Since $(P_t^D)_{t\ge0}$ is intrinsically
ultracontractive, then, for any $t>0$ there is a constant
$C_{t,D}\ge 1$ such that for all $x,$ $y\in D$,
$$C_{t,D}^{-1} \phi_1(x)\phi_1(y)\le p^D(t,x,y)\le C_{t,D} \phi_1(x)\phi_1(y),$$ see e.g. \cite[Theorem 3.2]{DS}.
Thus, for any $x\in D$ it holds that
$$P^D_t\I_{B(x,\delta_D(x))}(x)=\int_{B(x,\delta_D(x))}p^D(t,x,y)\,dy\le C'_{t,D} \|\phi_1\|_\infty \phi_1(x)\delta_D(x)^d.$$
For any compact subset $D_0$ of $D$ with $|D_0|>0$ and for any $x\in D$, we also have
$$P^D_t\I_{D_0}(x)=\int_{D_0}p^D(t,x,y)\,dy\ge C_{t,D}^{-1}\phi_1(x)\int_{D_0}\phi_1(y)\,dy\ge C''_{t,D_0,D} \phi_1(x).$$
Combining with both inequalities above, we know that for every
$t>0$ and any compact subset $D_0$ of $D$ with $|D_0|>0$,
there exists a constant $c_{1,t, D_0,D}>0$ such that
\begin{equation}\label{p3-1-3aa} P^D_t\I_{B(x,\delta_D(x))}(x)\le c_{1,t, D_0,D}P^D_t\I_{D_0}(x)\delta_D(x)^d, \quad x\in D.\end{equation}

Furthermore, by $\ND_{\Psi,\ge}$, \eqref{e:do} and Lemma \ref{l:sur_low}, there exist constants $c_2,c_3>0$ such that for every $t>0$ and $x\in D$ with $|x|$ large enough,
\begin{align}
P_{t}^D\I_{B(x,\delta_D(x))}(x)\ge c_2 e^{-c_3t/\Psi(\delta_D(x))}.
\label{p3-1-3a}
\end{align}
On the other hand, according to \eqref{e:do} and $\OD_{\Gamma,\le}$, there is a constant $t_0\in (0,1)$ such that for any
compact set $D_0\subset D$, $t\in (0,t_0]$
and $x \in D$ with $|x|\ge 2 R_0:=2\sup_{z\in D_0}|z|$ large enough,
\begin{equation}\label{p3-1-4a}
\begin{split}
P_{t}^D \I_{D_0}(x)\delta_D(x)^d&\le \int_{D_0}
p_D(t,x,y)\,dy\le c_{4,t, D}|D_0|\sup_{y \in D_0}\Gamma(|x-y|)\\
&\le c_{5,t,D_0, D}\Gamma(|x|-R_0)\le c_{5,t,D_0,D}\Gamma(|x|/2).
\end{split}
\end{equation}
Below, we fix a compact subset $D_0\subset D$ with $|D_0|>0$.
Therefore, combining \eqref{p3-1-3aa} with \eqref{p3-1-3a} and \eqref{p3-1-4a}, we arrive at that there is $t_0>0$ such that for all $t\in (0,t_0]$  and
$x \in D$ with $|x|$ large enough
so that $\delta(x)\le 1$,
\begin{equation}\label{e:ssll} e^{-c_3t/\Psi(\delta_D(x))}\le  c_{6,t}\Gamma(|x|/2).\end{equation}

Next, we assume that
\eqref{p3-1-2} does not hold.
Then there exist a constant $c_7>0$ and a sequence $\{x_n\}_{n=1}^{\infty}\subseteq D$ such that
$\lim_{n \rightarrow \infty}|x_n|=\infty$ and $
{|\log\Gamma(|x_n|/2)|} \Psi (\delta_D(x_n))\ge {c_7}
$ for all $n\in \N$. 
Thus, for any $t>0$,
\begin{equation*}\label{p3-1-5a}
\begin{split}
\exp\Big(-c_7t/\Psi(\delta_D(x_n))\Big)\ge \exp\Big(-t|\log\Gamma(|x_n|/2)| \Big)=\Gamma(|x_n|/2) ^{t}.
\end{split}
\end{equation*}
Taking $t =t_1:=((c_3 t_0)/(2c_7))\wedge(1/2)$ in the inequality above and using \eqref{e:ssll},
we get
$$0< c_{8,t_1}^{-1/(1-t_1)} \le \Gamma(|x_n|/2).$$
Since $\lim_{s \rightarrow \infty}
\Gamma(s)=0$,
letting  $n \to \infty$ we get a contradiction. 
 That is,
if $(P_t^D)_{t \ge 0}$ is intrinsically
ultracontractive, then \eqref{p3-1-2} does hold. The proof is complete.
\end{proof}

\section{
Intrinsic ultracontractivity
of Dirichlet semigroups: explicit results
}\label{sec444}
We continue considering the symmetric Hunt process $X=\{X_t,t\ge0; \Pp^x,x\in \R^d\}$ as in Subsection \ref{subsection1-1}. The associated Dirichlet form $(\mathscr{E}, \mathscr{F})$ is given by  \eqref{non-local}.
Let $D$ be an open set of $\R^d$.  Denote by $(P_t^D)_{t\ge0}$ and $p^D(t,x,y)$ the Dirichlet semigroup and the Dirichlet heat kernel associated with the killed process $X^D$ of the process $X$ upon exiting $D$, respectively. \emph{Throughout this section, we suppose that
both $\J_{\Phi,\ge}$ and $\ext_{\Psi,\ge }$ hold.} See Definitions \ref{de:J} and \ref{d:ext}.
Recall that we also always assume that the Dirichlet heat kernel $p^D(t,x,y)$ exists, and $p(t,\cdot,\cdot)$ is bounded, continuous and strictly positive on $D\times D$ for every $t>0$.

In this section, we will apply results in previous sections to establish
criteria for the intrinsically ultracontractivity of $(P^D_t)_{t\ge0}$ on two specific
types of open sets.
One is horn-shaped regions, and the other one is unbounded and disconnected open sets with locally $\kappa$-fat property.

\subsection{Horn-shaped regions}\label{sec4444}

 Let $D_f=\{x\in \R^d: x_1>0, |\tilde x|< f(x_1)\}$ be a horn-shaped region with the reference function $f$, where  $f:(0,\infty)\to (0,\infty)$ is bounded, continuous and satisfies that $\lim_{u\to\infty} f(u)=0$.
 As mentioned above, we assume that $\J_{\Phi,\ge}$ holds.
Then, according to
Corollary \ref{c:kacomp},
it is easy to see that
the semigroup $(P_t^{D_f})_{t\ge0}$ is compact if
$D_f^c$ is $\kappa$-fat.
Furthermore, by Proposition \ref{p5-1-0} and the assumption that the Dirichlet heat kernel $p^{D_f}(t,x,y)$ exists and $p^{D_f}(t,\cdot,\cdot)$ is bounded, continuous and strictly positive on $D_f\times D_f$ for every $t>0$, the corresponding ground state $\phi_1$ can be chosen to be bounded, continuous and strictly positive on $D_f$.

To illustrate how powerful Theorem \ref{T:IU} is, we begin with horn-shaped regions with general reference functions.
For non-negative measurable function $f$, let $f^*(r)=\sup_{s\ge r} f(s)$ and ${f^*}^{-1}(r)=\inf\{s>0: f^*(s)\le r\}$ for $r>0$.

\begin{proposition}\label{p4-2}
Assume that $\J_{\Phi,\ge}$ and $\ext_{\Psi,\ge }$ hold.  Let $D_f$ be a horn-shaped region
such that $D_f^c$ is
$\kappa$-fat.
Then, the following two statements hold true.

\begin{itemize}
\item[(1)] There are positive constants
$c_1, c_2$
such that for all $x\in D_f$,
\begin{equation*}
\phi_1(x)\ge c_1 \Psi(\delta_{D_f}(x)\wedge c_2)\left(\inf_{|y-z|\le |x|+
c_1} J(y,z)\right).
\end{equation*}

\item[(2)]
The super Poincar\'{e} inequality \eqref{e:whole} holds with
\begin{equation}\label{p4-2-1}
\begin{split}
\beta(s)=&C_1 (\Phi^{-1}(s)\wedge 1)^{-d} \Psi(\Phi^{-1}(s)\wedge 1)^{-2} \sup_{|y-z|\le
C_1{f^*}^{-1}(C_2(\Phi^{-1} (s)\wedge 1))}\frac{1}{J(y,z)^2}.
\end{split}
\end{equation}
Here $C_1, C_2$ are positive constants.
Consequently, the
semigroup $(P^{D_f}_t)_{t\ge0}$ is intrinsically ultracontractive, if
$\beta(s)$ given above
satisfies \eqref{e:iurate}.\end{itemize}
\end{proposition}
\begin{proof}
(1) The first assertion is a consequence of \eqref{e:lower1-1}.

\smallskip

(2) Denote by $D=D_f$ for simplicity. Since $D^c$ is $\kappa$-fat, we can use \eqref{p3-1-1-1} in Proposition \ref{p3-1}. Let $c_0>0$
be the constant $c_2$ in \eqref{p3-1-1-1}.
Applying
\eqref{p3-1-1-1} and \eqref{beta}, we find that \eqref{e:whole} holds with the rate function $\beta(s)$ satisfying that
\begin{align*}
\beta(s)&\le c_1\inf\left\{
\frac{ (\Phi^{-1}(s)\wedge r)^{-d}}{\inf_{x\in D_{R+1, r/2} }\phi_1^2(x)}:
\inf_{x\in
D\setminus D_{R,r}}V_D(x)\ge 1/s
\textrm{ and  }r \in (0,1/2)  \right\}\\
&\le c_1\inf\left\{
\frac{ (\Phi^{-1}(s)\wedge r)^{-d}}{\inf_{x\in D_{R+1, r/2} }\phi_1^2(x)}:
    \sup_{x\in D\setminus D_{R,r}}{\Phi(\delta_D(x))} \le c_0 s
  \textrm{ and }r \in (0,1/2) \right\}.
\end{align*}
According to \eqref{e:Phi}, we can take a constant $c_* \in (0, r_0/2)$ small enough so that $c_* \Phi^{-1}(s) \le \Phi^{-1}(c_0s)$  for all $s>0$.
Next, we choose $s_0>0$ small enough such that for any $s \in (0,s_0]$,
$r=c_*(\Phi^{-1}(s)\wedge 1) < 1/2$
and $R={f^{*-1}}\big(c_*(\Phi^{-1}(s)\wedge 1)\big)<\infty$. Taking this $r$ and $R$ in the infimum of the last term in the display above, and using assertion (1) for lower bound estimates of $\phi_1$, we
find that $\beta(s)$ is not bigger than $$c_2
{ (\Phi^{-1}(s)\wedge 1)^{-d}}\Psi(\Phi^{-1}(s)\wedge 1)^{-2}
\sup_{|y-z|\le
{f^*}^{-1}(c_*(\Phi^{-1} (s)\wedge 1))}\frac{1}{J(y,z)^2}.$$ This proves the second assertion.
\end{proof}

Instead of
 Proposition \ref{p:lower1}, we can use Proposition \ref{p:lower2} to obtain the following result.
\begin{proposition}\label{p4-1}
Assume that $\J_{\Phi,\ge}$ and  $\ext_{\Psi,\ge }$ hold. Let $D_f$ be a horn-shaped region such that $D_f^c$ is $\kappa$-fat.
Then the following statements hold true.
\begin{itemize}
\item[(1)]  There are positive constants $c_1, c_2$ such that for all $x\in D_f$,
\begin{equation*}\begin{split}\label{e:phi-f} \phi_1(x)\ge &c_1 \Psi(\delta_{D_f}(x)\wedge c_1)\exp\left(-c_2(1+|x|)\log \left(e+\frac{|x|}{\inf\limits_{c_1\le s\le |x|+1} f(s)}\right)\right).\end{split}\end{equation*}

\item[(2)] The super Poincar\'{e} inequality \eqref{e:whole} holds with
\begin{equation}\label{p4-1-22}
\begin{split}
\beta(s)=&C_1 (\Phi^{-1}(s)\wedge 1)^{-d} \Psi(\Phi^{-1}(s)\wedge 1)^{-2}\\
&\times \exp\bigg\{C_1\bigg[\big(1+F(s)\big)
\log\bigg(e+\frac{F(s)}{\inf_{C_2\le r \le F(s)} f(r)}\bigg)\bigg]\bigg\},
\end{split}\end{equation}
where $$F(s)={f^*}^{-1}(C_3(\Phi^{-1}(s)\wedge1)$$ and
$C_i$ $(i=1,2,3)$ are positive constants.
Consequently, the
semigroup $(P^{D_f}_t)_{t\ge0}$ is intrinsically ultracontractive, if
$\beta(s)$ defined above
satisfies \eqref{e:iurate}.
 \end{itemize}
\end{proposition}

\begin{proof} We
 denote $D=D_f$ throughout the proof.

\smallskip

(1) Choose $r_*\ge 4$ large enough such that $f(r)\le 2^{-5}$ for any
$r\ge r_*-1$.
We first
consider $x\in D$ with
$x_1\ge 2r_*$. Take $x^{(i)}=(r_*+{i(x_1-r_*)}/{n},\tilde 0)=:(x_1^{(i)},\tilde 0)$ for $0\le
i\le n-1$ with
$n:=\lceil 4(x_1-r_*)\rceil$,
 and
$x^{(n)}=(x_1,\tilde x)=x$. Then,
$$2^{-3}\le |x^{(i-1)}-x^{(i)}|=\frac{x_1-r_*}{n}\le 2^{-2},\quad 1\le i\le n-1,$$
and
$$ 2^{-3}\le \frac{x_1-r_*}{n}\le |x^{(n-1)}-x^{(n)}| \le \frac{x_1-r_*}{n}+f(x_1)\le 2^{-1}.$$
Let
$$r_i=\frac{\delta_{D}(x^{(i)})}{3}\wedge2^{-5} \wedge \tilde r_0,\quad 0\le i\le n-1  \quad \text{ and } \quad r_n=\frac{\delta_D(x)}{3}\wedge
2^{-5}\wedge \tilde r_0,$$
where $\tilde r_0$ denotes the constant in \eqref{e:exit}.
By the definition of horn-shaped region, for all $0\le i\le n-1$,
$$r_i\ge c_1\inf_{|s-{x_1^{(i)}}|\le f({x_1^{(i)}})} f(s)\ge c_2 \inf_{r^*-1\le s\le |x|+1} f(s).$$ Indeed, let $y=(y_1, \tilde{y})\in \partial D$ such that $|x^{(i)}-y|=\delta_D(x^{(i)})$. Since $|\tilde{y}|=f(y_1)$ and $|\widetilde{x^{(i)}}|=0$,
we have
\begin{equation}\label{e:point}
\begin{split}
\delta_D(x^{(i)})&{ \ge |\tilde y-\tilde x^{(i)}|}= f(y_1)\ge \inf_{s\in { B(x_1^{(i)}, \delta_D(x^{(i)}))}} f(s)\ge \inf_{s\in {B(x_1^{(i)}, f( x_1^{(i)}))}} f(s).
\end{split}
\end{equation}

 Therefore,
combining all the estimates above with \eqref{e:lower2-1}, we obtain
that for every $x_1\ge 2r_*$,
\begin{equation*}\label{e:aa}
\begin{split}
\phi_1(x) & \ge c_3\Psi(\delta_D(x)\wedge
c_3)
\exp\Bigg[-c_4\Big(n\log n+\sum_{i=0}^{n-1}\log\frac{1}{r_i}\Big)\Bigg]\\
&\ge c_3\Psi(\delta_D(x)\wedge
c_3)\exp\Bigg[-c_5 |x|\log \Bigg(\frac{|x|}{\inf\limits_{r^*-1\le s\le |x|+1} f(s)}\Bigg) \Bigg],
\end{split}
\end{equation*}
where in the last inequality we have used the fact that
$c_6|x|\le n \le c_7|x|$ if $x_1\ge 2r_*$. This proves that \eqref{p4-1-22} holds for every $x\in D$ such that $x_1\ge 2r_*$.

Next, we consider $x\in D$ such that $x_1< 2r_*$.  Since $D_{r_*}:=\{x\in D: x_1< 2r_*\}$ is bounded, we can find $x_0\in D$ and positive constants
$s_0,s_1,a_1,a_2,N_0$ such that
\begin{itemize}
\item [(i)] $\overline{B(x_0,s_0)}\subseteq D_{r_*}$;

\item [(ii)] For every
$x\in D_{r_*}\setminus B(x_0,s_0)$,
 $x\sim_{(n(x); a_1, a_2)} x_0$ for some positive integer
$n(x)\le N_0$, and the connected points $\{x^{(i)}:1\le i \le n(x)\}$ satisfies that $\delta_D(x^{(i)})\ge s_1>0$
for all $1\le i \le n(x)$.
\end{itemize}
Therefore, by \eqref{e:lower2-1} we know that \eqref{p4-1-22} holds for every $x\in D_{r_*}\setminus B(x_0,s_0)$.
On the other hand, since $\overline{B(x_0,s_0)}\subseteq D$ and
$\phi_1$ is continuous, strictly positive on $D$,
$\inf_{z \in B(x_0,s_0)}\phi_1(z)\ge c_8$ for some constant $c_8>0$. So, by changing the constants properly,
\eqref{p4-1-22} still holds for any
$x \in D_{r_*}$.

Combining all the estimates above, we have shown that \eqref{p4-1-22} holds for all $x \in D$.

\smallskip

(2) With (1) at hand, the argument for the proof  of (2) is the same as that for Proposition \ref{p4-2} (2). So we skip the details.
 \end{proof}

As a consequence of Propositions \ref{p4-2} and \ref{p4-1}, we have the following corollary.

\begin{corollary}\label{th1}
Assume that $\J_{\Phi,\ge}$ and $\ext_{\Psi,\ge }$ hold.
Let $D_f$ be a horn-shaped region such that $D_f^c$ is
$\kappa$-fat. Then, the following statements hold.
\begin{itemize}
\item[(1)]
Suppose that
there exist constants $\alpha,\theta, s_0, c_1,c_2>0$ such that
\begin{align}
\label{e:Expl0}
J(x,y)\ge c_1|x-y|^{-\alpha},\quad |x-y|>1,
\end{align}
and
\begin{equation}\label{th1-2}
 f(s) \le c_2\Phi^{-1}(\log^{-\theta }s),\quad s\ge s_0.
\end{equation}
 Then
$(P_t^{D_f})_{t \ge 0}$ is intrinsically ultracontractive if
$
\theta>{1}.
$

\item[(2)]
Suppose that
there exist constants $\theta, s_0,c_1, c_2,c_3>0$ and $\gamma\in (0,\infty]$
 such that \begin{align}
\label{e:Expl}
J(x,y)\ge c_1e^{-c_2|x-y|^{\gamma}},\quad |x-y|>1,
\end{align}
and
\begin{equation}\label{th1-4}
 f(s)\le c_3 \Phi^{-1}( s^{-\theta}),\quad s\ge s_0.
\end{equation}
Then
$(P^{D_f}_t)_{t\ge0}$ is intrinsically ultracontractive if
$
\theta>{\gamma}.
$

\item[(3)] Suppose that $\gamma \in (1,\infty]$ in \eqref{e:Expl}
and that
\begin{equation}\label{th1-4-1}
 c_3s^{-\eta}\le f(s)\le c_4 \Phi^{-1}( s^{-\theta}),\quad s\ge s_0
\end{equation}
for some constants $c_3, c_4,s_0,\theta,\eta>0$. Then
$(P^{D_f}_t)_{t\ge0}$ is intrinsically ultracontractive if
$
\theta>1.
$

\item[(4)] Suppose that $\gamma \in (1,\infty]$ in \eqref{e:Expl}
and that
 $$
 c_3\exp(-c_4s^{\theta_1})\le f(s)\le c_5 \exp(-c_6s^{\theta_2}),\quad s\ge s_0
$$
 for some constants $c_3, c_4, c_5, c_6, s_0>0$ and $\theta_1\ge \theta_2>0$.  Then
$(P^{D_f}_t)_{t\ge0}$ is intrinsically ultracontractive.
\end{itemize}
\end{corollary}

\begin{proof}
(1) By \eqref{th1-2}, for $s>0$ small enough, it holds that
$${f^*}^{-1} \circ \Phi^{-1} (s) \le c_1\exp\big(c_2s^{-{1}/{\theta}}\big).$$
From the above estimate, \eqref{e:Phi}, \eqref{e:Psi}, \eqref{p4-2-1} and
the assumption \eqref{e:Expl0}, we know that \eqref{e:whole} holds with
\begin{equation*}
\beta(s)\le
c_3\exp\Big(c_4 s^{-{1}/{\theta}}\Big),\quad
0<s<s_1
\end{equation*}
for some $s_1>0$ small enough. Hence, it is easy to see that when
$\theta>1 $, \eqref{e:iurate} is satisfied, which
shows that $(P_t^{D_f})_{t \ge 0}$ is intrinsically ultracontractive.

\smallskip

(2) By \eqref{th1-4}, it
is easy to verify that for $s>0$ small enough
\begin{equation}\label{th1-4-2}
{f^*}^{-1}\circ \Phi^{-1} (s) \le c_5 s^{-{1}/{\theta}}.
\end{equation}
Combining this with \eqref{p4-2-1}, \eqref{e:Phi}, \eqref{e:Psi} and
the assumption \eqref{e:Expl} with $\gamma\in(0,1]$, we get that \eqref{e:whole} holds \begin{equation*}
\beta(s)\le
c_6\exp\Big(c_7 s^{-{\gamma}/{\theta}}\Big),\quad
0<s<s_1
\end{equation*}
for some $s_1>0$ small enough. Obviously
$\theta>{\gamma}$ implies \eqref{e:iurate}, which
shows that $(P_t^{D_f})_{t \ge 0}$ is intrinsically ultracontractive if
$\theta>{\gamma}$.

\medskip

(3) Note that the inequality \eqref{th1-4-2} is still true. Then, according to \eqref{p4-1-22}, we have the following estimate for the
rate function $\beta(s)$ in \eqref{e:whole} --- there exist constants  $s_1>0$ and $c_i>0$ $(i=8,9)$ such that
\begin{equation*}
\beta(s)\le c_8\exp\bigg[c_9s^{-{1}/{\theta}}\Big(\log
\frac{1}{s}\Big)\bigg],\quad 0<s<s_1,
\end{equation*} where the first inequality in \eqref{th1-4-1} was used.
 Hence, \eqref{e:iurate} holds if $\theta>1$, which
proves that $(P_t^{D_f})_{t \ge 0}$ is intrinsically ultracontractive
when $\theta>1$.

(4) The proof is based on Proposition \ref{p4-1} as that of (3), and we can see that
there exist constants $s_1>0$  and $c_i>0$ $(i=10, \dots,17)$ such that for all $0<s<s_1$
\begin{align*}
\beta(s)
\le& c_{10} (\Phi^{-1}(s)\wedge 1)^{-d} \Psi(\Phi^{-1}(s)\wedge 1)^{-2}\\
&\times
\exp\left[c_{11} \left(\log \frac{1}{\Phi^{-1}(s)}\right)^{1/\theta_2}
\log \left(\frac{  c_{12}\log(1/\Phi^{-1}(s))^{1/\theta_2} }{
\exp\left[-c_{13} \left(\log (1/\Phi^{-1}(s))\right)^{\theta_1/\theta_2}
\right]}\right)\right]
\\
\le &c_{14}\exp\left[c_{15} \left(\log \frac{1}{\Phi^{-1}(s)}\right)^{(1+\theta_1)/\theta_2}\right]\\
\le & c_{16}\exp\left[c_{17} \left(\log \frac{1}{s}\right)^{(1+\theta_1)/\theta_2}\right],
\end{align*}
where in the second inequality we used \eqref{e:Phi} and \eqref{e:Psi}, and the
last inequality follows from \eqref{e:Phi} again.
The rate function $\beta$ above satisfies \eqref{e:iurate}, which yields that $(P_t^{D_f})_{t \ge 0}$ is intrinsically ultracontractive.
\end{proof}

Thanks to the milder assumptions in Corollary \ref{th1},
we can obtain sufficient conditions for intrinsic ultracontractivity of $(P_t^{D_f})_{t \ge 0}$ for a class of jumping processes with variable orders as follows.
\begin{example}\label{ex4-1}
Suppose that a function $\alpha:\R^d\rightarrow (0,2)$ satisfies
$0<\la\le \alpha(x)\le \ua<2$ and
\begin{equation*}
\big|\alpha(x)-\alpha(y)\big|\le
c_1\log^{-1}\left(\frac{2}{|x-y|}\right),\quad |x-y|<1
\end{equation*}
for some positive constant $c_1$.
We consider the non-local symmetric Dirichlet form $(\mathscr{E}, \mathscr{F})$ given by  \eqref{non-local}, and suppose that the jumping kernel $J(x,y)$ satisfies
\begin{align*}
c_2\left(\frac{1}{|x-y|^{d+\alpha(x)}}\wedge \frac{1}{|x-y|^{d+\alpha(y)}}\right) \le & J(x,y) \\
\le& c_3\left(\frac{1}{|x-y|^{d+\alpha(x)}}\vee\frac{1}{|x-y|^{d+\alpha(y)}}\right)
\end{align*}
for all $x,y\in\R^d$ and some positive constants $c_2,c_3$.
Then,
according to \cite[Example 2.3 and Theorem 3.5]{BKK}, there exists a symmetric Hunt process $X=(X_t,t\ge0; \Pp^x, x\in \R^d)$ on $\R^d$ associated with $(\mathscr{E}, \mathscr{F})$, and the process $X$ possesses the transition density function (i.e.,\ heat kernel)
$p(t,\cdot,\cdot):\R^d\times \R^d \rightarrow \R_+$ so that $p(t,x,y)$ is jointly continuous on $(t,x,y)$. Following the argument of \cite[Proposition 8]{CWi14}, one can prove that $p(t,x,y)$ is strictly positive for every $t>0$ and $x,y\in \R^d$.

Let $D_f$ be a horn-shaped region such that $D_f^c$ is $\kappa$-fat. Since $D_f$ is connected,
it is easy to verify that the associated Dirichlet heat kernel $p^{D_f}(t,\cdot,\cdot)$ is bounded, continuous and strictly positive on $D_f\times D_f$ for every $t>0$, see e.g.\ Proposition \ref{p:pDcont} and \cite[Corollary 7 and Remark 8 (2)]{CWi14}.  Therefore, all the assumptions in Subsection \ref{subsection1-1} are fulfilled in this setting.

It is clear that $\J_{\Phi,\ge}$ holds with $\Phi(r)=r^{\underline{\alpha}}$. On the other hand, according to Example \ref{exp:order}, $\ext_{\Psi,\ge}$ holds with $\Psi(r)=r^{\ua+(\ua-\la)d/\la}$. In fact, according to \cite[Theorem 2.1 and Example 2.3]{BKK} and the continuity assumption on $\alpha(x)$, we can obtain that $\ext_{\Psi,\ge}$ holds with $\Psi(r)=r^{\ua}$. Now, according to Corollary \ref{th1} (1), we know that $(P_t^{D_f})_{t \ge 0}$ is intrinsically ultracontractive, if there are constants $\theta>{1}/{\la}$ and $c_0,s_0>0$ such that for all $s\ge s_0$, $$f(s) \le c_0\log^{-\theta }s.$$
\end{example}

\ \

For the remaining part of  this section, we consider the regular Dirichlet form $(\mathscr{E}, \mathscr{F})$ whose jumping kernel $J(x,y)$ given by \eqref{e:jgeneral}. (Note that we do not
assume that Assumptions ${\bf(K_\eta)}$ and ${\bf (SD)}$ hold here.) As mentioned in Subsection \ref{sec1.2}, associated with $(\mathscr{E}, \mathscr{F})$ there is an Hunt process $X$ on $\R^d$, who has a transition density function $p(t,x,y)$ with respect to the Lebesgue measure satisfying that for every $t>0$,
$p(t,\cdot,\cdot): \R^d\times\R^d \to (0,\infty)$ is bounded, continuous and strictly positive.
According to \cite[Lemma 2.5]{CKK3} and \cite[Theorem 2.4 (ii)]{CKK3}, we know that
for every open set $D$
both $\ext_{\Phi,\ge}$ and $\ND_{\Phi,\ge}$ hold $D$ with $r_0=1$.
Furthermore, if $\gamma=0$, then, by
\cite[Theorem 1.2]{CK1} (for the case $\gamma_1=\gamma_2=0$ in  \cite[(1.12)]{CK1}),
$\OD_{\Gamma,\le}$ holds with $t_0=1$ and
\begin{equation}\label{e3-4a}\Gamma(s)=
{s^{-d}\Phi(s)^{-1}};\end{equation}
if $\gamma>0$, then, by \cite[(1.13) and (1.16) in Theorem 1.2 and (1.20) in Theorem
1.4]{CKK1},
$\OD_{\Gamma,\le}$ holds with $t_0=1$ and
\begin{equation}\label{e3-4}\Gamma(s)=\exp\Big(-c(1+s)^{\gamma \wedge 1}
\log^{{(\gamma-1)_+}/{\gamma}}(1+s)\Big)\end{equation} for some constant $c>0$.

Now, we can prove the assertions for the intrinsic ultracontractivity of $(P_t^{D_f})_{t\ge0}$ in Theorems \ref{ex1-1n1} and  Theorem \ref{ex1-1n3}.

\begin{proof}
[Proofs of Theorems $\ref{ex1-1n1}(1)(a)$ and $\ref{ex1-1n1}(2)(a)$]
The
sufficiency
of the intrinsic ultracontractivity of $(P_t^{D_f})_{t\ge0}$ can be easily seen from Corollary
\ref{th1}(1)--(3).
So, one only need to verify the
necessity of the corresponding assertions.

(1) Suppose that $\gamma=0$ and that \eqref{ex1-1-1ag} holds with $\theta\le 1$.
Then, by the definition of horn-shaped region, $$\Phi(\delta_{D_f}(x_n))\ge c_1{\log^{-\theta}n}$$ for $x_n:=(n,\tilde 0)$ and $n\ge1$ large enough. This along with \eqref{e3-4a}
implies that \eqref{p3-1-2} does not hold. Thus, by Proposition \ref{p3-1a},
$(P_t^{D_f})_{t \ge 0}$ is not intrinsically ultracontractive.

(2) Suppose that
$\gamma \in (0, \infty]$  and that  \eqref{ex1-1-2g} holds with $\theta\le \gamma \wedge 1$. Then,
$$\Phi(\delta_{D_f}(x_n))\ge c_1n^{-\theta}$$ for $x_n:=(n,\tilde 0)$ and $n\ge1$ large enough. Combining this with \eqref{e3-4} and \eqref{p3-1-2}, we can see from Proposition \ref{p3-1a} that
$(P_t^{D_f})_{t \ge 0}$ is not intrinsically ultracontractive.
\end{proof}

\begin{proof}[Proof of Theorem $\ref{ex1-1n3}(a)$]
This immediately follows from Corollary \ref{th1} (4).
\end{proof}

\subsection{Unbounded and disconnected open set with locally $\kappa$-fat property}
Recall that we consider the regular Dirichlet form $(\mathscr{E}, \mathscr{F})$ whose jumping kernel $J(x,y)$ given by \eqref{e:jgeneral}.
In this part, let $h:
[0,\infty)
\rightarrow [12,\infty)$ be a measurable function such that
$
\lim_{s \rightarrow \infty}h(s)=\infty.
$
Define
\begin{equation}\label{e4-1}
\begin{split}
D_0&=\bigcup_{n=1}^{\infty}\bigcup_{m=0}^{\lfloor h(n)\rfloor-1}\Big(n+\frac{m}{\lfloor h(n)\rfloor},n+\frac{m+1/2}{\lfloor h(n)\rfloor}\Big),\\
D&=\{x\in \R^d: |x|\in D_0\ \text{or}\ |x|\le 1\}.
\end{split}
\end{equation} The construction of the open set $D$ above is partially inspired by \cite[Example 4]{Kw}.

It is easy to see that for each $n\ge 1$, the set $\overline{D^c\cap\{x\in \R^d: n<|x|<n+1\}}$ is $\Big(\kappa,\frac{1}{4\lfloor h(n)\rfloor}\Big)$-fat at every
point of $\overline{D^c\cap\{x\in \R^d:n<|x|<n+1\}}$, and $\delta_D(x)\le \frac{1}{\lfloor h(n)\rfloor}$ for all $x \in D$ with $n<|x|<n+1$. This along with
\eqref{p3-1-1} yields that for all $n\ge 0$,
\begin{equation}\label{e4-2}
V_D(x) \ge c_1\Phi(\delta_D(x))^{-1}\ge c_2\Phi(h(n)^{-1})^{-1},\quad x\in D\text{ and  }n<|x|<n+1.
\end{equation}
In particular, due to the fact that $\lim_{n \rightarrow \infty}h(n)=\infty$, \eqref{e:com-3} holds, and so $(P_t^D)_{t\ge0}$ is compact. On the other hand, by Propositions \ref{p:pDcont}, \ref{p:pDp} and \ref{p5-1-0} in the Appendix, we know that the associated ground state $\phi_1$ can be chosen to be  bounded, continuous and strictly positive on $D$.

The next result illustrates again that our results for the intrinsic
ultracontractivity of Dirichlet semigroup are optimal in some sense.
\begin{theorem}\label{ex3-1} Consider the regular Dirichlet form $(\mathscr{E}, \mathscr{F})$ whose jumping kernel $J(x,y)$ given by \eqref{e:jgeneral}.
Let $D$ be the open set defined by \eqref{e4-1}. Then, we have the following statements.
\begin{itemize}
\item[(1)]
Suppose that $\gamma=0$ in \eqref{e:Exp} and
\begin{equation}\label{ex3-1-1}
h(s)\simeq \frac{1}{\Phi^{-1}(\log^{-\theta}s)}, \quad s \ge 2
\end{equation} for some $\theta>0$. Then
$(P^{D}_t)_{t\ge0}$ is intrinsically
ultracontractive if and only if
$\theta>1.$

\item[(2)] Suppose  that $\gamma \in (0,\infty]$ in \eqref{e:Exp} and
\begin{equation}\label{ex3-1-2}
h(s)\simeq  \frac{1}{\Phi^{-1}(s^{-\theta})}, \quad s \ge 1
\end{equation} for some $\theta>0$.
 Then  $(P^{D}_t)_{t\ge0}$ is intrinsically
ultracontractive if and only if
$\theta>{\gamma} \wedge 1.$
\end{itemize}
\end{theorem}
\begin{proof}
(1) Suppose that $\gamma=0$ and \eqref{ex3-1-1} holds.
Since
$\lim_{x \in D \text{ and } |x|\rightarrow \infty}\delta_D(x)=0$ and $\ext_{\Phi,\ge}$ holds, it follows from \eqref{e:lower1-1} that there is a constant $c_1>0$ such that for all $x\in D$
\begin{equation*}
\phi_1(x)\ge \frac{c_1\Phi(\delta_D(x))}{|x|^d\Phi(|x|)}.
\end{equation*}
Applying this estimate and \eqref{e4-2} into \eqref{beta}, we get that
the intrinsic super Poincar\'e inequality \eqref{e:whole} holds with the rate function $\beta(s)$ as follows
\begin{equation*}
\begin{split}
\beta(s)&\le c_2\inf\left\{
\frac{ (\Phi^{-1}(s)\wedge r)^{-d}}{\inf_{x\in D_{R+1, r/2} }\phi_1^2(x)}:
 \inf_{x\in
D\setminus D_{R,r}}V_D(x)\ge 1/s
\textrm{ and  }r \in (0,c_3)  \right\}\\
&\le c_4\inf\left\{
\frac{ (\Phi^{-1}(s)\wedge r)^{-d}R^{2d}\Phi(R)^2}{\Phi^2(r)}:
 c_5\Phi\big(r\vee h(R)^{-1}\big)\le s\textrm{ and  }r \in (0,c_3)  \right\}.
\end{split}
\end{equation*}
In the infimum above taking $r=c_6\Phi^{-1}(s)$ and $R=\exp(c_7s^{-{1}/{\theta}})$ for $s>0$ small enough and with some suitable positive constants $c_6, c_7$ (thanks to \eqref{e:Phi} and \eqref{ex3-1-1}),  we arrive at
\begin{equation*}
\begin{split}
\beta(s)\le c_8\exp(c_8s^{-{1}/{\theta}}),\quad 0<s\le s_0
\end{split}
\end{equation*}
for some constant $s_0>0$. This implies that when $\theta>1$, the rate function above satisfies \eqref{e:iurate}, hence $(P_t^D)_{t \ge 0}$ is intrinsically ultracontractive.

If \eqref{ex3-1-1} holds with $\theta\le 1$, then, by the definition of $D$, we can find a sequence
$\{x_n\}_{n=2}^{\infty}\subset D$
 such that
$n<|x_n|< n+1$ and $$\delta_D(x_n)\ge
\frac{1}{4\lfloor h(n)\rfloor}
\ge c_9\Phi^{-1}\big(\frac{1}{\log^{\theta}n}\big).$$
This together with
\eqref{e:Phi}, \eqref{e3-4a}  and the fact that $\theta \le 1$ shows that \eqref{p3-1-2} does not hold. Thus, by Proposition \ref{p3-1a}, we know $(P_t^D)_{t \ge 0}$ is not intrinsically
ultracontractive.

(2) We first consider the case $\gamma>1$. Assume that \eqref{ex3-1-2} holds.
For every $x \in D$ with $|x| \ge 4 $, let $n=\lfloor 2|x|   \rfloor$,
$x^{(n)}=x$, $x^{(0)}=0$, $x^{(1)}=\frac{1}{2}\cdot\frac{x}{|x|}$ and
$$x^{(i)}:=\Bigg(\Big\lfloor \frac{i}{2}\Big\rfloor
+\frac{\Big\lfloor{\frac{(1+2(i-2\lfloor \frac{i}{2}\rfloor)}{4}}h\big(\lfloor \frac{i}{2}\rfloor\big)\Big\rfloor+\frac{1}{4}}{\lfloor h\big(\lfloor \frac{i}{2}\rfloor\big)\rfloor}\Bigg)\cdot
\frac{x}{|x|} \in D,\quad 2\le i \le  n-1.$$
Since $$
\frac{2}{3} \le \frac{\lfloor \frac{3}{4} a \rfloor}{\lfloor  a \rfloor} \le \frac{9}{11},  \quad
\frac{1}{6} \le \frac{\lfloor \frac{1}{4} a \rfloor}{\lfloor  a \rfloor} \le \frac{3}{11}  \quad \text{for all }a \ge 12,
$$
 using the fact $h\ge 12$, we see that $\frac{1}{5}\le |x^{(i-1)}-x^{(i)}|\le \frac{7}{8}$
for all $1\le i \le n$ and
$\delta_D(x^{(i)})\ge \frac{1}{4\lfloor h(\lfloor \frac{i}{2}\rfloor)\rfloor}$ for all $0\le i \le n-1$. Then, taking such
$\{x^{(i)}\}_{0\le i\le n}$ into \eqref{e:lower2-1}, we obtain that for all $x\in D$ with $|x|$ large enough
\begin{equation*}
\begin{split}
\phi_1(x)&\ge c_{10}\Phi\big(\delta_D(x)\big)\exp\Big(-c_{11}\Big(n\log n+\sum_{i=0}^{n-1} \log
\Big[ h\Big(\Big\lfloor \frac{i}{2}\Big\rfloor\Big)\Big]\Big)\Big)\\
&\ge   c_{12}\Phi\big(\delta_D(x)\big)\exp\Big(-c_{13}|x|\log |x|\Big),
\end{split}
\end{equation*}
where in the last inequality we used the following fact deduced from \eqref{ex3-1-2} and \eqref{e:Phi} that
$$\log \Big[h\big(\big\lfloor \frac{i}{2}\big\rfloor\big)\Big]\le
c_{14}\log \frac{1}{\Phi^{-1}((i+1)^{-\theta})}\le c_{15}\log (e+i),\quad 0\le i \le n.$$
Furthermore, according to the lower bound estimate for $\phi_1$ above and \eqref{e4-2}, we know
the intrinsic super Poincar\'e inequality \eqref{e:whole} holds with
\begin{equation*}
\begin{split}
\beta(s)&\le c_{16}\inf\left\{
\frac{ (\Phi^{-1}(s)\wedge r)^{-d}}{\inf_{x\in D_{R+1, r/2} }\phi_1^2(x)}:
 \inf_{x\in
D\setminus D_{R,r}}V_D(x) \ge 1/s\textrm{ and  }r \in (0,r_0/2)  \right\}\\
&\le c_{16}\inf\left\{
\frac{ (\Phi^{-1}(s)\wedge r)^{-d}e^{-c_{17}R\log R}}{\Phi^2(r)}\!:
 c_{18}\Phi\big(r\vee h(R)^{-1}\big)\!\le s\textrm{ and  }r \in (0,r_0/2)  \right\}.
\end{split}
\end{equation*}
Therefore,
in the infimum above choosing $r=c_{19}\Phi^{-1}(s)$ and $R=c_{20}s^{-{1}/{\theta}}$ for $s>0$ small enough and with some suitable positive constants $c_{19}, c_{20}>0$ (thanks to \eqref{e:Phi} and \eqref{ex3-1-2}),
we will arrive at
\begin{equation*}
\begin{split}
\beta(s)\le c_{21}\exp\big(c_{22}s^{-{1}/{\theta}}|\log s|\big),\quad 0<s\le s_0
\end{split}
\end{equation*}
for some $s_0\in (0,1)$. In particular, when $\theta>1$, \eqref{e:iurate} holds, and so $(P_t^D)_{t \ge 0}$ is intrinsically ultracontractive.

Assume \eqref{ex3-1-2} holds with $\theta\le 1$. Then, we can find a sequence $\{x_n\}_{n=1}^{\infty}\subset D$ such that
$n<|x_n|<n+1$ and $$\delta_D(x_n)\ge\frac{1}{4\lfloor h(n)\rfloor}\ge c_{23}\Phi^{-1}\big(n^{-\theta}\big).$$
This combined with \eqref{e3-4}
shows that \eqref{p3-1-2} does not hold, thanks to $\theta\le 1$. Hence, according to Proposition \ref{p3-1a}, we know that $(P_t^D)_{t \ge 0}$ is not intrinsically
ultracontractive.

Next we consider the situation that $\gamma\le 1$. In this case, we can directly apply \eqref{e:lower1-1} to derive that
\begin{equation}\label{ex3-1-3}
\begin{split}
& \phi_1(x)\ge c_{24}\Phi(\delta_D(x))\exp\Big(-c_{25}|x|^{\gamma}\Big)
\end{split}
\end{equation}
holds for all $x \in D$ with $|x|$ large enough.

Using \eqref{e4-2}, \eqref{ex3-1-2} and \eqref{ex3-1-3}, and following the same argument as above, we can obtain that the intrinsic super Poincar\'e inequality \eqref{e:whole} holds with
\begin{equation*}
\beta(s)\le c_{26}\exp\big(c_{27}s^{-{\gamma}/{\theta}}\big),\quad 0<s\le s_0.
\end{equation*}
If $\theta>\gamma$, then \eqref{e:iurate} holds, and so $(P_t^D)_{t \ge 0}$ is intrinsically ultracontractive.

If \eqref{ex3-1-2} holds with $\theta\le \gamma$, then, as the same procedure as above, we can show that
\begin{equation*}
\limsup_{x\in D\text{ and }|x| \rightarrow \infty}{\Phi(\delta_D(x))}{|x|^{\gamma}}>0.
\end{equation*}
So \eqref{p3-1-2} does not hold, and, by Proposition \ref{p3-1a}, $(P_t^D)_{t \ge 0}$ is not intrinsically
ultracontractive.
\end{proof}

\begin{remark}
We close this section with some comments on our approaches on  the compactness and the intrinsic ultracontractivity of $(P_t^D)_{t\ge0}$.
Throughout the arguments up to this section,  we make use of abstract assumptions like $\J_{\Phi,\ge}$, $\ext_{\Psi,\ge}$, $\ND_{\Psi,\ge}$ and $\OD_{\Gamma,\le}$.
Such assumptions have been used in \cite{CKW} to study heat kernel estimates for non-local Dirichlet forms on
general metric measure spaces. In fact,
the arguments above do not heavily depend on the characteristics of  Euclidean space.
We believe that our methods above can be used
to study the related topics  for non-local Dirichlet forms on general metric measure spaces.
\end{remark}

\section{Two-sided estimates for ground state on horn-shaped regions}\label{ss:tsehs}
This section, as a continuation of Subsection \ref{sec4444}, is devoted to establishing two-sided estimates for ground state on horn-shaped regions.
We concentrate on
the regular Dirichlet form $(\mathscr{E}, \mathscr{F})$ with jumping kernel $J(x,y)$ given by \eqref{e:jgeneral}.
 Let $D_f$ be a horn-shaped region with the reference function $f$.
   The associated Dirichlet semigroup is denote by $(P_t^{D_f})_{t\ge0}$.

In order to obtain explicit estimates for $\phi_1$,
we need additionally assumptions on the coefficient function $\kappa(x,y)$ and the scaling function $\Phi(r)$ of jumping kernel $J(x,y)$ and the reference function $f$ of horn-shaped region $D_f$. Recall that a function $g \in C^1(0,\infty)$ is a $C^{1,1}$ function, if there is a constant $c_g>0$ satisfying $\|g'\|_\infty\le c_g$ and $|g'(s)-g'(t)|\le c_g|s-t|$ for all $s,t>0$.
Throughout this section, the jumping kernel $J(x,y)$ is given by \eqref{e:jgeneral} and we further assume
\begin{itemize}\it
\item[(1)] Assumptions ${\bf(K_\eta)}$ and ${\bf (SD)}$ hold;

\item[(2)] The reference function $f$ of $D_f$ is a $C^{1,1}$ function;

\item[(3)] The semigroup $(P_t^{D_f})_{t\ge0}$ is intrinsically
ultracontractive.
\end{itemize}

Note that, $f\in C^{1,1}$ implies that
$D_f^c$ is a $\kappa$-fat set. Thus the semigroup $(P_t^{D_f})_{t\ge0}$ is compact.
Furthermore, by Propositions \ref{p:pDcont},
 \ref{p:pDp} and \ref{p5-1-0},
the ground state $\phi_1$ is bounded, continuous and strictly positive on $D_f$.
We also emphasize here that we do not assume neither $f$ is non-decreasing nor $\limsup_{r\to\infty} f'(r)=0.$ Such assumptions $f$ were used in \cite[(A1) and (A2) in p.\ 382]{LPZ} to study lower bound estimates of ground state of killed Brownian motion on horn-shaped region.

The following is the main result in this section. Recall that $f^*(r)=\sup_{s\ge r} f(s)$ and
$f_*(r)=\inf_{1\le s\le r} f(s)$.
\begin{theorem}\label{ex1-1h}
Under the setting and all the assumptions above, we have the following statements.
\begin{itemize}
\item[(1)]
If  $\gamma=0$ in \eqref{e:Exp}, then there are constants $c_1, c_2>0$ such that for all $x\in D_f$ with $|x|$ large enough,
\begin{equation*}\label{e:stah}\begin{split}
c_1\Phi(\delta_{D_f}(x))^{1/2}&\Phi(f_*(x_1+1))^{1/2}|x|^{-d}\Phi(|x|)^{-1}\\
&  \le \phi_1(x) \le c_2\Phi(\delta_{D_f}(x))^{1/2}\Phi(f^*(x_1-2))^{1/2}|x|^{-d}\Phi(|x|)^{-1}.\end{split}\end{equation*}

\item[(2)]
If $\gamma \in (0,1]$ in \eqref{e:Exp}, then there are constants $c_i>0$ $(i=1,\ldots,4)$ such that
for all $x\in D_f$ with $|x|$ large enough,
\begin{equation*}\label{ex1-1-2ah}\begin{split}
c_1\Phi(\delta_{D_f}(x))^{1/2}&\Phi(f_*(x_1+1))^{1/2}\exp\big(-c_2|x|^{\gamma}\big)\\
&
\le \phi_1(x) \le
c_3\Phi(\delta_{D_f}(x))^{1/2}\Phi(f^*(x_1-2))^{1/2}
\exp\big(-c_4|x|^{\gamma}\big).\end{split}\end{equation*}

\item[(3)]If $\gamma \in (1,\infty]$ in \eqref{e:Exp}, then,
for any increasing function $g(r)$ satisfying that
$c_0 \log^{1/\gamma} r  \le g(r) \le r/4$
for all $r>0$ large enough and some $c_0>0$,
there are constants $c_i>0$ $(i=1,\ldots,4)$ such that
for all $x\in D_f$ with $|x|$ large enough,
\begin{equation}\label{ex1-1-2ah111}\begin{split}
\phi_1(x)&\le
c_3\Phi(\delta_{D_f}(x))^{1/2}\Phi(f^*(x_1-2))^{1/2}\\
&\quad \times
\exp\Bigg[-c_4 \frac{|x|}{g(|x|)}\left(g(|x|)^\gamma {\wedge} \log \frac{1}{ f^*(|x|/4)}\right)\Bigg]\end{split}\end{equation} and \begin{equation}\label{ex1-1-2ah112}\begin{split}
\phi_1(x)\ge &c_1\Phi(\delta_{D_f}(x))^{1/2}\Phi(f_*(x_1+1))^{1/2}\\
&\times \exp\Bigg\{-c_2\bigg[|x|^{\gamma} \wedge \left( \frac{|x|}{g(|x|)}\left(g(|x|)^\gamma \vee \log \frac{1}{ f_*(2|x|)}\right)\bigg)\right]\Bigg\},\end{split}\end{equation}
\end{itemize}
\end{theorem}

Theorem \ref{ex1-1h} follows from Proposition \ref{p5-1} and Proposition \ref{p5-2} below, which are concerned with upper bound estimates and lower bound estimates of $\phi_1$, respectively. As an application of Theorem \ref{ex1-1h}, here we present the second part of proofs of Theorem \ref{ex1-1n1} and Theorem \ref{ex1-1n3} about two-sided estimates for $\phi_1$.

\begin{proof}
[Proofs of Theorem $\ref{ex1-1n1}(1)(b)$ and $\ref{ex1-1n1}(2)(b)$]
Theorem \ref{ex1-1n1}(1)$(b)$ and (2)$(b)$ with $\gamma\in (0,1]$ follow from  Theorem \ref{ex1-1h} (1) and (2), respectively. Concerning Theorem \ref{ex1-1n1} (2)$(b)$ with $\gamma\in (1,\infty]$, we take $g(r)=4\log^{{1}/{\gamma}}(e+r)$ in \eqref{ex1-1-2ah111} and \eqref{ex1-1-2ah112}.
Then, we can get the desired assertion.
Note that, in two-sided estimates for $\phi_1$ in the statement of Theorem \ref{ex1-1n1}(2)$(b)$, the factor $\Phi(f(|x|))^{1/{2}}$ can be absorbed into the exponential term, since for $x\in D_f$ such that $|x|$ large enough
\begin{align*}
c_1
\exp\Big(-c_2|x|\log^{{(\gamma-1)}/{\gamma}}|x|\Big)&\le \Phi(f(|x|))^{1/{2}}\exp\Big(-c_0|x|\log^{{(\gamma-1)}/{\gamma}}|x|\Big)\\
&\le c_3\exp\Big(-c_4|x|\log^{{(\gamma-1)}/{\gamma}}|x|\Big).
\end{align*}
 \end{proof}

\begin{proof}[Proof of Theorem $\ref{ex1-1n3}(b)$]
When $\theta\ge \gamma$,
taking $g(r)={r}/{4}$ in \eqref{ex1-1-2ah111} and \eqref{ex1-1-2ah112}, we can get \eqref{ex5-1-5}.
When $\theta \in (0,\gamma)$, we choose $g(r)=4r^{{\theta}/{\gamma}}$ in \eqref{ex1-1-2ah111} and \eqref{ex1-1-2ah112} to obtain \eqref{ex5-1-5}.
\end{proof}

\subsection{Upper bound estimates for ground state}

Let $U\subset \R^d$ be an open set and let $x\in \partial U$. We say that $U$ is $C^{1,1}$ near $x$ if there exist
a localization radius $r_x>0$, a constant $\lambda_x>0$,
a $C^{1,1}$-function
$\varphi_x:\R^{d-1}\to \R$ satisfying $\varphi_x(0)=0$, $\nabla \varphi_x(0)=(0,\dots, 0)$,
$\| \nabla \varphi_x \|_{\infty}\le \lambda_x$, $|\nabla \varphi_x(y)-\nabla \varphi_x(z)|\le \lambda_x |y-z|$, and an orthonormal coordinate system $CS_x$ with its origin at $x$ such that
$$
B(x,r_x)\cap U=\{y=(y_1, \tilde{y})\in B(0,r) \textrm{ in } CS_x:\, y_1>\varphi_x(\tilde{y})\}\, ,
$$
where $\tilde{y}:= (y_2, \dots, y_{d})$.
The pair $(r_x,\lambda_x)$ is called the $C^{1, 1}$ characteristics of $U$ at $x$.
An open set $U\subset \R^d$ is said to be a (uniform) $C^{1,1}$ open set with characteristics $(R,\Lambda)$,
if it is $C^{1,1}$ near every boundary point of $\partial U$ with the same characteristics $(R,\Lambda)$.

In order to obtain  upper bound estimates for ground state $\phi_1$, we first present the following upper bound estimate for the expectation of exit time. In what follows,
we denote by $D=D_f$ for simplicity.
\begin{lemma}\label{l5-1}
There exists a constant $c_1>0$ such that
for every
$x \in D$ with $|x|$ large enough,
\begin{equation}\label{l5-1-1}
\Ee_x\big[\tau_{B(z_x,1)\cap D}\big]\le c_1
{\Phi(\delta_D(x))}^{1/2}{\Phi(f^*(x_1-2))}^{1/2},
\end{equation}where $z_x \in \partial D$ such that
$|x-z_x|=\delta_D(x)$.
\end{lemma}
\begin{proof}
Throughout this proof we will consider $x\in D$ such that $f^*(x_1-2) \le 1/2$ and $\delta_D (x)\le 1/2$.
Suppose that $z_x=\big(z_{1},\tilde z\big)$ with $|\tilde z|=f(z_1)$. By a proper orthonormal transformation
of $\tilde x$ in $\R^{d-1}$, we can assume that
$z_x=\big(z_1,f(z_1),0,\dots, 0\big)$.
Let $y=\big(z_1,f(z_1)+\delta_D(x),\dots,0\big)$. Then, $y\in D^c$ and $\delta_{D^c}(y)\le |y-z_x|= \delta_D(x)$.

Suppose $y_x \in \partial D$ such that $|y-y_x|=\delta_{D^c}(y)$. Since $f$  is a $C^{1,1}$ function,
due to the uniform exterior ball condition (see e.g.\ \cite[p.\ 1309]{CKS}), there is a constant $0<\kappa<{1}/{2}$
(which is independent of $x \in D$) such that for every $x\in D$ with $|x|$ large enough,
$B(u_x,\kappa) \subset D^c$, where $u_x:=y_x+\kappa\cdot({y-y_x})/{|y-y_x|}$.
For every $z \in B(z_x,1)\cap D$ with $|x|$ large enough,
\begin{equation*}
\begin{split}
|z-u_x|\le |z-z_x|+|y-z_x|+|u_x-y|\le 1+\delta_D(x)+\kappa\le 2,
\end{split}
\end{equation*} and
so $ B(z_x,1)\cap D\subseteq B(u_x,2)\setminus B(u_x,\kappa)=:D_x$.
Clearly
$D_x$ is a $C^{1,1}$ domain with characteristics $(R_0,\Lambda)$, which are independent of $x$ when $|x|$ is large enough. According to \cite[Theorem 1.6]{GKK}, there exists a constant $c_1>0$ such that for every $x \in D$ with $|x|$ large enough
and $z \in D_x$
\begin{equation}\label{l5-1-2}
G_{D_x}(x,z)\le {c_1}
\frac{\Phi(|x-z|)}{|x-z|^{d}}\left(1\wedge \frac{\Phi(\delta_{D_x}(x))}{\Phi(|x-z|)}\right)^{1/2}\left(1\wedge \frac{\Phi(\delta_{D_x}(z))}{\Phi(|x-z|)}\right)^{1/2},
\end{equation}
where $G_{D_x}(\cdot,\cdot):D_x\times D_x \rightarrow [0,\infty)$ denotes the Green function of the process $X$ on $D_x$.
We want to emphasis that the constant $c_1$ above only depends on the
characteristics $(R_0,\Lambda)$,
and it is independent of $x$.

Using $B(z_x,1)\cap D\subseteq D_x$ and applying \eqref{l5-1-2}, we find that \begin{equation}\label{l5-1-3}
\begin{split}
\Ee_x\big[\tau_{B(z_x,1)\cap D}\big]
&=\int_{B(z_x,1)\cap D}G_{B(z_x,1)\cap D}(x,z)\,dz\\
&\le \int_{B(z_x,1)\cap D}G_{D_x}(x,z)\,dz\\
&\le c_1
\int_{B(z_x,1)\cap D}
\frac{\Phi(|x-z|)}{|x-z|^{d}}\left(1\wedge \frac{\Phi(\delta_{D_x}(x))}{\Phi(|x-z|)}\right)^{1/2}\,dz\\
&\le c_1\Phi(\delta_{D_x}(x))^{1/2}\int_{B(z_x,1)\cap D}\frac{\Phi(|x-z|)^{1/2}}{|x-z|^{d}}\,dz.
\end{split}
\end{equation}
Since $y_x\in \partial B(u_x,\kappa)\subseteq \partial D_x$, by \eqref{e:Phi} we have
\begin{equation*}
\begin{split}
\Phi(\delta_{D_x}(x))\le\Phi( |x-z_x|+|z_x-y|+|y-y_x|)\le \Phi(3\delta_D(x)) \le c_2 \Phi(\delta_D(x)) .
\end{split}
\end{equation*}
On the other hand, since $\delta_D (x)\le 1/2$, we have $B(z_x,1)\cap D\subseteq U_x^1 \cup U_x^2$,
where $U_x^1=B(x,f^*(x_1-2))$ and $U_x^2=\{z \in \R^d: f^*(x_1-2)\le |z-x| \le 2\}\cap D$. Thus,
\begin{equation*}
\begin{split}
\int_{B(z_x,1)\cap D}\frac{\Phi(|x-z|)^{1/2}}{|x-z|^{d}}\,dz&
\le \int_{U_x^1}\frac{\Phi(|x-z|)^{1/2}}{|x-z|^{d}}\,dz
+\int_{U_x^2}\frac{\Phi(|x-z|)^{1/2}}{|x-z|^{d}}\,dz=:I_1+I_2.
\end{split}
\end{equation*}
Using \eqref{e:Phi},   we have  \begin{align*}
I_1&\le  c_3 \int_0^{f^*(x_1-2)}\frac{\Phi(r)^{1/2}}{r}\,dr\\
&=
c_3\Phi(f^*(x_1-2))^{{1}/{2}}    \int_0^{f^*(x_1-2)}\frac{\Phi(r)^{1/2}}{\Phi(f^*(x_1-2))^{{1}/{2}}r}\,{dr} \\
&   \le c_4\Phi(f^*(x_1-2))^{{1}/{2}}    \int_0^{f^*(x_1-2)}\frac{r^{\la/2-1}}{f^*(x_1-2)^{\la/2}}\,{dr}  \\
& \le  c_5 \Phi(f^*(x_1-2))^{{1}/{2}}.
\end{align*}
Moreover,  note that the Lebesgue surface measure of $U_x^2 \cap \{z \in \R^d:|z-x|=r\}$ is not bigger than
$$ c_6\max_{|s-x_1| \le 2}f(s)^{d-1} \le  c_6f^* (x_1-2)^{d-1}$$
 for any $r \le 2$. Using  \eqref{l5-1-3} and the fact that $1+{\ua}/{2}<2\le d$, we get
 \begin{align*}
I_2
&\le
c_7f^* (x_1-2)^{d-1} \int_{f^*(x_1-2)}^2     \frac{ \Phi(r)^{1/2}}{r^d} \,dr \\
&=  c_8\Phi(f^*(x_1-2))^{1/2} f^* (x_1-2)^{d-1}   \int_{f^*(x_1-2)}^2      \frac{ \Phi(r)^{1/2} } {\Phi(f^*(x_1-2))^{1/2}r^d }
\,dr \\
& \le
c_8  \Phi(f^*(x_1-2))^{1/2} f^* (x_1-2)^{d-1}  \int_{f^*(x_1-2)}^2      \frac{r^{\ua/2-d} } {f^*(x_1-2)^{\ua/2} }
\,dr \\
&\le
c_9   \Phi(f^*(x_1-2))^{1/2}.
\end{align*}

Combining all the estimates with \eqref{l5-1-3} yields the desired conclusion \eqref{l5-1-1}.
\end{proof}

The following two lemmas are needed, see \cite[Lemma 1.10]{CKS} or \cite[Lemma 5.1]{GKK} for
the first one, and  \cite[Corollary 2.4]{GKK} for the second one.

\begin{lemma}\label{l5-2}
Let $U_1$ and $U_3$ be two open subsets of
an open set $U\subset \R^d$ such that $dist(U_1,U_3)>0$, and let $U_2=
U\setminus (U_1\cup U_3)$. Then, for
every $t>0$, $x \in U_1$ and $y \in U_3$,
\begin{equation}\label{l5-2-1}
\begin{split}
p^U(t,x,y)\le &\Pp^x\big(X_{\tau_{U_1}}\in U_2\big)
\sup_{0<s<t,z\in U_2}p^U(s,z,y)\\
&+\big(t \wedge \Ee^x[\tau_{U_1}]\big)
\sup_{u \in U_1,z\in U_3}J(u,z).
\end{split}
\end{equation}
\end{lemma}

\begin{lemma}\label{l5-3}
There exists a constant $c>0$ such that for every open set $U$, $r \in (0,1]$ and $y \in \R^d$,
\begin{equation}\label{l5-3-1}
\Pp^x \big(X_{\tau_{B(y,r)}\cap
U} \in
U\big)\le  \frac{c\, \Ee^x[\tau_{B(y,r)\cap U}]}{\Phi(r)},\quad x \in
B(y,{r}/{2})\cap U.
\end{equation}
\end{lemma}

Now, we can give upper bound estimates for the ground state $\phi_1$.

\begin{proposition}\label{p5-1}The following three statements hold.

\begin{itemize}
\item[(1)] Suppose  that $\gamma=0$ in \eqref{e:Exp}. Then there exists
a constant $c_1>0$ such that for every $x \in D$ with $|x|$ large enough,
\begin{equation*}\label{p5-1-2}
\phi_1(x)\le c_1 \Phi(\delta_{D}(x))^{1/2}\Phi(f^*(x_1-2))^{1/2}|x|^{-d}\Phi(|x|)^{-1}.
\end{equation*}

\item[(2)]
Suppose that $\gamma \in (0,1]$ in \eqref{e:Exp}. Then there exist constants
$c_i>0$ $(i=1,2)$ such that for every $x \in D$ with $|x|$ large enough,
\begin{equation}\label{p5-1-2a}
\phi_1(x)\le c_1 \Phi(\delta_{D}(x))^{1/2}\Phi(f^*(x_1-2))^{1/2}
\exp\Big(-c_2|x|^{\gamma}\Big).
\end{equation}

\item[(3)] Suppose that $\gamma \in (1,\infty]$ in \eqref{e:Exp}.
Then,
for any
increasing function  $g(r)$  satisfying that $4\le r/g(r) \le {r}/{4}$ for $r>0$ large enough,
there exist constants
$c_i>0$ $(i=1,2)$ such that for every $x \in D$ with $|x|$ large enough,
\begin{equation}\label{p5-1-1}
\begin{split}
\phi_1(x)&\le
c_1\Phi(\delta_{D}(x))^{1/2}\Phi(f^*(x_1-2))^{1/2}\\
&\quad \times \exp\Bigg\{-\frac{c_2|x|}{g(|x|)}\Big[g(|x|)^{\gamma}\wedge\log\Big(1+\frac1{f^*({|x|}/{4})}\Big)\Big]\Bigg\}.
\end{split}\end{equation}
In particular, for $\gamma=\infty$, there exist constants
$c_i>0$ $(i=3,4)$ such that for every $x \in D$ with $|x|$ large enough,
\begin{equation}\label{p5-1-11}
\begin{split}
\phi_1(x)&\le
c_3\Phi(\delta_{D}(x))^{1/2}\Phi(f^*(x_1-2))^{1/2} \exp\Big[-c_4{|x|}\log\Big(1+\frac{1}{f^*({|x|}/{4})}\Big)\Big].
\end{split}\end{equation}
\end{itemize}
\end{proposition}
\begin{proof}

Note that the intrinsic ultracontractivity of $(P_t^D)_{t \ge 0}$ implies that
\begin{equation*}
\phi_1(x)\le \frac{1}{c\phi_1(y_0)}p^D(1,x,y_0)\le c' p^D(1,x,y_0),\quad x \in D,
\end{equation*}
where $y_0=(1,\tilde 0)$ and $c'>0$ is a constant independent of $x\in D$.
Therefore, in order to get the desired assertions it suffices to consider upper bound estimates of
$p^D(1,x,y_0)$. Throughout the proof, we consider
$x \in D$ with $|x|$ large enough  such that $f^*(|x|/3)\vee f^*(x_1-2) \vee \delta_D (x) \le 1/2$. Let $z_x\in \partial D$ such that $|z_x-x|=\delta_D(x)$, and set $U_x=B(z_x,1)\cap D$.

\smallskip

\noindent
(1) By \cite[Theorem 1.2]{CK1}, it holds that for $t>0$ and $x,y\in D$,
\begin{align}
\label{e:gamma0ub}
p^D(t,x,y)\le p(t,x,y)\le c_0\Big((\Phi^{-1}(t))^{-d} \wedge \frac{t}{|x-y|^{d} \Phi(|x-y|)}\Big).
\end{align}
Recall $U_x=B(z_x,1)\cap D$. Let $U_1=U_x$,
$U_2=\{z \in D: |z-x| \le {|x-y_0|}/{2}\} \setminus U_1$ and $U_3=D\setminus (U_1\cup U_2)$
 in \eqref{l5-2-1}. Then, by \eqref{e:gamma0ub}, we have
\begin{equation}
\label{e:gamma0ub1}\begin{split}
\sup_{0<s<1,z\in U_2}p(s,z,y_0)\le &\sup_{0<s<1,|z-y_0|\ge |x-y_0|-|z-x|\ge
{|x-y_0|}/{2}}p(s,z,y_0)\\
\le & c_1|x|^{-d}\Phi(|x|)^{-1}
\end{split} \end{equation} and
\begin{equation}\begin{split}
\label{e:gamma0ub2}
\sup_{u \in U_1,z\in U_3}J(u,z)\le & \sup_{|u-z|\ge |x-z|-|x-u|\ge {|x-y_0|}/{2}-2\ge
{|x-y_0|}/{3}}J(u,z)\\
\le & c_1|x|^{-d}\Phi(|x|)^{-1}.
\end{split}\end{equation}
 Thus
\begin{equation*}\label{p5-1-3a}
\begin{split}
p^D(1,x,y_0)
&\le \Pp^x\big(X_{\tau_{U_1}}\in U_2\big)
\sup_{0<s<1,z\in U_2}p^D(s,z,y_0)\\
&\quad+\big(1 \wedge \Ee^x[\tau_{U_1}]\big)
\sup_{u \in U_1,z\in U_3}J(u,z)\\
&\le c_2\Ee^x[\tau_{B(z_x,1)\cap D}]\Big(\sup_{0<s<1,z\in U_2}p(s,z,y_0)+
\sup_{u \in U_1,z\in U_3}J(u,z)\Big)\\
&\le c_3\Phi(\delta_{D}(x))^{1/2}\Phi(f^*(x_1-2))^{1/2}|x|^{-d}\Phi(|x|)^{-1},
\end{split}
\end{equation*}
where the second inequality follows from \eqref{l5-3-1} and in the
third inequality we have used \eqref{l5-1-1}, \eqref{e:gamma0ub1} and \eqref{e:gamma0ub2}.

\smallskip

\noindent
(2) Suppose $\gamma \in (0,1]$. Then, by \cite[Theorem 1.2]{CKK1}, for any $t \in (0,1]$ and $x,y\in D$ with $|x-y|\ge 1$,
\begin{equation*}
p^D(t,x,y)\le p(t,x,y)\le c_1\exp\Big(-c_2|x-y|^{\gamma}\Big).
\end{equation*}
Using this instead of \eqref{e:gamma0ub} and following the same argument as (1), One can get \eqref{p5-1-2a} immediately.

\smallskip
\noindent
(3) Suppose $\gamma \in (1,\infty]$. In this part we need much more delicate arguments instead of applying \eqref{l5-2-1} directly. The proof is a little long, and it is split into four steps.

(i)
 For any $x\in D$ with $|x|$ large enough such that $g(|x|)\wedge (|x|/g(|x|))\ge 4$, define
\begin{equation*}
\begin{split}
&V_k=\big\{z \in D: |z-x|\le k g(|x|)\big\},\\
&h_k=\sup_{0<s\le 1,z\in V_k}p^D(s,z,y_0),\quad
1\le k \le N=\bigg \lfloor\frac{|x|}{4g(|x|)}\bigg\rfloor,\\
&h_{N+1}\equiv 1.
\end{split}
\end{equation*}
Recall that  $U_x=B(z_x,1)\cap D$.
By the strong Markov property, we have
\begin{equation} \label{e:I_20}
\begin{split}
p^D(1,x,y_0)&=\Ee^x\Big[
p^D(1-\tau_{U_x},X_{\tau_{U_x}},y_0)\I_{\{\tau_{U_x}<1\}}\Big]\\
&=\Ee^x\Big[
p^D(1-\tau_{U_x},X_{\tau_{U_x}},y_0)\I_{\{\tau_{U_x}<1\}}\I_{\{X_{\tau_{U_x}}\in V_1\setminus U_x\}}\Big]\\
&\quad +\Ee^x\Big[
p^D(1-\tau_{U_x},X_{\tau_{U_x}},y_0)\I_{\{\tau_{U_x}<1\}}\I_{\{X_{\tau_{U_x}}\in D\setminus V_1\}}\Big]\\
&=:I_1+I_2.
\end{split}
\end{equation}

Note that,  by \eqref{l5-3-1},
\begin{equation}\label{e:I_200}
I_1\le \Pp^x\big(X_{\tau_{U_x}}\in V_1\setminus U_x\big) \sup_{0<s\le 1,z\in V_1}p^D(s,z,y_0)\le c_1\Ee^x[\tau_{U_x}] h_1.
\end{equation}

On the other hand, since $D\setminus V_1=\Big(\sum_{k=2}^N V_k\setminus V_{k-1}\Big)\bigcup V_{N+1}$, where
$V_{N+1}:=D\setminus V_N$, by the L\'evy system \eqref{e:levy} of the process $X$,
\begin{equation}\label{e:I_21}\begin{split}
I_2=&\int_0^1\int_{U_x}\int_{D\setminus V_1}
p^{U_x}(s,x,y)J(y,z)p^D(1-s,z,y_0)\,dz\,dy\,ds\\
=&\sum_{k=2}^N\int_0^1\int_{U_x}\int_{V_k\setminus V_{k-1}}
p^{U_x}(s,x,y)J(y,z)p^D(1-s,z,y_0)\,dz\,dy\,ds\\
&+\int_0^1\int_{U_x}\int_{V_{N+1}}
p^{U_x}(s,x,y)J(y,z)p^D(1-s,z,y_0)\,dz\,dy\,ds\\
\le& \sum_{k=2}^N\Bigg[\bigg(\int_0^1\int_{U_x}p^{U_x}(s,x,y)\,dy\,ds\bigg) |V_k\setminus V_{k-1}|
\sup_{y\in U_x, z\in V_k\setminus V_{k-1}}J(y,z)\\
&\qquad\times \sup_{0<s\le 1,z\in V_k\setminus V_{k-1}}
p^D(s,z,y_0)\Bigg]\\
&+ \Bigg[\int_0^1\bigg(\int_{U_x}p^{U_x}(s,x,y)\,dy\bigg)\bigg(\int_{D}p^D(1-s,z,y_0)\,dz\bigg)\,ds\Bigg] \\
&\quad\times\sup_{y \in U_x,z\in V_{N+1}}J(y,z)\\
\le&\Ee^x[\tau_{U_x}] \sum_{k=2}^N\left(
\sup_{y\in U_x, z\in V_k\setminus V_{k-1}}J(y,z) \right) |V_k\setminus V_{k-1}| h_k\\
&+ \Ee^x[\tau_{U_x}] \sup_{y \in U_x,z\in V_{N+1}}J(y,z).
\end{split}\end{equation}
For any $z \in V_N$ and any $x\in D$ with $|x|$ large enough,
 \begin{align*}
z_1 &\ge x_1-Ng(|x|) \ge x_1- ({|x|}/({4g(|x|)}) +1)g(|x|)\\
&=x_1-{|x|}/{4} -g(|x|) \ge x_1 -{|x|}/{2}  \ge {|x|}/{3},
\end{align*}
and so, by the definition of horn-shaped region,
 \begin{align*}
 \label{e:I_22}
\max_{ 2\le k \le N} |V_k\setminus V_{k-1}|
  \le c_1f^*(|x|/3)^{d-1} g(|x|)\le c_1 2^{1-d} g(|x|).
 \end{align*}
Since
for every $ 2\le k \le N$,
$y\in U_x$ and $z\in D \setminus V_{k}$,
$$
|y-z| \ge |x-z|-|x-y| \ge kg(|x|) -2 \ge (k-1)g(|x|),
$$ which, along with the assumption that $\gamma>1$, implies that for every $ 2\le k \le N$
\begin{equation*}\label{e:I_23}\begin{split}
  \sup_{y\in U_x, z\in V_k\setminus V_{k-1}}J(y,z)&\le
c_2 \sup\Big\{e^{-C_0|y-z|^{\gamma}}:|y-z|\ge (k-2)g(|x|)\Big\}\\&\le
c_2 \exp\Big(-C_0(k-2)g(|x|)^{\gamma}\Big)
\end{split}\end{equation*}
and
\begin{equation*}
 \label{e:I_24}\begin{split}
\sup_{y\in U_x, z\in V_{N+1}}J(y,z)&\le
c_2 \sup\Big\{e^{-C_0|y-z|^{\gamma}}:|y-z| \ge (N-1)g(|x|)\Big\}\\&\le
c_2  \exp\Big(-C_0(N-1)g(|x|)^{\gamma}\Big).
\end{split}\end{equation*}

Applying all the estimates above into \eqref{e:I_21}, we get that
$$I_2\le c_3\Ee^x[\tau_{U_x}]  \sum_{k=2}^{N+1}\exp\Big(-C_1(k-1)g(|x|)^{\gamma}\Big) h_k,$$ which together with \eqref{e:I_20} and \eqref{e:I_200}
in turn
yields that
\begin{equation}\label{p5-1-3}\begin{split}
&p^D(1,x,y_0)\\
&\le c_4\Ee^x[\tau_{U_x}]  \sum_{k=1}^{N+1}\exp\Big(-C_1(k-1)g(|x|)^{\gamma}\Big) h_k \\
&\le c_5\Phi(\delta_{D}(x))^{1/2}\Phi(f^*(x_1-2))^{1/2}\sum_{k=1}^{N+1}\exp\Big(-C_1(k-1)g(|x|)^{\gamma}\Big) h_k,
\end{split}\end{equation} where we have used \eqref{l5-1-1} in the last inequality above.

(ii) For any $1\le k \le N$, let $x_1^{(k)}=x_1-2-kg(|x|)$.
For any fixed $u \in V_k$, let $z_u$ be a point on $\partial D$ satisfying $|u-z_u|=\delta_D(u) \le 1/2$, and let $U_u=B(z_u,1)\cap D$.
We first observe that $U_u \subset V_{k+1}$, and that
for every $ k+1 \le l \le N$,
$y\in U_u$ and $z\in D \setminus V_{l}$
\begin{align*}
|y-z| &\ge |x-z|-|x-u|-\delta_D(u)-|z_u-y| \\
&\ge (l-k) g(|x|) -2 \ge 2^{-1}(l-k)g(|x|).
\end{align*}
Hence, for $k+2 \le l \le N$,
\begin{equation*}
 \label{e:I_23k}\begin{split}
 \sup_{y\in U_u, z\in V_l\setminus V_{l-1}}J(y,z)&\le
c'_2 \sup\Big\{e^{-C_0|y-z|^{\gamma}}:|y-z|\ge 2^{-1}(l-k-1)g(|x|)\Big\}\\&\le
c'_2 \exp\Big(-C_2(l-k-1)g(|x|)^{\gamma}\Big)
\end{split}\end{equation*}
and
\begin{equation*}
 \label{e:I_24k}\begin{split}
\sup_{y\in U_u, z\in V_{N+1}}J(y,z)&\le
c'_2 \sup\{e^{-C_0|y-z|^{\gamma}}:|y-z| \ge2^{-1} (N-k)g(|x|)\}\\&\le
c'_2  \exp\Big(-C_2(N-k)g(|x|)^{\gamma}\Big).
\end{split}\end{equation*}
Thus, following the same arguments to derive \eqref{p5-1-3},  we have that for any $t \in (0,1]$ and $u \in V_k$,
\begin{align*}
&p^D(t,u,y_0)\\
&\le  \Pp^u(X_{\tau_{U_u}}\in V_{k+1}) \sup_{0<s\le 1,z\in V_{k+1}}p^D(s,z,y_0)\\
&\quad +\sum_{l=k+2}^N\int_0^t\int_{U_u}\int_{V_l \setminus V_{l-1}}
p^{U_u}(s,u,y)J(y,z)p^D(t-s,z,y_0)\,dz\,dy\,ds\\
&\quad +\int_0^t\int_{U_u}\int_{V_{N+1}}
p^{U_u}(s,u,y)J(y,z)p^D(t-s,z,y_0)\,dz\,dy\,ds\\
&\le
c_6  \Ee^u[\tau_{U_u}] \Bigg[h_{k+1}+\! \sum_{l=k+2}^N\left(
\sup_{y\in U_u, z\in V_l\setminus V_{l-1}}\!J(y,z)\! \right) \!|V_l \setminus V_{l-1}| h_l+\! \left(\sup_{y \in U_u,z\in V_{N+1}}\!J(y,z) \!\right)\!\! \Bigg]\\
&\le
c_7 \Ee^u[\tau_{U_u}]\sum_{l=k+1}^{N+1} \exp\Big(-C_3(l-(k+1))g(|x|)^{\gamma}\Big) h_l \\
&\le c_8 \Phi(f^*(x_1^{(k)}))\sum_{l=k+1}^{N+1} \exp\Big(-C_3(l-(k+1))g(|x|)^{\gamma}\Big) h_l.
\end{align*}
Here in the last inequality above
we have used \eqref{l5-1-1}, \eqref{e:Phi} and the facts that $\delta_D(u)\le f(u_1) \le f^*(u_1)$ and
$
u_1 \ge x_1-kg(|x|)$ to get that \begin{align*}
\Ee^u[\tau_{U_u}] &\le c'_6 \Phi(\delta_{D}(u))^{1/2}\Phi(f^*(u_1-2))^{1/2}\\
&\le c'_6 \Phi(f^*(u_1))^{1/2}\Phi(f^*(u_1-2))^{1/2}\le c'_7 \Phi(f^*(x_1^{(k)})).
\end{align*}
Therefore, for $1\le k  \le N$,
\begin{equation}\label{p5-1-4}
\begin{split}
h_k&=\sup_{0<s\le 1, z \in V_k}p(s,z,y_0)\\
&\le c_8 \Phi(f^*(x_1^{(k)}))\sum_{l=k+1}^{N+1} \exp\big(-C_3(l-(k+1))g(|x|)^{\gamma}\big) h_l.
\end{split}
\end{equation}

(iii)
Below, we will apply \eqref{p5-1-4} into \eqref{p5-1-3}  to estimate the remaining
$h_k$ term for $2\le k\le N-1$. Note that if we continue the procedure for $l$ times, then it only remains $h_i$ with index $l+1 \le i \le N$.
For simplicity, we relabel and use again constants $C, C_i>0$ $(i\ge1)$ in the argument below without confusion.
Thus
\begin{align*}
\sum_{k=1}^{N+1}&\exp\big(-C(k-1)g(|x|)^{\gamma}\big) h_k \\
 \le &  C_1 \Phi(f^*(x_1^{(1)})) \sum_{l=2}^{N+1} \exp\big(-C(l-2)g(|x|)^{\gamma}\big) h_l  +
 \sum_{k=2}^{N+1}\exp\big(-C(k-1)g(|x|)^{\gamma}\big) h_l \\
 =
 &  C_1 \Phi(f^*(x_1^{(1)}))h_2 + C_1 \Phi(f^*(x_1^{(1)})) \sum_{l=3}^{N+1} \exp\big(-C(l-2)g(|x|)^{\gamma}\big) h_l \\
 & +
\exp\big(-Cg(|x|)^{\gamma}\big) h_2  +
 \sum_{k=3}^{N+1}\exp\big(-C(k-1)g(|x|)^{\gamma}\big) h_k \\
  \le
 &  C_1^2 \Phi(f^*(x_1^{(1)})) \Phi(f^*(x_1^{(2)}))\sum_{l=3}^{N+1} \exp\big(-C(l-3)g(|x|)^{\gamma}\big) h_l \\
 & + C_1 \Phi(f^*(x_1^{(1)})) \sum_{l=3}^{N+1} \exp\big(-C(l-2)g(|x|)^{\gamma}\big) h_l \\
 & +
C_1 \Phi(f^*(x_1^{(2)}))\exp\big(-Cg(|x|)^{\gamma}\big)  \sum_{l=3}^{N+1} \exp\big(-C(l-3)g(|x|)^{\gamma}\big) h_l \\
 & +
 \sum_{k=3}^{N+1}\exp\big(-C(k-1)g(|x|)^{\gamma}\big) h_k \\
\le
 &  C_1^3 \Phi(f^*(x_1^{(1)})) \Phi(f^*(x_1^{(2)})) \Phi(f^*(x_1^{(3)}))  h_4\\
 &+ C_1^2 \Phi(f^*(x_1^{(1)}))  \Phi(f^*(x_1^{(2)})) \exp\big(-  Cg(|x|)^{\gamma}\big) h_4 \\
 & +
C_1^2 \Phi(f^*(x_1^{(2)}))  \Phi(f^*(x_1^{(3)}))  \exp\big(-C g(|x|)^{\gamma}\big) h_4 \\
&+  C_1^2 \Phi(f^*(x_1^{(1)})) {\Phi(f^*(x_1^{(3)}))}\exp\big(-  Cg(|x|)^{\gamma}\big)  h_4 \\
 & + C_1 \Phi(f^*(x_1^{(1)}))  \exp\big(-2Cg(|x|)^{\gamma}\big) h_4 \\
 &+
C_1 \Phi(f^*(x_1^{(2)}))\exp\big(-2Cg(|x|)^{\gamma}\big)   h_4 \\
 & +
C_1 \Phi(f^*(x_1^{(3)}))  \exp\big(-2Cg(|x|)^{\gamma}\big)  h_4 \\
& +\exp\big(-3Cg(|x|)^{\gamma}\big) h_4 \\
&+ \cdots.
\end{align*}
In the argument above and below we assume that $C_1>1$.

Note also that, by \cite[(1.16) in Theorem 1.2]{CKK1},  \begin{align*} h_{N} &=\sup_{0<s\le 1, z \in V_{N}}p(s,z,y_0)
\\&\le \sup_{0<s<1,|z-y_0|\ge |x-y_0|-|z-x|\ge
{|x-y_0|}/{2}}p(s,z,y_0)\le c_9 \exp(-c_{10}|x|^\gamma) \le c_9.
\end{align*}
and
$h_{N+1} \equiv 1$.

Therefore, repeating the procedure $N-1$ times and using the notational convention that $\prod_{i=1}^0\Phi(f^*(x^{(k_i)}_1) =1$,  we can conclude from \eqref{p5-1-3} that
\begin{align*}
&p^D(1,x,y_0)\\
&\le C_2
\Phi(\delta_{D}(x))^{1/2}\Phi(f^*(x_1-2))^{1/2}\\
&\quad \times
\sum_{l=0}^{N} C_1^l\left(\sum_{1\le k_1<k_2\dots <k_l\le N }\prod_{i=1}^l  \Phi(f^*(x^{(k_i)}_1))\right)\exp\big(-C_3(N-l)g(|x|)^{\gamma}\big).
\end{align*}

(iv) Since $f^*$ is non-increasing, we have for every $1\le k_1<k_2\dots <k_l\le N-1$,
\begin{equation*}
\sum_{i=1}^l \log\left(1+\frac{1}{f^*(x^{(k_i)}_1)}\right)
\ge \sum_{i=N-l+1}^N  \log\left(1+\frac{1}{f^*(x^{(i)}_1)}\right)
\ge c_{11}l\log \left(1+\frac{1}{f^*({|x|}/{4})}\right),
\end{equation*}
where the last inequality follows from the fact that $x_1^{(k)}\ge {|x|}/{4}$ for every $1\le k \le N$.
The inequality above finally leads to the following estimate
\begin{equation*}
\begin{split}
p^D(1,x,y_0)
&\le C_4C_1^N
\Phi(\delta_{D}(x))^{1/2}\Phi(f^*(x_1-2))^{1/2}\\
&\quad\times
\sup_{0\le l \le N}\exp\Big\{-C_{5}\Big[l\log\Big(1+\frac{1}{f^*({|x|}/{4})}\Big)+(N-l)g(|x|)^{\gamma}\Big]\Big\}\\
&\le C_6\Phi(\delta_{D}(x))^{1/2}\Phi(f^*(x_1-2))^{1/2} \\
&\quad  \times \exp\Bigg\{-\frac{C_7|x|}{g(|x|)}\Big[g(|x|)^{\gamma} \wedge\log\Big(e+\frac{1}{f^*({|x|}/{4})}\Big)\Big]\Bigg\},
\end{split}
\end{equation*}
where in the first inequality we used \eqref{e:Phi}, and in the last inequality we have used the fact that $N=\left\lfloor \frac{|x|}{4g(|x|)}\right \rfloor$. Therefore, we have now obtained \eqref{p5-1-1}.
In particular, when $\gamma=\infty$,
\begin{equation*}
\begin{split}
\phi_1(x)&\le
c'\Phi(\delta_{D}(x))^{1/2}\Phi(f^*(x_1-2))^{1/2} \exp\Bigg(-\frac{c''|x|}{g(|x|)}\log\Big(1+\frac{1}{f^*({|x|}/{4})}\Big)\Bigg).
\end{split}\end{equation*}
Thus, by taking $g(r)=4$ for $r>0$, we obtain \eqref{p5-1-11}.
\end{proof}

\subsection{Lower bound estimates for ground state}
In this subsection, we turn to lower bound estimates of ground state $\phi_1$.  We begin with the following lower bound estimate for the
survival probability.
Recall that $D=D_f$ is a horn-shaped region.
\begin{lemma}\label{l5-4} There exists a constant $c>0$ such that for every $x\in D$ with $|x|$ large enough
\begin{equation}\label{l5-4-1}
\Pp^x\big(\tau_{B(z_x,1)\cap D}>\Phi(f_*(x_1+1))\big)\ge
\frac{c\Phi(\delta_{D}(x))^{1/2}}{\Phi(f_*(x_1+1))^{1/2}},
\end{equation}where $z_x \in \partial D$ such that
$|x-z_x|=\delta_D(x)$.
\end{lemma}
\begin{proof}
We assume that $x\in D$ with $|x|$ large enough such that
$|x|\ge 1$ and $f(x_1) \le 1/4$.
Since $f$ is a $C^{1,1}$ function, due to the local interior ball
condition (see \cite[p.\ 1039]{CKS}), there exist constants $\kappa \in (0,1)$ and
 $c_0>0$ such that, for all $x\in D$ with $|x|$ large enough one can choose a ball $U:=B\big(\xi_x,\kappa f_*(x_1+1)\big)$
satisfying that
$x\in U \subset  B(z_x,1) \cap D$ and $c_0\delta_D(x)\le \delta_U(x)\le \delta_D(x)$.

Let $X^0:=(X^0_t)_{\ge0}$  (on the same probability space) be a symmetric jump process on $\R^d$, whose jumping kernel
$J^0(x,y)$ is given by \eqref{e:jgeneral} with  $\gamma =0$,
and $\kappa$ and $\Phi$ are same as those for jumping kernel $J(x,y)$.
One can regard $X^0$ as the process obtained  by  the Meyer's construction (see \cite[Remark
3.5]{BBCK})
through increasing the intensity of
jumps for the process $X$ larger than $1$,
so that
$X_t=X_t^0$ for any $t\in(0, N_1^0)$, where
$$
N^0_1:=\inf\big\{t\ge0: |X^0_t-X^0_{t-}|>1\big\}.$$ In the following, for any subsect $A\subset \R^d$, let
$$\tau^0_{A}=\inf\big\{t>0:X_t^0\notin A \big\}$$
 be the
first exit time from $A$ of the process $X^0$. Note
that, under $\Pp^x$ the event $\{X_t^0 \in U\textrm{
for any } t\in [0,\Phi(f_*(x_1+1))]\}$  implies that the
process $(X_t^0)_{t\ge0}$ does not have any jump bigger than $1$ in time interval
$[0,\Phi(f_*(x_1+1))]$.
That is,  under $\Pp^x$,
$$\{\tau^0_{U}>\Phi(f_*(x_1+1))\}\subseteq \{\tau_{U}>\Phi(f_*(x_1+1))\}.$$ Therefore, for any $x\in D$ with $|x|$ large enough,
\begin{equation}\label{l5-4-1n}\begin{split}
\Pp^x\big(\tau_{B(z_x,1)\cap D}>\Phi(f_*(x_1+1))\big)\ge &\Pp^x\big(\tau_{U}>\Phi(f_*(x_1+1))\big)\\
 \ge &
 \Pp^x\big(\tau^0_{U}>\Phi(f_*(x_1+1))\big).\end{split}
\end{equation}

Next, we choose
an orthonormal coordinate system $(CS)$
with origin at $z_x$ and, for $\delta >0$, let $X^{(\delta)}:=
\{ \delta^{-1}X^0_{\Phi(\delta) t}: t \ge 0\}$ be the
scaled process of $X^0$. Define
$$\kappa_{(\delta)}(  x,  y) =\kappa( \delta x, \delta y)\quad  \text{and}   \quad   \Phi_{(\delta)}(r)=\Phi(\delta r)/\Phi(\delta). $$
Then
the jumping kernel
$J^{(\delta)}(x,y)$
of the process $X^{(\delta)}$ with respect to the
Lebesgue measure
is related to that of $X^0$ by the following formula
$$J^{(\delta)}(x, y)= \delta^d \Phi(\delta) J^0( \delta x, \delta y)=
 \frac{ \kappa_{(\delta)}( x, y) }{\Phi_{(\delta)}( |x-y|)|x-y|^{d}}.$$
Clearly, by {\bf(SD)},  $\Phi_{(\delta)} \in C^1(0,\infty) $ and
$$ r\mapsto -((\Phi_{(\delta)}(r)^{-1}r^{-d}))'/r=-\Phi(\delta)\delta^d (\delta r)^{-1}
({\Phi(\delta r)^{-1}(\delta r)^{-d}})'$$  is decreasing on $(0,\infty)$.
It is also clear that, by \eqref{e:Phi}, we have
\begin{align*}
\lC\Big(\frac{R}{r}\Big)^{\la}\le \frac{\Phi_{(\delta)}(R)}{\Phi_{(\delta)}(r)}\le \uC\Big(\frac{R}{r}\Big)^{\ua}\quad\hbox{ for
every }~\delta>0, ~ 0<r \le R.
\end{align*}
Finally, if $\delta \le 1$ and $\eta> \ua/2$, then by {\bf (K$_\eta$)}, for every $x,\,h\in \Rd$ with $|h|\leq 1$
 $$|\kappa_{(\delta)}(x,x+h)-\kappa_{(\delta)} (x,x)|=
 |\kappa(\delta x, \delta x+ \delta h)-\kappa(\delta x, \delta x)|
 \leq L |\delta h|^{\eta}   \leq L | h|^{\eta}.
 $$
 Therefore, for $\Phi_{(\delta)}(r)$ and $\kappa_{(\delta)}(x,  y)$ defined above,  \eqref{e:Phi}, {\bf (K$_\eta$)} and {\bf(SD)} hold uniformly for all $ \delta \le 1$.

Since $f_*(x_1+1)^{-1} U=B\big(\frac{\xi_x}{f_*(x_1+1)},\kappa\big)$,
by using
\cite[Lemma 7.2]{GKK}
to
$X^{(\delta_1)}$ with  $ \delta_1 =f_*(x_1+1)$ and \eqref{e:Phi}, we have
\begin{align*}
\Pp^x\left(\tau^0_{U}>\Phi(f_*(x_1+1))\right)&=\Pp^{\delta_1^{-1}x} \left(\tau^{(\delta_1)}_{\delta_1^{-1} U}>1\right)  \ge
c_1\Phi_{(\delta_1)}\big(\delta_{\delta_1^{-1} U}( \delta^{-1}x)\big)\\
&= c_1 \Phi(\delta_1)^{-1/2} \Phi\big(\delta_1 \delta_{\delta_1^{-1} U}( \delta_1^{-1}x)\big)^{1/2}\\
&=\frac{c_1\Phi(\delta_{U}(x))^{1/2}}{\Phi(f_*(x_1+1))^{1/2}}\ge
\frac{c_1\Phi(c_0\delta_{D}(x))^{1/2}}{\Phi(f_*(x_1+1))^{1/2}}\\
&\ge
\frac{c_2\Phi(\delta_{D}(x))^{1/2}}{\Phi(f_*(x_1+1))^{1/2}},
\end{align*}
where $\tau^{(\delta_1)}_{A}$ denotes the first exit time
from $ A $ for the process $X^{(\delta_1)}$.
This together with \eqref{l5-4-1n} yields \eqref{l5-4-1}.
\end{proof}

\begin{proposition}\label{p5-2}
The following two statements hold.
\begin{itemize}
\item[(1)] Suppose that $\gamma=0$ in \eqref{e:Exp}. Then there exists
a constant $c_1>0$ such that for every $x \in D$ with $|x|$ large enough,
\begin{equation}\label{p5-2-1}
\phi_1(x)\ge c_1 \Phi(\delta_{D}(x))^{1/2}\Phi(f_*(x_1+1))^{1/2}|x|^{-d}\Phi(|x|)^{-1}
\end{equation}

\item[(2)] Suppose that $\gamma \in (0,\infty]$ in \eqref{e:Exp}. Then,
for any
increasing function  $g(r)$ satisfying that $4\le r/g(r) \le {r}/{4}$ for $r>0$ large enough,
there exist constants
$c_i>0$ $(i=1,2,3,4)$ such that for every $x \in D$ with $|x|$ large enough,
\begin{equation}\label{p5-2-2}
\phi_1(x)\ge c_1 \Phi(\delta_{D}(x))^{1/2}\Phi(f_*(x_1+1))^{1/2} \exp\big(-c_2|x|^{\gamma}\big),
\end{equation}
and
\begin{equation}\label{p5-2-3}
\begin{split}
\phi_1(x)\ge & c_3\Phi(\delta_{D}(x))^{1/2}\Phi({ f}_*(x_1+1))^{1/2}\\
&\times \exp\Bigg[- c_4 \frac{|x|}{g(|x|)} \bigg(g^\gamma(|x|)+\log \frac{|x|}{ g(|x|)}+
\log \Big(1+\frac{1}{f_*(|x|+1)}\Big)\bigg)\Bigg],\end{split}
\end{equation}
\end{itemize}
\end{proposition}
\begin{proof}
Recall that $\ext_{\Phi,\ge}$ holds with the constant $\tilde r_0=1$ in \eqref{e:exit}  in this setting
(see e.g. \cite[Lemma 2.5]{CKK3}).
Let $x_0=(1,\tilde 0)$, $r_{0}=({\delta_D(x_0)}/{4})\wedge 1$ and
$t_0= c_0 \Phi(1)$, where $c_0$ is a positive constant $C_0$ in \eqref{e:exit}.
We will always consider $x \in D$ such that $|x|$ is large enough and  $\Phi(f_*(x_1+1))\le {t_0}/{4}$.

\smallskip

(1) The proof of \eqref{p5-2-1} is based on that of Proposition \ref{p:lower1} with some modifications.
Set $D_0=B(x_0,2r_0)$.
According to Lemma \ref{L:lower1},
in order to prove the desired lower bound \eqref{p5-2-1} of $\phi_1$, it suffices
to verify the same lower bound (possibly with different constants) for $P_{t_0}^D\I_{D_0}(x)$.

In the following, let $\tilde D_0=B(x_0,r_0)$. For any fixed $x \in D$ with $|x|$ large enough such that $\Phi(f_*(x_1+1))\le
{t_0}/{4}$,
let $D_1=B(z_x,1)\cap D$, where $z_x\in \partial D$ such that $|x-z_x|=\delta_D(x)$.
Then, using the strong Markov property and $\ext_{\Phi,\ge}$, and following the same argument of
\eqref{e3-3}, we can get that
\begin{equation}\label{e:phil1}
\begin{split}
&P^{D}_{t_0}\I_{D_0}(x)\\
&\ge \Pp^x\Big(0<\tau_{D_1}\le \Phi(f_*(x_1+1)),
X_{\tau_{D_1}}\in \tilde D_0; \forall s\in[\tau_{D_1}, t_0], X_s\in
D_0\Big)\\
&\ge \Pp^x\Bigg(\Pp^{X_{\tau_{D_1}}}\Big(\tau_{D_0}>t_0\Big);0<\tau_{D_1}\le \Phi(f_*(x_1+1)),
X_{\tau_{D_1}}\in \tilde D_0\Bigg)\\
&\ge  \Pp^x\Bigg(\Pp^{X_{\tau_{D_1}}}\Big(\tau_{B(X_{\tau_{D_1}},r_{0})}>t_0\Big);0<\tau_{D_1}\le \Phi(f_*(x_1+1)),
X_{\tau_{D_1}}\in \tilde D_0\Bigg)\\
&\ge c_1\Pp^x\Big(0<\tau_{D_1}\le \Phi(f_*(x_1+1)), X_{\tau_{D_1}}\in \tilde D_0\Big).
\end{split}
\end{equation}
According to the L\'evy system \eqref{e:levy}, we have
\begin{equation}\label{e:phil2}
\begin{split}
&\Pp^x\Big(0<\tau_{D_1}\le \Phi(f_*(x_1+1)), X_{\tau_{D_1}}\in \tilde D_0\Big)\\
&=\int_0^{\Phi(f_*(x_1+1))}\,ds
\int_{D_1} p^{D_1}(s,x,y)\,dy\int_{\tilde D_0} J(y,z)\,dz\\
&\ge  \left(\inf_{y\in D_1, z\in \tilde D_0} J(y,z)\right) |\tilde D_0| \int_0^{\Phi(f_*(x_1+1))}\int_{D_1} p^{D_1}(s,x,y)\,dy\,ds\\
&\ge c_2|x|^{-d}\Phi(|x|)^{-1} \Phi(f_*(x_1+1))
\Pp^x\Big( \tau_{D_1}>\Phi(f_*(x_1+1))\Big)\\
&\ge c_3\Phi(\delta_{D}(x))^{1/2}\Phi(f_*(x_1+1))^{1/2} |x|^{-d}\Phi(|x|)^{-1},
\end{split}
\end{equation}
where in the second inequality we have used that for $x \in D$ with $|x|$ large enough
\begin{equation*}
\inf_{y\in D_1, z\in \tilde D_0} J(y,z)\ge \inf_{|y-z|\le  2{|x|}}J(y,z)\ge c_4|x|^{-d}\Phi(|x|)^{-1},
\end{equation*}
and the last inequality follows from \eqref{l5-4-1}.

Combining \eqref{e:phil1} with \eqref{e:phil2}, we prove \eqref{p5-2-1}.

\smallskip

(2) Since the proof of \eqref{p5-2-2} is almost the same as \eqref{p5-2-1}, we omit it here.
The proof of \eqref{p5-2-3} is partly motivated by those of Propositions \ref{p:lower2} and \ref{p4-1}. We consider $x\in D$ with $|x|\ge 2r_*$ for some $r_*>0$ large enough. Set $ n:=\Big \lceil \frac{x_1-r_*}{g(|x|)}\Big \rceil$,
$x^{(i)}=\big(r_*+{i(x_1-r_*)}/{n},\tilde 0\big):=\big(x^{(i)}_1,\tilde 0\big) $ and
$r_i=({\delta_D(x^{(0)})}/{4})\wedge ({\delta_D(x^{(i)})}/{4})\wedge 1$ for any $0\le i \le n-1$.

Define
\begin{equation*}
\begin{split}
& D_n=B(z_x,1)\cap D,\,\, D_i=B(x^{(i)},2r_i),\,\, \tilde D_i=B(x^{(i)},r_i),\quad 0 \le i  \le n-1;\\
& \tilde \tau_{D_n}=\tau_{D_n},\,\, \tilde \tau_{D_i}=\inf\{t>\tilde \tau_{D_{i+1}}: X_t \notin D_i\},\quad 0\le i \le n-1.
\end{split}
\end{equation*}
Noting that for $0\le i\le n-2$,
\begin{equation*}
\frac{g(|x|)}{2}
\le |x^{(i+1)}-x^{(i)}|=\frac{x_1-{r_*}}{n}
\le g(|x|),
\end{equation*}
we can easily check that $0<$dist$(D_i,D_{i+1})\le g(|x|)$ for $0\le i\le n-1$,
which in turn implies that
\begin{equation}\label{p5-2-5}
\inf_{u \in D_i,z\in D_{i+1}}J(u,z)\ge c_1\exp\big(-c_2g^\gamma(|x|)\big),\quad 0\le i \le n-1.
\end{equation}
By using the strong Markov property and following the argument of \eqref{e:piuf}, we have
\begin{equation*}
\begin{split}
P^D_{t_0}\I_{D_0}(x)
& \ge
\Pp^x\Big(0<\tilde \tau_{D_n}\le \Phi(f_*(x_1+1)), 0<\tilde{\tau}_{D_{i}}-\tilde{\tau}_{D_{i+1}}\le \frac{c_0\Phi(r_i)}{2n},\\
&\qquad\quad X_{\tilde{\tau}_{D_i}}\in \tilde D_{i-1}\
{\rm for\ each}\ 1\le i \le n-1; \forall_{s\in[\tilde{\tau}_{D_1},t_0] } X_{s} \in D_0\Big)\\
& = \Pp^x\Big(0<{\tau}_{D_{n}}\le \Phi(f_*(x_1+1)) , X_{{\tau}_{D_{n}}}\in
\tilde D_{n-1}\\
&\qquad\quad\times\Pp^{X_{\tilde{\tau}_{D_{n}}}}\Big(0<\tau_{D_{n-1}}\le \frac{c_0\Phi( r_{n-1})}{2n},X_{{\tau}_{D_{n-1}}}\in
\tilde D_{n-2}\\
&\qquad\quad\,\times \Pp^{X_{\tilde{\tau}_{D_{n-1}}}}\Big(\cdots
\Pp^{X_{\tilde{\tau}_{D_{2}}}}\Big(0<{\tau}_{D_{1}}\le \frac{c_0\Phi(r_1)}{2n},X_{\tau_{D_{1}}}\in
\tilde D_0\\
&\qquad\quad\,\,\times
\Pp^{X_{\tilde{\tau}_{D_{1}}}}\Big(\forall_{s\in[0,t_0-\tilde {\tau}_{D_1}]} X_{s}
\in D_0\Big)\Big)\cdots\Big)\Big)\Big).
\end{split}
\end{equation*}

For any $r>0$, set $\tilde f(r)=\inf_{|s-r|\le f(r)} f(s)$. It is clear that for $r>0$ large enough, $\tilde f(r) \ge f_*(r)$.
For every $2 \le i \le n-1$, if $X_{\tilde{\tau}_{D_{i}}}\in \tilde D_{i-1}$,
then,  by the L\'evy system \eqref{e:levy},  we have
\begin{align*}
&
\Pp^{X_{\tilde{\tau}_{D_{i}}}}\Big(0<{\tau}_{D_{i-1}}\le \frac{c_0\Phi(r_{i-1})}{2n},X_{{\tau}_{D_{i-1}}}\in
\tilde D_{i-2}\Big)\\
&\ge \inf_{y\in \tilde D_{i-1}} \Pp^y\Big(0<{\tau}_{D_{i-1}}\le \frac{c_0 \Phi(r_{i-1})}{2n},X_{{\tau}_{D_{i-1}}}\in
\tilde D_{i-2}\Big)\\
&= \inf_{y\in \tilde D_{i-1}} \int_{D_{i-1}}\,dz\int_0^{\frac{c_0 \Phi(r_{i-1})}{2n}}
p^{D_{i-1}}(s,y,z)\,ds\int_{\tilde D_{i-2}}J(z,u)\,du\\
&\ge \left( \inf_{z\in D_{i-1}, u\in \tilde D_{i-2}} J(z,u)\right) |\tilde D_{i-2}|
\inf_{y\in \tilde D_{i-1}} \int_0^{\frac{c_0\Phi(r_{i-1})}{2n}}\Pp^y(\tau_{D_{i-1}}>s)\,ds\\
&\ge  c_3r_{i-2}^{d} \frac{\Phi(r_{i-1})}{2n} \exp\big(-c_2g^\gamma(|x|)\big)
\inf_{y\in \tilde D_{i-1}}\Pp^y\Big(\tau_{D_{i-1}}> \frac{c_0\Phi(r_{i-1})}{2n}\Big)\\
&\ge c_4  \big({\tilde{f}(x^{(i-2)}_1)\wedge r_0}\big)^{d}
\Phi\big({\tilde{f}(x_1^{(i-1)})\wedge r_0}\big)
\frac{g(|x|)}{|x|} \\
&\quad \times \exp\big(-c_2 g^\gamma(|x|)\big) \inf_{y\in \tilde D_{i-1}}\Pp^y\Big(\tau_{B(y, r_{i-1})}> \frac{c_0\Phi(r_{i-1})}{2}\Big)\\
&\ge c_5  \big({ \tilde{f}(x_1^{(i-2)})\wedge r_0}\big)^{d}
\Phi\big({\tilde{f}(x_1^{(i-1)})\wedge r_0}\big)
\exp\bigg(-c_6 \Big(g^\gamma(|x|)+\log \frac{|x|}{ g(|x|)}\Big)\bigg),
 \end{align*}
where the third inequality is due to \eqref{p5-2-5}, in the forth inequality we have used the facts
that $n\le c'{|x|}/{g(|x|)}$ and, by the definition of horn-shaped region, see e.g.\ \eqref{e:point},
$$r_{i-2}=\delta_D(x^{(i-2)})\wedge r_0\ge {c_7}\big(\tilde{f}(x^{(i-2)}_1)\wedge r_0\big),$$
and in the last one we used $\ext_{\Phi,\ge}$.

Furthermore, according to the L\'evy system \eqref{e:levy}, \eqref{p5-2-5} and \eqref{l5-4-1},  we know immediately that
\begin{equation*}
\begin{split}
&\Pp^x\Big(0<\tilde \tau_{D_n}\le \Phi(f_*(x_1+1)), X_{\tilde \tau_{D_n}} \in \tilde D_{n-1}\Big)\\
&=\int_{D_n}\,dz\int_0^{\Phi(f_*(x_1+1))}p^{D_n}(s,x,z)\,ds\int_{\tilde D_{n-1}}J(z,u)\, du\\
&\ge c_8r_{n-1}^d \Phi(f_*(x_1+1)) \exp\big(-c_2 g^\gamma(|x|)\big) \Pp^x\big(\tau_{B(z_x,1)\cap D}>\Phi(f_*(x_1+1))\big)\\
&\ge c_9\big(\tilde{f}(x^{(n-1)}_1)\wedge r_0\big)^{d}\Phi(\delta_{D}(x))^{1/2}\Phi(f_*(x_1+1))^{1/2}\exp(-c_2 g^\gamma(|x|)).
\end{split}
\end{equation*}

On the other hand, using $\ext_{\Phi,\ge}$ again, if
$X_{\tilde{\tau}_{D_1}}\in \tilde D_0$, then
\begin{equation*}
\begin{split}
\Pp^{X_{\tilde{\tau}_{D_{1}}}}\Big(\forall_{s\in  [0,t_0-\tilde {\tau}_{D_1}]} X_{s}
\in D_0\Big)&\ge\inf_{y \in \tilde
D_0}\Pp^y\Big (\tau_{D_0}>t_0\Big)\\
&\ge \inf_{y \in \tilde
D_0}\Pp^y\Big
(\tau_{B(y,r_0)}>t_0\Big)\ge c_{10}.
\end{split}
\end{equation*}

Combining all the estimates above, we arrive at
\begin{equation*}
\begin{split}
&P^D_{t_0}\I_{D_0}(x)\\
& \ge c_{11}\Phi(\delta_{D}(x))^{1/2}\Phi(f_*(x_1+1))^{1/2}\\
&\quad \times
 \exp\bigg(-c_{12}n \Big(g^\gamma(|x|)+\log \frac{|x|}{ g(|x|)}\Big)\bigg)
 \prod_{i=1}^{n-1}(\tilde{f}(x^{(i-1)}_1\wedge r_0))^{d}
 \Phi(\tilde{f}(x_1^{(i)}\wedge r_0))
 \\
&\ge c_{13} \Phi(\delta_{D}(x))^{1/2}\Phi(f_*(x_1+1))^{1/2}\\
&\quad \times \exp\bigg[- c_{14}\frac{|x|}{g(|x|)} \Big(g^\gamma(|x|)+\log \frac{|x|}{ g(|x|)}+
\log\frac{1} {\inf_{r_*\le s \le |x|+1}f(s)}\Big)\bigg],
\end{split}
\end{equation*}
where the last inequality follows from \eqref{e:Phi} and the fact that
\begin{equation*}
\begin{split}
\sum_{i=1}^{n-1} \log\frac{1}{\tilde{f}(x_1^{(i-1)})}&\le
c_{16}\sum_{i=1}^n \log\bigg(e+\frac{1}{\inf_{r_*\le s \le |x|+1}f(s)}\bigg)\\
&\le c_{17}\frac{|x|}{g(|x|)}\log\Big(1+\frac{1}{\inf_{r_*\le s \le |x|+1}f(s)}\Big).
\end{split}
\end{equation*}
Hence we have proved \eqref{p5-2-3}.
\end{proof}

\section{Appendix: Boundedness, continuity and strict positivity of Dirichlet heat kernel}\label{section4-1}
In this appendix, we make some comments on assumptions of Dirichlet heat kernel in Subsection \ref{subsection1-1}. Let $(\mathscr{E}, \mathscr{F})$ be the Dirichlet form given by  \eqref{non-local} such that its jumping kernel $J(x,y)$ satisfies \eqref{e:dfc}.
 Since $(\mathscr{E}, \mathscr{F})$ is regular on $L^2(\R^d;dx)$, it associates a symmetric Hunt process $X=(X_t)_{t\ge0}$ starting from quasi-everywhere on $\R^d$.
 Suppose that there exist $0<\alpha_1\le\alpha_2<2$
and $c_1,c_2\in(0,\infty)$ such that
\begin{equation}\label{e3-1}
\frac{c_1} {|x-y|^{d+\alpha_1}}\le J(x,y)\le \frac{c_2}{|x-y|^{d+\alpha_2}}, \quad
0<|x-y|\le 1. \end{equation}
Then, according to  \cite[Proposition 3.1]{CKK1},
there exist a properly $\mathscr{E}$-exceptional set $\Nn\subset \R^d$ and  a transition
density (also called heat kernel)  $p(\cdot,\cdot,\cdot):(0,\infty)\times (\R^d
\setminus \Nn) \times (\R^d \setminus \Nn) \rightarrow [0,\infty)$
such that
 \begin{equation}\label{e:pupper}
p(t,x,y)\, \le \, c
\left(t^{-d /\alpha_1} \vee t^{-d/2}\right), \quad x,y\in\R^d\setminus \Nn, t>0.
\end{equation}

In the following, we further suppose that \begin{equation}\label{e3-2_n}
C_*:=\sup_{|x-y|>1} J(x,y)<\infty,
\end{equation} and that
\begin{equation}\label{e3-2_nn}
\Nn=\emptyset \text{ and } p(\cdot,\cdot,\cdot):(0,\infty)\times \R^d\times \R^d\rightarrow (0,\infty) \text{ is continuous and strictly positive}.
\end{equation}
See \cite{BKK,CK, CK1, CKK2,CKK1} and the references therein for
sufficient conditions on such assumptions.
In particular, the associated semigroup $(P_t)_{t\ge0}$ enjoys the strong Feller property.

\begin{proposition}\label{p:pDcont} Under assumptions above, for every open set $D$ and $t>0$, the
transition density for the process $X^D$, $p^D(t,\cdot,\cdot): D \times D \rightarrow (0,\infty)$,  is
continuous.
\end{proposition}

\begin{proof}
We first construct the process $Z=(Z_t)_{t\ge0}$ from $X$ by
removing jumps of size larger than $1$ via Meyer's construction (see
\cite[Remark 3.4]{BBCK}).
Let $p_Z(t,x,y)$ be the transition density of $Z$.
Note that, from \cite[(3.30)]{BBCK} one can check that there exist constants $c_i>0$ ($1\le i\le 4$) such that for all $|x-y| \le c_1$
\begin{align}
\label{epZ}
p_Z(t,x,y)  \le c_2 t |x-y|^{-c_3},
 \quad t <  c_4 |x-y|^{\alpha_2}.
\end{align}
By
\cite[Lemma 3.7(2)]{BBCK}, we have
\begin{align}
\label{epZ1}p(t,x,y)\, \le \, p_Z(t,x,y) +
 t \|J(x,y)\I_{\{|x-y|> 1\}}\|_{\infty}, \quad x,y\in\R^d, t>0,
\end{align}
We will claim that for all $ \beta>0$,
\begin{align}
\label{e:r: bound}
\sup_{|x-y| \geq \beta, 1 > t>0} p(t,x,y) +\sup_{|x-y| \geq \beta, t \ge 1} p(t,x,y) =:M_{1, \beta}+M_{2,\beta} < \infty.
 \end{align}
In fact, $M_{2,\beta} < \infty$ for all $\beta>0$ by \eqref{e:pupper}. On the other hand, by \eqref{e:pupper}, \eqref{e3-2_n},  \eqref{epZ} and \eqref{epZ1}, for  all $ \beta\le c_1$,
\begin{align*}
 M_{1,\beta}
&\le \sup_{|x-y| \geq \beta, t <1 \wedge (c_4 |x-y|^{\alpha_2}) } p(t,x,y)  +\sup_{|x-y| \geq \beta, 1>t\ge c_4 |x-y|^{\alpha_2}} p(t,x,y)  \nonumber\\
&\le
c_5 \sup_{|x-y| \geq \beta, t <1 \wedge (c_4 |x-y|^{\alpha_2}) } (t |x-y|^{-c_3}+ C_*t) +c_5 \sup_{ 1>t\ge c_4 \beta^{\alpha_2}}
 t^{-d /\alpha_1} \nonumber\\
&\le  c_6 (\beta^{-c_3} +C_* + \beta^{-d \alpha2 /\alpha_1}) < \infty.
\end{align*}
On the other hand, by Meyer's construction and \cite[Lemma 3.8, (3.29) and (3.20)]{BBCK}, one can obtain that $\Pp^x(\tau_D\le t)\to 0$  uniformly in any compact $K\subset D$ as $t\to0$.
Using the continuity of $p(\cdot, \cdot, \cdot)$, the strong Feller property of
$(P_t)_{t\ge 0}$ and \eqref{e:r: bound},
one can follow the proof of \cite[Theorem 2.4]{CZ} line by line and
show that, for each $t>0$,
$(x, y) \mapsto \Ee^x\big[p(t-\tau_D,X_{\tau_D},y)\I_{\{t\ge \tau_D\}}\big]$
is continuous on $D \times D$.
Thus, by \eqref{e:dden}, for every $t>0$ the
function $p^D(t,\cdot,\cdot): D \times D \rightarrow (0,\infty)$ is
continuous.
\end{proof}

Under assumptions above, if $D$ is a domain of $\R^d$ (i.e., $D$ is a connected open set), then
it is easy to verify that Dirichlet heat kernel $p^D(t,x,y)$ is
strictly positive for any $(t,x,y)\in (0,\infty)\times D\times D$.
See \cite[Corollary 7 and Remark 8 (2)]{CWi14} or \cite[Proposition 2.2 (i)]{G}.
For disconnected open set, we need some additional assumption. Following \cite[Condition (RC), p.\ 1120]{CWi14} (or see \cite[Definition 4.3]{KS1}), we call
an open set $D$ is {roughly connected} by the process $X$,
if for any $x,y\in D$, there exist $m\ge 1$ and distinct connected components $\{D_i\}_{i=1}^m$ of $D$, such that
$x\in D_1$, $y\in D_m$ and for every
$1 \le i \le m-1$,
dist$(D_i, D_{i+1})< r_J$,
where
$$r_J:=\inf\Big\{r>0: \inf_{|x-y|\le r}J(x,y)>0\Big\}.$$

The following proposition essentially has been proved in \cite[Proposition 6 and Remark 8]{CWi14}.

\begin{proposition}\label{p:pDp} Under assumptions above, if  the open set $D\subset \R^d$ is {roughly connected} by the process $X$, then the
Dirichlet heat kernel $p^D(t,x,y)$ is
strictly positive for any $(t,x,y)\in (0,\infty)\times D\times D$.
\end{proposition}

In particular, if there are some constants $0<a_1\le a_2\le 1$ such that for all $x,y\in D$,  $x\sim_{(n; a_1, a_2)} y$
(that is, $x$ is connected with $y$ in a reasonable way with respect to constants $0<a_1\le a_2<1$, see Definition \ref{d:c-reason}),
then the open set $D$ is {roughly connected} by the process $X$ under assumption \eqref{e:Jlow}.

\ \

The result below should be known, see e.g.\ \cite[Chapter V, Theorem 6.6]{Sch} or \cite[Theorem XIII. 43]{RS2}. The reader can see \cite[Proposition 1.2]{CW16} for the proof.

\begin{proposition}\label{p5-1-0}
Suppose that $(P_t^D)_{t \ge 0}$ is compact, and let $\phi_1$ be its first eigenfunction $($called ground state$)$. If $p^D(t,x,y)$ is bounded, continuous and strictly positive on $(0,\infty)\times D\times D$, then
$\phi_1$ also has a version which is bounded, continuous and strictly positive on $D$.
\end{proposition}

\bigskip

\noindent \textbf{Acknowledgements.} The
research of Xin Chen is supported by National Natural Science Foundation of China (No.\ 11501361), \lq\lq Yang Fan Project\rq\rq \, of Science and Technology Commission of Shanghai Municipality (No.\ 15YF1405900), and Fujian Provincial
Key Laboratory of Mathematical Analysis and its Applications
(FJKLMAA). The research of Panki Kim is supported by the National Research Foundation of Korea (NRF) grant funded by the Korea government (MSIP) (No.\ 2016R1E1A1A01941893).
The research of Jian Wang is supported by National
Natural Science Foundation of China (No.\ 11522106), the Fok Ying Tung
Education Foundation (No.\ 151002), the JSPS postdoctoral fellowship
(26$\cdot$04021), National Science Foundation of
Fujian Province (No.\ 2015J01003), the Program for Nonlinear
Analysis and Its Applications (No. IRTL1206), and Fujian Provincial
Key Laboratory of Mathematical Analysis and its Applications
(FJKLMAA).

\end{document}